\documentclass[final]{siamltex}
 \title{Stratified Steady Periodic Water Waves \thanks{This work was supported in part by NSF grant DMS-0405066}}
 \author{Samuel Walsh \thanks{Division of Applied Mathematics, Brown University, 182 George Street, Providence, Rhode Island 02912 ({\tt samuel\_walsh@brown.edu})}}
 \usepackage{amsfonts}
 \usepackage{amsmath}
 \usepackage{amssymb}
 \usepackage{mathrsfs}
\usepackage{txfonts}
 
\numberwithin{equation}{section}

\newcommand{\be}{\begin{equation} }
\newcommand{\ee}{\end{equation}}

\newcommand{\puffthm}[1]{\emph{Proof of Theorem #1. }}
\newcommand{\remark}{\emph{Remark. }}

\begin{document}
\maketitle

\begin{abstract}
This paper considers two-dimensional stratified water waves propagating under the force of gravity over an impermeable flat bed and with a free surface.  We prove the existence of a global continuum of classical solutions that are periodic and traveling.  These waves, moreover, can exhibit large density variation, speed and amplitude. 
\end{abstract}



\pagestyle{myheadings}
\thispagestyle{plain}
\markboth{SAMUEL WALSH}{STRATIFIED WATER WAVES}

\section{Introduction}

Stratification is a common feature of ocean waves, where the presence of salinity, in concert with  external gravitational force, can produce substantial heterogeneity in the fluid.  The pronounced effects that may accompany even a moderate density variation have earned stratified flows a great deal of scholarly attention, particularly in the realm of geophysical fluid dynamics.  In this paper we develop a theory for stratified gravity waves that are traveling and periodic.

If we imagine a wave on the open ocean, past experience suggests that it may be \emph{regular} in the following sense.  First, it is essentially two-dimensional.  That is, the motion will be identical along any line that runs parallel to a crest.  Second, if we regard the wave in a coordinate system moving with some constant speed, it appears steady.  Finally, the profile is periodic in the direction of motion, descending monotonically from a single crest to a single trough once per period.  

We now formulate governing equations for waves of this form.  Fix a Cartesian coordinate system so that the $x$-axis points in the direction of propagation, and the $y$-axis is vertical.  We assume that the floor of the ocean is flat and occurs at $y = -d$.  Let $y = \eta(x,t)$ be the free surface at the interface between the atmosphere and the fluid.  We shall normalize $\eta$ by choosing the axes so that the free surface is oscillating around the line $y = 0$.  As usual we let $u = u(x,y,t)$ and $v = v(x,y,t)$ denote the horizontal and vertical velocities, respectively.  Let $\rho = \rho(x,y,t) > 0$ be the density.  

Incompressibility of the fluid is represented mathematically by the requirement that the vector field $(u,v)$ be divergence free for all time
\be u_x + v_y = 0. \label{incompress} \ee
Taking the fluid to be inviscid, conservation of mass implies that the density of a fluid particle remains constant as it follows the flow.  This is expressed by the continuity equation
\be \rho_t + u \rho_x + v \rho_y = 0. \label{consmass1} \ee
Next, the conservation of momentum is described by Euler's equations
\be \left \{ \begin{array}{lll}
u_t + u u_x + v u_y  & = & -\frac{P_x}{\rho} \\
v_t + u v_x + v v_y & = & -\frac{P_y}{\rho} - g, \end{array} \right. \label{euler1} \ee
where $P=P(x,y,t)$ denotes the pressure and $g$ is the gravitational constant.  Here, of course, we assume that the only external force acting on the fluid is gravity.  

On the free surface, we must ensure that the pressure of the fluid matches the atmospheric pressure of the air above, that we shall denote $P_{\textrm{atm}}$.  Thus,
\be P = P_{\textrm{atm}}, \qquad \textrm{on } y = \eta(x,t). \label{presscond} \ee
The corresponding boundary condition for the vector field is motivated by the fact that fluid particles that reside on the free surface continue to do so as the flow develops.  This observation is manifested in the kinematic condition
\be v = \eta_t + u \eta_x, \qquad \textrm{on } y = \eta(x,t). \label{kincond1} \ee
Since we cannot have any fluid moving normal to the flat bed occurring at $y = -d$, we require
\be v = 0, \qquad \textrm{on } y = -d. \label{bedcond} \ee
Note that there is no accompanying condition on $u$ because in the inviscid case we allow for slip, that is, nonzero horizontal velocity along solid boundaries.  
  
We seek traveling periodic wave solutions $(u,v,\rho,P, \eta)$ to \eqref{incompress}-\eqref{bedcond}.  More precisely, we take this to mean that, for fixed $c > 0$, the solution appears steady in time and periodic in the $x$-direction when observed in a frame that moves with constant speed $c$ to the right.  The vector field will thus take the form $u = u(x-ct,y)$,  $v = v(x-ct, y)$, where each of these is $L$-periodic in the first coordinate.  Likewise for the scalar quantities: $\rho(x,y,t) = \rho(x-ct,y)$, $P = P(x-ct,y)$, and $\eta = \eta(x-ct)$, again with $L$-periodicity in the first coordinate.  We therefore take moving coordinates 
\[ (x-ct, y) \mapsto (x,y), \]
which eliminates time dependency from the problem.  In the moving frame \eqref{incompress}-\eqref{euler1} become 
\be \left \{ \begin{array}{lll} 
u_x + v_y & = & 0 \\
(u-c)\rho_x + v\rho_y & = & 0 \\
(u-c)u_x + v u_y  & = & -\frac{P_x}{\rho} \\
(u-c) v_x + v v_y & = & -\frac{P_y}{\rho} - g \end{array} \right. \label{euler2} \ee
throughout the fluid domain.  Meanwhile, the reformulated boundary conditions are
\be \left \{ \begin{array}{llll} 
v & = & (u-c) \eta_x & \textrm{on } y = \eta(x) \\
v & = & 0 & \textrm{on } y = -d \\
P & = & P_{\textrm{atm}} & \textrm{on } y = \eta(x) \end{array} \right. \label{boundcond} \ee
where $u,v,\rho, P$ are taken to be functions of $x$ and $y$, $\eta$ is a function of $x$, and all of them are $L$-periodic in $x$.

In the event that $u = c$ somewhere in the fluid we say that \emph{stagnation} has occurred, as in the moving frame the fluid appears to be stationary at that point.  In order to avoid roll-up and other instability phenomena, we shall restrict our attention to the case where $u < c$ throughout.  

Recall that we have chosen our axes so that $\eta$ oscillates around the line $y = 0$.  Mathematically this equates to
\be \fint_{0}^{L} \eta(x)dx = 0. \label{normalsurface} \ee 
In effect, this couples the depth to the problem, so we need not treat $d$ as a free parameter.  The trade-off, as we shall see, is that this normalization will have significant technical consequences later.  

An often indispensable tool in the study of incompressible fluids is the stream function --- especially in the case of two-dimensional flow in bounded domains.  This is a function whose curl is the vector field $(u,v)$ and thus, the gradient is orthogonal to the field at each point in the fluid.  For a homogeneous fluid, a standard method is to then reformulate the problem for the stream function instead of the flow.  What recommends this approach is the fact that the boundaries of the domain will be level sets of the stream function.  Therefore, by adopting streamline coordinates, we can fix the domain and eliminate the free surface problem.

However, for a heterogeneous fluid this generally proves insufficient, as it does not sufficiently capture the effects of stratification.  Instead, we employ a more generalized object whose use was pioneered by Yih and Long, among others (cf. \cite{Y1}, \cite{Lo} for example).  Observe that by conservation of mass and incompressibility, $\rho$ is transported and the vector field is divergence free.  Therefore we may introduce a (relative) \emph{pseudo-stream function} $\psi = \psi(x,y)$, defined uniquely up to a constant by: 
\[ \psi_x = -\sqrt{\rho}v,\qquad \psi_y = \sqrt{\rho} (u-c). \]
Here we have the addition of a $\rho$ term to the typical definition of the stream function for an incompressible fluid.  This neatly captures the inertial effects of the heterogeneity of the flow (see the treatment in \cite{Y1}, for example).  The particular choice of $\sqrt{\rho}$ is merely to simplify algebraically what follows.

It is a straightforward calculation to check that $\psi$ is indeed a (relative) stream function in the usual sense, i.e. its gradient is orthogonal to the vector field in the moving frame at each point in the fluid domain:
\[ (u-c) \psi_x + v \psi_y = 0. \]
Moreover, \eqref{boundcond} implies that the free surface and flat bed are each level sets of $\psi$.  For definiteness we choose $\psi \equiv 0$ on the free boundary, so that $\psi \equiv p_0$ on $y = -d$, where $p_0$ is the quantity 
\be p_0 := \int_{-d}^{\eta(x)} \sqrt{\rho(x,y)}\left[u(x,y) - c\right] dy. \label{defp0} \ee
To see that this is well-defined, that is, the integral on the right-hand side is independent of $x$, we calculate
\begin{eqnarray*}
\frac{d p_0}{d x} & = & \eta_x\left( \sqrt{\rho(x,\eta(x))}\left[u(x,\eta(x))-c)\right] \right)+ \int_{-d}^{\eta(x)} \partial_x \left(\sqrt{\rho(x,y)}\left[u(x,y)-c\right]\right)dy \\ 
& = & \sqrt{\rho(x,\eta(x))}v(x,\eta(x)) - \int_{-d}^{\eta(x)} \partial_y \left(\sqrt{\rho}v\right)dy = 0.\end{eqnarray*}   
We shall call $p_0$ the (relative) \emph{pseudo-volumetric mass flux}; it represents the amount of fluid flowing through a vertical line extending from the bed to the free surface and with respect to the transformed vector field $(\sqrt{\rho}(u-c), \sqrt{\rho}v)$.  The level sets of $\psi$ will be called the \emph{streamlines} of the flow.

Since $\rho$ is transported, it must be constant on the streamlines and hence, we may think of it as a function of $\psi$.  Abusing notation we may let $\rho:[p_0,0] \to \mathbb{R}^{+}$ be given such that 
\be \rho(x,y) = \rho(-\psi(x,y)) \label{defrho} \ee
throughout the fluid. The reason for taking $-\psi$ here will become apparent in the next section.  When there is risk of confusion, we shall refer to the $\rho$ occurring on the right-hand side above as the \emph{streamline density function}.  We shall focus our attention on the case where the density is nondecreasing as depth increases.  This is entirely reasonable from a physical standpoint.  Indeed, hydrodynamic stability requires that the depth be monotonically increasing with depth, making this assumption standard in the literature.  The level set  $-\psi = p_0$ corresponds to the flat bed, and the set where $-\psi = 0$ corresponds to the free surface.  Therefore, we require that the streamline density function is nonincreasing as a function of $-\psi$.

By Bernoulli's theorem the quantity
\[ E := P + \frac{\rho}{2}\left( (u-c)^2 + v^2\right) + g\rho y, \]
is constant along streamlines, that is 
\[ (u-c) E_x + v E_y = 0.\]
This can be verified directly from \eqref{euler2}. In the case of an inviscid fluid, $E$ represents the energy of the fluid particle at $(x,y)$.  The first term on the right-hand side, $P$, gives the energy due to internal pressure.  The second and third terms combined describe the kinetic energy, while the last is gravitational potential energy.  When evaluated on the free surface, $E$ is usually referred to as the \emph{hydraulic} or \emph{total head} of the fluid.

Under the assumption that $u < c$ throughout the fluid, and given the fact that $E$ is constant along streamlines, there exists a function $\beta:[0,|p_0|] \to \mathbb{R}$ such that 
\be \frac{dE}{d\psi}(x,y) = \beta(\psi(x,y)). \label{defbeta} \ee
For want of a better name we shall refer to $\beta$ as the \emph{Bernoulli function} corresponding to the flow.  Physically it describes the variation of specific energy as a function of the streamlines.   Define
\[ B(p) := \int_0^p \beta(-s)ds\]
for $p_0 \leq p \leq 0$ and let $B$ have minimum value $B_{\textrm{min}}$.  

Let us briefly outline a few notational conventions.  Let $\overline{D_\eta}$ denote the closure of the fluid domain
\[ D_\eta := \{ (x,y) \in \mathbb{R}^2 : -d < y < \eta(x) \}.\]
For any integer $m \geq 1$ and $\alpha \in (0,1)$, we say that a bounded domain $D \subset \mathbb{R}^2$ is $C^{m+\alpha}$ provided that each point in the boundary, denoted $\partial D$, is locally the graph of a $C^{m+\alpha}$ function.  Furthermore, for fixed $m \geq 1$ and $\alpha \in (0,1)$ we define the space $C_{\textrm{per}}^{m+\alpha}(\overline{D})$ to consist of those functions $f:\overline{D}\to\mathbb{R}$ with H\"older continuous derivatives (of exponent $\alpha$) up to order $m$ and that are $L$-periodic in the $x$-variable.  Similarly, we shall take $C_{\textrm{per}}^{m}(\overline{D})$ to be the space of $m$-times continuously differentiable functions which are $L$-periodic in $x$.

Our main result is the following  \\

\begin{theorem}
Fix a wave speed $c > 0$, wavelength $L > 0$, and a relative mass flux $p_0 < 0$.  Fix any $\alpha \in (0,1)$, and let the functions $\beta \in C^{1+\alpha}([0,|p_0|])$ and $\rho \in C^{1+\alpha}([p_0,0])$ be given such that the \emph{(L-B)} condition holds (see \textsc{Definition \ref{defLBcond}} for a precise statement).   Also we assume the streamline density function $\rho$ is nonincreasing.

Consider traveling solutions to the stratified water wave problem \eqref{incompress}-\eqref{bedcond} of speed $c$, relative pseudo-mass flux $p_0$, Bernoulli function $\beta$ and streamline density function $\rho$ such that $u < c$ throughout the fluid.  There exists a connected set $\mathcal{C}$ of solutions $(u,v,\rho, \eta)$ in the space $C_{\textrm{per}}^{2+\alpha}(\overline{D_\eta}) \times C_{\textrm{per}}^{2+\alpha}(\overline{D_\eta}) \times C_{\textrm{per}}^{2+\alpha}(\overline{D_\eta}) \times C_{\textrm{per}}^{3+\alpha}(\mathbb{R})$ with the following properties.
\begin{romannum}
\item $\mathcal{C}$ contains a laminar flow (with a flat surface $\eta \equiv 0$ and all streamlines parallel to the bed)
\item Along some sequence $(u_n, v_n, \rho_n, \eta_n) \in \mathcal{C}$, either $\max_{\overline{D_{\eta_n}}} u_n \uparrow c$, $\min_{\overline{D_{\eta_n}}} u_n \downarrow -\infty$, or $\mathcal{C}$ contains more than one distinct laminar solution.
\end{romannum}
Furthermore, each non-laminar flow $(u,v,\rho,\eta) \in \mathcal{C}$ is regular in the sense that:
\begin{romannum}
\item $u,~v,~\rho$ and $\eta$ each have period $L$ in $x$;
\item within each period the wave profile $\eta$ has a single crest and trough; say the crest occurs at $x =0$; 
\item $u$, $\rho$ and $\eta$ are symmetric, $v$ antisymmetric across the line $x =0$;
\item a water particle located at $(x,y)$ with $0 < x < \frac{L}{2}$ and $y > -d$ has positive vertical velocity $v > 0$; and 
\item $\eta^\prime(x) < 0$ on $(0,\frac{L}{2})$.
\end{romannum}
\label{mainresult}
\end{theorem}

\noindent \\ A few remarks.  The (L-B) condition above arises in the local theory as a means of guaranteeing the existence of a minimizer to a particular Sturm-Liouville problem.  In \textsc{Lemma \ref{sizecondsufficiencylemma}} we will show that a sufficient condition for (L-B) to hold is the following:
\begin{eqnarray}
 g\rho(0)p_0^2  & > & \int_{p_0}^0\bigg\{ \frac{4\pi^2}{L^2}(2B(p) -2B_{\min{}} + 2\epsilon_0)^{3/2}+  \nonumber \\
  & & (p-p_0)^2 ((2B(p)-2B_{\min{}}+2\epsilon_0)^{1/2}+g\rho^\prime(p))\bigg\}dp  \label{sizecond} \end{eqnarray}
where 
\be \epsilon_0 := \max\left\{\left(2g \|\rho^\prime \|_\infty p_0^2 e^{|p_0|}\right)^{2/3},~\left(2g\|\rho^\prime \|_\infty\right)^2,~\left(4 \|\rho^\prime\|_\infty\right)^2,~\left(8g|p_0|\rho(0)\right)^{2/3}  \right\}.\label{defepsilon0} \ee
 Though this is far from necessary, it has the advantage of being entirely explicit.  
 
 In order to better situate this result in the larger context of geophysical fluid dynamics, let us give a quick comment on length-scales.  As is evident from \textsc{Theorem \ref{mainresult}}, our method allows us to treat mathematically waves with arbitrary speed and period.  From a physical standpoint, however, we caution that our model ceases to be valid in certain scales.  First, to ensure that the flow is incompressible, the Mach number must be far less than one.  Also, in order to justify working within the inviscid regime, we must assume that the speed of the wave is substantially faster than the time scale for diffusion in the fluid (a situation typical of salinity in ocean wave, cf. \cite{AT} for a more careful derivation).  Finally, to safely neglect Coriolis effects in \eqref{euler1}, we need the rate of ambient rotation, $\Omega$, to be of larger order than the ratio $c/L$.

Finally, we note that stagnation will occur along some sequence in $\mathcal{C}$ unless serious pathologies occur.  We shall prove in \S 7 that, if there is some $\{u_n, v_n, \eta_n\}$ in $\mathcal{C}$ with
\[ \liminf_{n\to \infty} \min_{\overline{D_{\eta_n}}} u_n \downarrow -\infty, \qquad \liminf_{n \to \infty} \min_{\overline{D_{\eta_n}}} (c-u_n) > 0\]
then $\eta_n \to 0$ uniformly.  That is, the fluid domain is pinching off to nothing, while a current of arbitrarily fast leftward-moving particles develops.

Now let us discuss the history of this problem.  Physically, one of the most distinctive features of stratified flows is their ability to exhibit so-called internal waves.  These are flows in which the motion is essentially driven by a density gradient in the fluid rather than gravitational force.  Qualitatively, this can result in some highly counterintuitive behavior. For example, in the ocean one sometimes encounters ``dead water'', a phenomenon where a stratified wave is moving quite rapidly near the floor, but relatively slowly near the surface.  A ship encountering such a wave would observe quiescent water, but  experience a large amount of drag (cf. e.g. \cite{Cu}, for a lengthy discussion).  Given such phenomena, previous mathematical investigations of heterogeneous fluids have largely focused on the study of these internal waves, where the effects of density variation are most pronounced.  The typical setup is to consider a two-dimension stratified fluid (or several layered homogeneous fluids) confined between two impermeable horizontal boundaries.  Some substantial efforts in this genre include \cite{A1, AT, K1, T1}, among many others.

The free surface problem has mainly been studied in the context of solitary waves, that is, waves that limit to a constant height at $\pm \infty$.  A scalar governing equation for the pseudo-stream function was developed by Long \cite{Lo} in 1953, and later improved by Yih (cf. \cite{Y1} and the references therein) that greatly aided these efforts.  The standard approach, for solitary waves, became to take the Long--Yih equation; assuming a given upstream density profile, one then work downstream to deduce properties of the wave.  This is done, for example, in \cite{BBT, PS, LF} and \cite{T2}.  

On the other hand, traveling, periodic stratified waves have not received nearly as much treatment in the literature.  The first substantial results are due to Dubreil--Jacotin in 1937 \cite{DJ2}, generalizing her work on the homogeneous case in \cite{DJ1}.  Remarkably, she had, even at that early date, already derived Long's equation (in fact, Long's equation is referred to as the Dubreil--Jacotin--Long equation in some circles).  Dubreil--Jacotin was able to analyze the fixed boundary problem, i.e. existence of traveling periodic stratified waves between two horizontal plates. Linearizing around these solutions, she gave a system of integral equations governing small amplitude, traveling stratified gravity waves, from which she was able to produce a class of solutions of that type under certain assumptions.  Yanowitch \cite{Ya} later obtained similar results by using a variational argument with an alternate governing equation due to Love, again assuming small amplitude.  In 1953, Ter--Krikorov \cite{TK}, by means of an asymptotic argument and Yih's equation, proved the existence of long, stationary stratified waves.  This was in answer to a large body of results on the long wave problem that assumed the flow was homogeneous and potential.

In this paper we take a new approach rooted in the work of Constantin and Strauss on the homogeneous case (cf. \cite{CS1, CS2} and others).  In 2004, these authors were able to prove the existence of a large class of rotational (homogeneous) gravity waves that were regular in the sense we discussed earlier.  This was done under very weak assumptions relating the volumetric mass flux to the strength of the vorticity \cite{CS1}.  By generalizing the methods of Constantin and Strauss to the heterogeneous case,  we shall inherit many of the benefits of that paper.  For instance, there will be no mathematical restrictions placed upon the wave speed or wavelength.  Moreover, as the hypotheses of \textsc{Theorem \ref{mainresult}} make clear, we need only some inequality between the density variation, specific energy and volumetric mass flux in order to conclude existence of a global continuum of solutions.

 To emphasize the effects of heterogeneity, we hew closely to the organizational structure of \cite{CS1}.  We begin, in \S 2, by using $\psi$ to change variables and thereby fix the domain.  We proceed to derive a scalar, nonlinear boundary value problem that describes the height above the flat bed in the new coordinates.  As a consequence of \eqref{normalsurface}, the reformulated problem takes the form of a quasi-linear elliptic differential operator added to an integro-differential operator.  It is this second term that shall require the utmost most care; we will show it completely accounts for the effects of the stratification.  Given the long history of integral equations in the study of heterogeneous waves, perhaps the presence of the integro-differential operator here can be viewed as natural.  We emphasize, though, that one of the strengths of our method is that the governing equation is still ``more-or-less'' elliptic (in a sense that will become clear later) and thus, substantially easier to treat.
  
In \S 3, we prove the existence of a 1-parameter family of laminar flows.  For these solutions the height equation reduces to an ordinary differential equation.  However, because of the stratification term, this takes the form of a nonlinear boundary value problem with operator coefficients.  Analyzing the linearized problem along the curve of laminar flows, we use a result of Crandall and Rabinowitz to prove that there exists a simple generalized eigenvalue from which bifurcates a local curve of small amplitude solutions.  As one might expect given the volume of work on small amplitude stratified waves, this proves substantially more difficult than the corresponding results for the homogeneous case in \cite{CS1}.

The analysis of \S 4 continues the local curve to a global bifurcation curve using a degree theoretic argument.  We are thereby able to seamlessly treat both small and large amplitude solutions.  However, this require strong \emph{a priori} estimates in order to guarantee that the operator has the necessary compactness for it to be admissible in the sense of degree theory.  This will be achieved by first ``freezing'' the stratification term, then applying Schauder theory to the resulting uniformly elliptic operator.  Finally, we estimate the frozen operator in order to return to the full problem.  The main result of this section is a global bifurcation theorem in the form of an  alternative, following in the footsteps of Rabinowitz \cite{R1}.  

We devote \S 5 to investigating the nodal properties along the global curve.  It is shown that $\mathcal{C}$ can only return to the curve of laminar flows in a certain interval of parameter values.   

In \S 6 we provide uniform bounds in the $C^{3+\alpha}$-norm along the global curve.  These, with the results of \S 5, allows us to prove that \textsc{Theorem \ref{mainresult}} follows from the global bifurcation theorem.  This is done in \S 7.

\section{Reformulation of the Problem}

The main goal of this section is to reformulate problem \eqref{euler2}-\eqref{boundcond} so that the fluid domain $\overline{D_\eta}$ is transformed into a fixed domain, whose closure we shall denote $\overline{R}$.  As usual for two-dimensional incompressible flow, the main tool here will be the pseudo-stream function $\psi$ introduced in the previous section.  Recall that we have $\psi \equiv 0$ on the free surface, $\psi \equiv -p_0$ on the flat bed, where $p_0$ is the relative pseudo-mass flux and $\psi$ is defined uniquely by the requirement that
\be \psi_x = -\sqrt{\rho} v, \qquad \psi_y = \sqrt{\rho}(u-c). \label{defpsi} \ee
In light of \eqref{defpsi} and \eqref{incompress}-\eqref{consmass1}, the governing equations inside the fluid become
\be \left\{ \begin{array}{lll}
\psi_y \psi_{xy} - \psi_x \psi_{yy} & = & -P_x \\
-\psi_y \psi_{xx} + \psi_x\psi_{xy} & = & -P_y - \rho g \end{array} \right. \qquad \textrm{in } \overline{D_\eta}, \label{equivform} \ee
whereas the boundary conditions \eqref{boundcond} are
\be \left\{ \begin{array}{llll}
\psi_x & = & -\psi_y \eta_x, & \textrm{on } y = \eta(x) \\
P & = & P_{\textrm{atm}}, & \textrm{on } y = \eta(x) \\
\psi_x & = & 0, & \textrm{on } y = -d. \end{array} \right. \label{equivboundcond} \ee

Recall we have that the quantity 
\be E = P + \frac{\rho}{2}\left( (u-c)^2 + v^2\right) + g\rho y, \label{defE1} \ee
is constant along streamlines.  In particular then, evaluating the above equation on the free surface $\psi \equiv 0$, we find
\be |\nabla \psi|^2 + 2g \rho(x, \eta(x)) \left(\eta(x)+d\right) = Q, \qquad \textrm{on } y = \eta(x) \label{defQ} \ee
where the constant $Q := 2(E|_\eta-P_{\textrm{atm}} + gd)$.  Note that this $Q$ gives roughly the energy density along the free surface of the fluid.  Critically, however, we shall see that as a consequence of our normalization of $\eta$, $d$ is not a parameter for the problem.  On the contrary, in all but the most trivial cases, $d$ varies along $\mathcal{C}$.  In our analysis, therefore, we shall instead be viewing $Q$ as parameterizing the continuum.

We now wish to find from \eqref{euler2} and \eqref{defE1} a scalar PDE satisfied by $\psi$. In the case where the fluid is homogeneous, it is easy to show that this will take the form of a semilinear equation:
\be \Delta \psi = \gamma(\psi) \label{homgoverneq} \ee
for some function $\gamma:[0, |p_0|]\to \mathbb{R}$.  It is not hard to show that, in this scenario, $\gamma$ will describe the change in vorticity as a function of the streamlines.  When one allows for density variation, however, one expects there must be an additional term accounting for the gravitational effects of the stratification.

One can prove that the corresponding relation in the heterogeneous case takes the form
\[ \frac{d E}{d \psi} = \Delta \psi + gy \frac{d \rho}{d \psi}. \]
This is known as Yih's equation or the Yih--Long equation.  In this paper, it shall serve as our governing equation for $\psi$.  Recall that in the previous section we introduced the Bernoulli function $\beta$ and streamline density function $\rho$.  Rewriting the above expression we arrive at an equation of roughly the same form as \eqref{homgoverneq}:
\be \beta(\psi) = \Delta \psi - gy \rho^\prime(-\psi). \label{yiheq} \ee
Were it not for the presence of $y$ on the right-hand side, this would reduce to the homogeneous case --- a fact that agrees with our intuitive notion that density stratification should reintroduce the depth into the problem as a serious consideration. We remark that comparing this to \eqref{homgoverneq} we see that the final term on the right accounts precisely for the gravitational effects of stratification, the inertial effects having been captured in the choice of the $\psi$.    

With \eqref{yiheq} in hand, we now make a change of variables to eliminate the free boundary.  The new coordinates we will denote $(q,p)$ where 
\[ q = x, \qquad p = -\psi(x,y).\]
This scheme is sometimes referred to as \emph{semi-Lagrangian} coordinates, in recognition of the fact that we are working, in some sense, halfway between streamline coordinates and the usual Lagrangian system (cf. \cite{T2}).  

By means of a scaling argument, we may take $L := 2\pi$.  Then, under the transformation
\[ (x,y) \mapsto (q,p),\]
the closed fluid domain $\overline{D_\eta}$ is mapped to the rectangle
\[ \overline{R} := \{ (q,p) \in \mathbb{R}^2 : 0 \leq q \leq 2\pi,~p_0 \leq p \leq 0 \}.\]
The purpose of the minus sign is simply to flip the rectangle so that the free surface will correspond to the top of $R$, while the flat bed will mapped to the bottom.  Given this, it will be convenient to put
\[ T := \{ (q,p) \in R : p = 0 \}, \qquad B := \{ (q,p) \in R : p = p_0 \}. \]
  
Note that, in light of \eqref{defrho} and \eqref{defbeta}, we have that
\[ \beta = \beta(-p), \qquad \rho = \rho(p).\]
Moreover, the assumption that the streamline density function is nonincreasing becomes
\be \rho_p \leq 0. \label{rhoincdepth} \ee
  
Next, following the ideas of Dubreil--Jacotin (cf. \cite{DJ1,DJ2}), define 
\be h(q,p) := y + d \label{defh} \ee
which gives the height above the flat bottom on the streamline corresponding to $p$ and at $x = q$.  We calculate:
\be \psi_y = -\frac{1}{h_p}, \qquad \psi_x = \frac{h_q}{h_p} \label{hqhpeq}. \ee
Note that this implies $h_p > 0$, because we have stipulated that $u < c$ throughout the fluid.  The change of variables then gives:
\be \begin{array}{lll}
\partial_x & = & \partial_q -\frac{h_q}{h_p}\partial_p   \\
\partial_y & = & h_p^{-1} \partial_p  \\
\partial_p & = & h_p \partial_y  \\
\partial_q & = & \partial_x +h_q \partial_y. \end{array} \label{dqdpdxdy} \ee
Using these expressions we can solve for $u$ and $v$ in \eqref{hqhpeq} to obtain
\be u  =  c - \frac{1}{\sqrt{\rho} h_p},~\qquad v = -\frac{h_q}{\sqrt{\rho} h_p}. \label{uvhphq} \ee

In order to formulate \eqref{yiheq} in terms of $h$, we observe that
\begin{eqnarray*}
\Delta \psi & = & \partial_x^2 \psi+ \partial_y^2 \psi \\
& = & -\partial_x (\sqrt{\rho} v) + \partial_y (\sqrt{\rho}(u-c)) \\
& = & \partial_x \bigg( \frac{h_q}{h_p} \bigg) + \partial_y (-h_p^{-1}) \\
& = & \bigg(\frac{-h_q}{h_p}\bigg) \partial_p \bigg( \frac{h_q}{h_p} \bigg) + \partial_q \bigg(\frac{h_q}{h_p}\bigg) + h_p^{-1} \partial_p (-h_p^{-1}) \\
& = & \bigg(\frac{-h_q}{h_p}\bigg) \bigg(\frac{h_p h_{pq} - h_q h_{pp}}{h_p^2}\bigg) + \frac{h_p h_{qq} - h_q h_{pq}}{h_p^2} + \frac{h_{pp}}{h_p^3} \end{eqnarray*}
Hence, Yih's equation \eqref{yiheq} becomes the following
\be -h_p^3 \beta(-p) =  (1+h_q^2)h_{pp} + h_{qq}h_p^2 - 2h_q h_p h_{pq} - g(h-d) h_p^3 \rho_p \label{interheighteq} \ee
where we have used \eqref{defh} to write $y = h-d$.  Recall, however, that we have normalized $\eta$ so that it has mean zero.  Taking the mean of \eqref{defh} along $T$, we obtain
\be d = d(h) = \fint_0^{2\pi} h(q,0)dq. \label{defd} \ee
That is, the average depth $d$ must be viewed as a linear operator acting on $h$.  Namely, it is the average value of $h$ over $T$.  Where there is no risk of confusion, we shall suppress this dependency and simply write $d$.  As we shall see, the addition of this integral term to the governing equations will be the single most significant departure from the homogeneous case, both technically and qualitatively.  In recognition of this fact, we shall refer to the mapping $h \mapsto -g(h-d(h))h_p^3 \rho_p$ as the \emph{stratification operator}.  Equation \eqref{interheighteq} will then take the form of a quasilinear elliptic operator added to the stratification operator.  

Next consider the boundaries of the transformed domain.  On the bed we must have by the definition of $h$ that 
\be h \equiv 0, \qquad \textrm{on } B. \label{heightboundcondB} \ee
If we were to reformulate the problem in terms of $y$ not $h$, then the integral term would disappear in \eqref{interheighteq}, but the boundary condition on $B$ would become $y \equiv -d$.  We see then that, in the presence of density variation, the depth cannot simply be eliminated from the problem.

Throughout the fluid we have by \eqref{defpsi} and \eqref{uvhphq}:
\[ \psi_x^2 = h_q^2 h_p^{-2}, \qquad \psi_y^2 = h_p^{-2}. \]
Given this, the definition of $Q$ given in $\eqref{defQ}$ becomes the requirement
\be 1 + h_q^2+ h_p^2\left( 2g \rho h - Q\right) = 0, \qquad \textrm{on } T. \label{heightboundcondT} \ee

Altogether, then, combining \eqref{interheighteq}, \eqref{heightboundcondB} and \eqref{heightboundcondT} we have that the fully reformulated problem is the following. Find $(h,~Q) \in C_{\textrm{per}}^{2+\alpha}(\overline{R}) \times \mathbb{R}$ satisfying
\be \left \{ \begin{array}{lll}
(1+h_q^2)h_{pp} + h_{qq}h_p^2 - 2h_q h_p h_{pq} - g(h-d(h))h_p^3 \rho_p = -h_p^3 \beta(-p), & p_0 < p < 0 \\
1+h_q^2 + h_p^2( 2g \rho h - Q)  = 0, & p = 0 \\
h = 0, & p = p_0, \end{array} \right. \label{heighteq} \ee  \\
and $h_p > 0$.  Here $\rho \in C^{1+\alpha}([p_0,0];\mathbb{R}^+)$ and $\beta \in C^{1+\alpha}([0,|p_0|]; \mathbb{R})$ are given function with $\rho_p \leq 0$.  

We now prove the equivalence of the height equation problem to the original Euler equation formulation.  Because the differential equations relating $\psi$ and $h$,  \eqref{equivform}-\eqref{equivboundcond}, do not include $\rho$, this result follows more or less from the methods of the constant density case in \cite{CS1}. Nonetheless, for clarity we recapitulate that argument here. \\

\begin{lemma} Problem \eqref{heighteq} is equivalent to problem \eqref{incompress}-\eqref{bedcond}
\label{equivalencelemma} \end{lemma}

\begin{proof} The preceding development shows that \eqref{incompress}-\eqref{bedcond} implies \eqref{heighteq} for some $\rho$, $\beta$ and $Q$. It remains only to prove the converse.  Fix $p_0$ and let $h$ be a solution to \eqref{heighteq} of class $C_{\textrm{per}}^{2+\alpha}(\overline{R})$ corresponding to a given value of $Q$ and streamline density function $\rho$, Bernoulli function $\beta$, with $h_p > 0$ in $\overline{R}$.

Define the $C^{1+\alpha}_{\textrm{per}}(\overline{R})$ functions 
\be F(q,p) := \frac{1}{h_p(q,p)}, \qquad G(q,p) := -\frac{h_q(q,p)}{h_p(q,p)}. \label{defFG} \ee
Then, as $h_{pq} = h_{qp}$,
\be F_q + F_p G - G_p F = -h_{pq} h_p^{-2} + h_q h_{pp} h_p^{-3} + h_{pq}h_p^{-2} - h_{pp} h_q h_p^{-3} = 0 \label{identityFG} \ee 
throughout $\overline{R}$.  Also, we note that the free surface of the flow is given by $\eta(x) = h(x,0)-d(h)$.

Our first task is to recover the pseudo-stream function $\psi$.  Fix $x_0 \in \mathbb{R}$ and let $\psi$ denote the solution of the ODE
\be \psi_y(x_0, y) = -F(x_0, -\psi(x_0,y)) \label{psiFode} \ee 
along with the initial condition $\psi(x_0,\eta(x_0)) = 0$.   By assumption, $h_p$ is bounded strictly away from zero on $\overline{R}$.  We may let $\delta > 0$ be given such that $F \geq \delta$.  Thus $\psi(x_0, \cdot)$ is increasing at a rate greater than $\delta$ as $y$ decreases.  Thus \eqref{psiFode} is solvable until $\psi(x_0,y) = -p_0$, for some $y$.  Then for each $x \in \mathbb{R}$, we may define $\psi(x,y)$ on some interval $[y(x), \eta(x)]$ with $y < \eta$.  By uniqueness of solutions to \eqref{psiFode} and the periodicity of $h$, it follows that $\psi$ is periodic in $x$ within its domain of definition.  

We claim that $y(x) = -d(h)$ for all $x \in \mathbb{R}$.  To prove this, let $H(x,y) := -G(x,-\psi(x,y))$ for $x \in \mathbb{R}$, $y \in [y(x), \eta(x)]$.  Then by \eqref{identityFG} and \eqref{psiFode}, 
\[ H_y(x,y) = G_p\psi_y = -G_p F = -F_q -F_p G = - F_q + F_p H.  \]
On the other hand, by $C^1$-dependence of $\psi$ on the parameter $x_0$, we may differentiate to find
\[ (\psi_x)_y = (\psi_y)_x = -F_q + F_p \psi_x. \]
Thus $\psi_x$ and $H$ satisfy the same ODE.  We have $\psi(x,\eta(x)) = 0$ by definition.  Differentiating this relation yields
\[ \psi_x  =  -\psi_y \eta^\prime = F(x, -\psi(x,\eta(x))) \eta^\prime. \]
Likewise, by definition of $\eta$ we have
\[ \eta^\prime(x) = h_q,\]
and thus
\[ H(x,\eta(x)) = -G(x,-\psi(x,\eta(x))) = F(x,-\psi(x,\eta(x))) \eta^\prime(x)\]
by \eqref{defFG}.  We have shown that $H$ and $\psi_x$ have identical initial data at $(x,\eta(x))$, hence by uniqueness we conclude
\be \psi_x(x,y) = H(x,y) = - G(x,-\psi(x,y)), \qquad \forall x \in \mathbb{R},~y \in [y(x), \eta(x)]. \label{HequalG} \ee

Now, $C^1$-dependence of $y$ on $x$ enables us to differentiate the relation $\psi(x,y(x)) = -p_0$,
\[ \psi_x(x, y(x)) + \psi_y(x,y(x)) \frac{d y}{dx} = 0, \qquad \forall x \in \mathbb{R}.\]
By \eqref{psiFode} and \eqref{HequalG} this implies
\[ G(x,-p_0) -F(x,-p_0) \frac{dy}{dx} = 0  , \qquad \forall x \in \mathbb{R}.\]
But as we have seen, $h_q = \eta^\prime$, so that $h_q(q, p_0) = 0$ and therefore $G(\cdot, -p_0) \equiv 0$.  Also, we have that $F \geq \delta > 0$, so we may conclude that $y(x) \equiv y_0$, a constant.

Finally, we must show that $y_0 = p_0$.  To do so we observe that $h(0,p_0) = 0$, $h(0,0) = \eta(0)+d(h)$.  Then by \eqref{psiFode} evaluated at $x_0 = 0$, we have
\[ \begin{array}{lcl} 
\eta(0)+d = h(0,0) = \int_{p_0}^0 h_p(0,p)dp & = & \int_{p_0}^0 \frac{dp}{F(0,p)} \\
& & \\
& = & \int_{y_0}^{\eta(0)} \frac{-d\psi(0,y)}{F(0,-\psi(0,y))} \\
& & \\
& = & \int_{y_0}^{\eta(0)}  dy \\ 
& & \\
& = & \eta(0) - y_0. \end{array}  \]
Thus $y_0 = -d$.

It remains now to show that $\psi$ constructed above constitutes a pseudo-stream function for a solution to \eqref{incompress}-\eqref{bedcond}.  We have already proved that $\psi = p_0$ on $y = -d$, and $\psi = 0$ on $y = \eta(x)$.  Also,
\be F^2+G^2 = \psi_y^2+\psi_x^2 = |\nabla \psi|^2, \qquad \textrm{in } \overline{R} \label{F2G2identity} \ee
by \eqref{defFG} and \eqref{HequalG}.   On the other hand, from the boundary condition at $p=0$ in \eqref{heightboundcondT}, we find $F^2+G^2 = Q-2g\rho h$ on the free surface.  But as $q = x$, and $h = y+d$, we may combine this with \eqref{F2G2identity}  to conclude that $|\nabla \psi|^2 +2g\rho (y+d) = 0$ on $y = \eta(x)$.

Finally, differentiating \eqref{psiFode} and \eqref{HequalG} and summing the two yields
\begin{eqnarray*}
\Delta \psi & = & F_p \frac{d\psi}{dy} + G_p \frac{d\psi}{dx} - G_q\\
& = & -F_p F - G_p G - G_q  \\
& = & h_{pp} h_p^{-3} + h_{pp} h_q^2h_{p}^{-3} - h_{pq} h_q h_p^{-2} + h_{qq}h_p^{-1}-h_q h_{pq} h_p^{-2}. \end{eqnarray*}
In light of (2.3), this becomes
\be \Delta \psi = -\beta(\psi) + gy \rho^\prime(-\psi), \qquad \textrm{in } \overline{R}. \label{yiheq2} \ee

We now define $u$ and $v$ through \eqref{uvhphq}, and, by abuse of notation, take $\rho(x,y) = \rho(-\psi(x,y))$. Then, recalling the moving frame transformation and the definition of the pseudo-stream function, we have by \eqref{yiheq2} and the arguments of the preceding paragraph that, $(u,v,\rho, \eta) \in C^{1+\alpha}_\textrm{per} (\overline{D_\eta}) \times C^{1+\alpha}_{\textrm{per}}(\overline{D_\eta}) \times C_{\textrm{per}}^{1+\alpha}(\overline{D_\eta}) \times C_{\textrm{per}}^{2+\alpha}(\overline{D_\eta})$ solves the original problem \eqref{incompress}-\eqref{bedcond}. \qquad\end{proof} \\

\section{Local Bifurcation}
The goal of this section is to prove the existence of a local curve of small-amplitude solutions to \eqref{incompress}-\eqref{bedcond}.  The main product of our efforts will be the following the theorem. \\

\begin{theorem} \emph{(Local Bifurcation)} Let $c > 0$, $p_0 < 0$, $\alpha \in (0,1)$, and $C^{1+\alpha}$ functions $\beta$ and $\rho$ defined on $[0, |p_0|]$, $[p_0, 0]$ respectively be given.  Consider the traveling solutions of speed $c$ with relative mass flux $p_0$ of the water wave problem \eqref{incompress}-\eqref{bedcond} with Bernoulli function $\beta$ and streamline density function $\rho$ such that $u < c$ throughout the fluid.  If $\beta$ and $\rho$ satisfy \emph{(L-B)} (see \textsc{Definition \ref{defLBcond}}) and $\rho_p \leq 0$, then there exists a $C^1$ curve $\mathcal{C}_{\textrm{loc}}$ of small-amplitude solutions $(u,v,\rho,\eta) \in C_{\textrm{per}}^{2+\alpha}(\overline{D_\eta}) \times C_{\textrm{per}}^{2+\alpha}(\overline{D_\eta}) \times C_\textrm{per}^{1+\alpha}(\overline{D_\eta}) \times C_{\textrm{per}}^{3+\alpha}(\overline{D_\eta})$.  The solution curve $\mathcal{C}_{\textrm{loc}}$ contains precisely one laminar flow (with $\eta \equiv 0$). \label{localbifurcationtheorem} \end{theorem} \\

\subsection{Overview}
Following the ideas set forth in the previous section, we shall work with the equivalent height equation formulation \eqref{heighteq}.  First we shall prove the existence of a 1-parameter family $\mathcal{T}$ of laminar solutions.  We then show the existence of a local curve of non-laminar solutions bifurcating from a simple eigenvalue of the linearized problem along $\mathcal{T}$.   The result will then follow from an application of the local bifurcation theory of Crandall and Rabinowitz.   

In several ways, it will be this section where the stratification term will be problematic.  This will be immediately apparent in the next lemma, where its presence will make unattainable an explicit formula for the laminar solutions.  The consequences of this, at least from a technical standpoint, will cascade through the subsequent development, as we will be forced to make perturbation arguments along a curve of solutions we can only hope to represent implicitly.  Though these concerns are purely mathematical in nature, they are not without physical analogue.  Indeed, as we remarked earlier, one of the striking features of stratified flows is their propensity to exhibit significant internal waves --- even in a small amplitude regime.   

\subsection{Laminar Flow}

Consider laminar flow solutions to the height equation \eqref{heighteq}.  By laminar we mean parallel shear flows where the free surface is flat.  Any such solution must then take the form $H = H(p)$, with $\eta \equiv 0$.  The height equation then reduces to an ODE with operator coefficients:
\be H_{pp} - g(H-d(H))H_p^3 \rho_p = -H_p^3 \beta(-p), \qquad p_0 < p < 0 \label{interlaminar} \ee
along with with the boundary conditions
\be H = 0, \qquad \textrm{on } p = p_0 \label{laminarboundcondB} \ee
and
\be 1+H_p^2(2g\rho H-Q)=0, \qquad \textrm{on } p = 0. \label{laminarboundcondT} \ee
Owing mainly to the presence of the stratification term, \eqref{interlaminar}-\eqref{laminarboundcondB} constitutes a nonlinear, non-autonomous ODE boundary value problem which we cannot solve explicitly.  Our approach is motivated by the observation that if we can replace $\beta$ and $\rho$ with functions of $y$, the resulting autonomous ODE becomes tractable.  Indeed, any solution to the height equation should satisfy $H_p > 0$, hence the vertical variable $y$ is a strictly increasing function of $p$, and so inverting the two is a valid change of variables.  This is an adaptation of a technique commonly seen in the analysis of solitary waves in a channel, where it can be convenient to invert $p$ and $y$ on the profile at $\pm\infty$ (cf. \cite{AT, LF,T1} and many others).
As in the first section, define
\[ B(p) := \int_0^p \beta(-s) ds, \qquad B_\textrm{min} := \min_{p \in [p_0, 0]} B(p) \leq 0.\]
\noindent
\begin{lemma}
\emph{(Laminar Flow)} 
Suppose that the streamline density function $\rho$ satisfies \eqref{rhoincdepth}.  Then there exists a 1-parameter family of solutions $H(\cdot;\lambda)$ to the laminar flow equation \eqref{interlaminar}-\eqref{laminarboundcondT} with $H_p > 0$, where $0 \leq  -2B_{\min{}} < \lambda < Q$. \label{laminarflowlemma} \end{lemma} 

\begin{proof} In accordance with \eqref{defh}, let $Y := H - d$ denote the vertical variable.  By specifying the free surface $\eta \equiv 0$, we force $Y(0) = 0$.  Thus $H(0) = d$, and so by \eqref{laminarboundcondT},
\be 1+Y_p^2(2g\rho d - Q) = 0, \qquad \textrm{on } p = 0. \label{laminardeq} \ee 
It therefore suffices to find $Y$ satisfying
\be Y_{pp} + \left(\beta(-p) - g Y \rho_p \right) Y_p^3 = 0, \qquad p_0 < p < 0 \label{interYlaminareq} \ee
along with the boundary conditions
\be Y = 0, \qquad \textrm{on } p = 0, \label{YboundcondT} \ee
\be Y = -d, \qquad \textrm{on } p = p_0, \label{YboundcondB} \ee
where $d$ (viewed as a positive real number) and $Y$ additionally satisfy \eqref{laminardeq} for some choice of $Q$.  

We now change variables.  Put $s := Y(p)$.  Then from \eqref{interYlaminareq} we calculate
\begin{eqnarray*}
Y_{pp}Y_p^{-3} & = & \left(-\frac{dB}{ds} + gY\frac{d\rho}{ds}\right) Y_p \\
& = & \left(-\frac{dB}{ds} + gY\frac{d\rho}{ds}\right) \frac{ds}{dp}. \end{eqnarray*}
Therefore we may rewrite \eqref{interYlaminareq} as 
\be \bigg(-\frac{1}{2} Y_p^{-2} \bigg)_p + \bigg(F(Y)\bigg)_p = 0, \qquad p_0 < p < 0 \label{factoredlaminareq}, \ee
where 
\[ F^\prime(Y) = \left( \frac{dB}{ds} - gY\frac{d\rho}{ds} \right)\bigg|_{s = Y} = \left(\frac{dB}{ds} - gY\frac{d\rho}{dp} \right)\bigg|_{Y = s},\]
or equivalently
\be \frac{dF}{ds} = \left(\frac{dB}{dp} - gs\frac{d\rho}{ds} \right) \frac{dp}{ds}. \label{defdFds} \ee
For definiteness we take $F(0) = 0$, so that integrating we find
\be F(Y) = {B}(p) + \int_Y^0 gr \frac{d\rho}{ds}(r)dr, \label{FequalBplusintY} \ee
and
\be F(Y) = {B}(p) + \int_p^0 g Y(r) \frac{d\rho}{dp}(r)dr.\label{FequalBplusintp} \ee
Note that by \eqref{rhoincdepth}, the integrand on the right-hand side is nonnegative, hence $F > B_{\textrm{min}}$.  
 
Returning to \eqref{factoredlaminareq} we integrate once to obtain
\be Y_p = \frac{1}{\sqrt{\lambda + 2F(Y)}}.  \label{Ypeq} \ee  
where $\lambda$ is a constant of integration with $\lambda > -2B_{\textrm{min}}$.  Note that, by \eqref{FequalBplusintY}, $F$ is actually an unknown.  Using the fact that $F(0) = 0$, \eqref{Ypeq} and \eqref{laminardeq} together imply
\be d= \frac{Q-\lambda}{2g\rho(0)}. \label{dQlambda} \ee
Inverting \eqref{Ypeq} and combining it with \eqref{defdFds} we arrive at a first order system for $p$ and $F$:
\be \left\{ \begin{array}{lll}
\frac{dp}{ds} & = & \sqrt{\lambda + 2F(s)}, \\
\frac{dF}{ds} & = & \bigg(B^\prime(p) - gs \rho^\prime (p) \bigg) \sqrt{\lambda+2F(s)}, \end{array} \right.
\label{pFsystem} \ee
where by abuse of notation $\rho$ and $B$ here have been extended continuously so that their domain encompasses the entire real line.  We have that both $Y$ and $p$ vanish when $s = 0$.  Locally, then, we can solve this initial value problem.  Let $(F,p) = (F(s;\lambda),p(s;\lambda))$ be such a solution for a fixed choice of $\lambda > -2B_{\textrm{min}}$ and let the maximum domain of definition be $(s_{-}(\lambda), s_{+}(\lambda))$.  We must show that $p$ attains the value $p_0$ somewhere in $(s_{-}, 0)$.

Elementary theory of ordinary differential equations tells us that, if the domain of definition is bounded from below, then the solution must blow-up as we approach the boundary.  More precisely, we must have $|F|^2+|p|^2 \to \infty$ as $s \to s_{-}$, provided that $s_{-} > -\infty$. First consider the $F$ equation in \eqref{pFsystem}.  Rearranging terms and integrating in $s$ yields
\[ F =  \bigg( \int _0^s \left(B^\prime(p) -gr \rho^\prime(p) \right) dr  \bigg)^2 - \lambda
\Longrightarrow |F(s;\lambda)| \leq |\lambda| + \bigg( \|\beta\|_\infty |s| + gs^2 \|\rho^\prime \|_\infty \bigg)^2. \]
It follows that $F$ is bounded for finite $s$.
Next, recall that we have chosen $\lambda > -2 B_{\min{}}$, whence
\[ \lambda + 2F > \lambda + 2 B_{\min{}} > 0. \]
Puting $\epsilon(\lambda) := \lambda + 2 B_{\min{}}$, by \eqref{pFsystem} we have 
\be \frac{dp}{ds} \geq \sqrt{\epsilon} > 0. \label{lowerbounddpds} \ee
But this uniform lower bound on $\frac{dp}{ds}$ implies that, if $s_{-} = -\infty$, then $p \to -\infty$ as $s \to s_{-}$.  On the other hand, if $s_{-} > -\infty$, then the arguments of the previous paragraph guarantee that $|F|$ is bounded on $(s_{-}, 0]$, so again we conclude $p\to -\infty$ as $s \to s_{-}$.  In either event, we may choose $d = d(\lambda) > 0$ to be the unique number satisfying $p(-d) = p_0$.  Inverting $p$ (which is valid by \eqref{lowerbounddpds}) we get a function $Y = Y(p;s)$ satisfying \eqref{interYlaminareq}-\eqref{YboundcondB}.  Finally, let $Q=Q(\lambda)$ be defined by the relation \eqref{dQlambda}, 
\be Q(\lambda) = \lambda + 2g\rho(0)d(\lambda). \label{defQoflambda} \ee 
Thus for each $\lambda > -2B_{\textrm{min}}$, $H(p;\lambda) := Y(p;\lambda) + d(\lambda)$ is a solution  the laminar flow equations with $Q = Q(\lambda)$. \qquad\end{proof} \\

\subsection{Eigenvalue Problem}
Up to this point we have produced a family, $\mathcal{T}$, of laminar solutions to the height equation parametrized by the variable $\lambda$, which is drawn from a suitable range determined \emph{a priori} by the given function $\beta$ and the value of $p_0$.  We will now linearize the full height equation around $H$ for fixed $\lambda$ by evaluating the Fr\'echet derivative.  Ultimately we show that the linearized problem has a simple eigenvalue at some $\lambda^*$.

Fix $\epsilon > 0$ and let $h(q,p) = H(p) + \epsilon m(q,p)$.  Then from \eqref{heighteq}
\begin{eqnarray*}
 (1+ \epsilon^2 m_q^2)(H_{pp} + \epsilon m_{pp}) + \epsilon m_{qq} (H_p + \epsilon m_p)^2 - 2 \epsilon^2 m_q m_{qp} (H_p + \epsilon m_p) \\ -g(H+\epsilon m - d(H) - \epsilon d(m))(H_p + \epsilon m_p)^3 \rho_p & = & -(H_p + \epsilon m_p)^3 \beta(-p), \end{eqnarray*}
for $p_0 < p < 0$.  Differentiating in $\epsilon$ and evaluating at $\epsilon = 0$ yields the linearized equation:
\be m_{pp} + m_{qq}H_p^2 -g \rho_p \Big(3m_p (H-d(H)) H_p^2 + (m-d(m))H_p^3\Big) = -3H_p^2 m_p \beta(-p), \label{interlinearheighteq} \ee
for $p_0 < p < 0$.

Applying the same procedure to the two boundary conditions we find
\[ 2g\rho m H_p^2 + 2( 2g\rho H -Q)m_pH_p = 0, \qquad \textrm{on } p = 0\]
and 
\[ m = 0, \qquad \textrm{on } p = p_0.\]
We may simplify the nonlinear condition on $T$ by noting that for all $\lambda$,
\be H(0) = \frac{Q-\lambda}{2g\rho(0)} \Longrightarrow 2g\rho(0) H(0) - Q = -\lambda.  \label{Hzerominuslambda} \ee
Also, from \eqref{hqhpeq},
\be H_p^2 = Y_p^2 = \frac{1}{\lambda+2F(Y)} = \lambda^{-1}, \qquad \textrm{on } p = 0. \label{Hp2lambdaminus1} \ee
Combining these we find that the the linearized problem is the following:  for fixed $\lambda > -2B_{\textrm{min}}$, find $m$ that is $2\pi$-periodic in $q$ satisfying
\be \left \{ \begin{array}{lll}
m_{pp} + m_{qq}H_p^2 -g \rho_p \Big(3m_p (H-d(H)) H_p^2 + (m-d(m))H_p^3\Big) & \\
\qquad =  -3H_p^2 m_p \beta(-p), & p_0<p<0 \\
m =  0, & p = p_0 \\
g\rho m  =  \lambda^{3/2} m_p, & p = 0. \end{array} \right. \label{linearheighteq} \ee
Observe that the only effect of the variable density on the boundary is in the addition of the constant $\rho(0)$ on $p = 0$.  This similarity with the constant density case will allow many of the same calculations to push through virtually unaltered.  On the other hand, the stratification term in \eqref{interlinearheighteq} will present a significant technical barrier as it introduces both the nonlocal operator $d$, and a zero-th order term.

From this section onward we shall consider only $\lambda$ in the range $\lambda \geq -2B_{\min{}}+\epsilon_0$, where $\epsilon_0$ is as in \eqref{defepsilon0}.  As will be made clear later, the purpose of slightly decreasing our range of admissible $\lambda$ is essentially to ensure that $H_p$ is bounded uniformly away from zero for $\lambda$ small.  

In order for bifurcation to occur, some control over the relative sizes of $\rho$, $B$, $\beta$ and $p_0$ is necessary --- even in the homogeneous case.  With this in mind, we make the following definition.

\begin{definition} We say the pseudo-volumetric mass flux $p_0$, streamline density function $\rho$ and Bernoulli function $\beta$ collectively satisfy the Local Bifurcation Condition provided 
\be  \inf_{\lambda \geq -2B_{\textrm{min}} + \epsilon_0} ~\inf_{\phi \in \mathscr{S}} \left\{ \frac{-g\rho(0)\phi(0)^2 + \int_{p_0}^0 H_p^{-3}(p;\lambda) \phi_p(p)^2 dp }{ \int_{p_0}^0 \phi^2(p) \left(H_p^{-1}(p;\lambda) + g \rho_p(p)\right)dp} \right\} < -1 \tag{L-B} \ee
where $\mathscr{S} := \{ \phi \in H^1((p_0,0)) : \phi \nequiv 0,~\phi(p_0) = 0 \}$ and $H(\cdot;\lambda)$ is the solution to \eqref{interlaminar}-\eqref{laminarboundcondT} given by \textsc{Lemma \ref{laminarflowlemma}}.  \label{defLBcond} \end{definition} \\

\noindent
We shall restrict our attention to case where (L-B) is satisfied.  The next lemma gives the precise motivation for this choice.  \\

\begin{lemma} \emph{(Eigenvalue Problem)} Assuming \eqref{rhoincdepth} and \emph{(L-B)}, there exists $\lambda^* > -2B_{\min{}} + \epsilon_0$ and a solution $m(q,p) \nequiv 0$ of \eqref{linearheighteq} that is even and $2\pi$-periodic in $q$. \label{eigenvalueproblem} \end{lemma}

\begin{proof} We seek special solutions of the form $m(q,p) = M(p) \cos(kq)$ for some $k \in\mathbb{Z}^\times$, as we require $2\pi$-periodicity in $q$.  Note that this implies
\[ d(m) = M(0) \fint_0^{2\pi}  \cos(kq) dq = 0.\]
Also, since solutions of this form will necessarily satisfy the ODE 
\[ m_{qq} = -k^2 m,\]
 we may rewrite the interior equation of \eqref{linearheighteq} in self-adjoint form:
\[ \{H_p^{-3} m_p\}_p + \{(H_p^{-1} + g \rho_p k^{-2})m_q\}_q = 0. \]
In turn, this requires that $M$ satisfy
\be\{H_p^{-3} M_p \}_p = (k^2 H_p^{-1} + g\rho_p)M, \qquad p_0 < p < 0. \label{sturmliouvilleode} \ee
Any value of $k$ will suffice, but for simplicity we focus on finding a solution where $k = 1$.  To do so we approach \eqref{sturmliouvilleode} as a Sturm-Liouville problem.   Note that, as a consequence of \eqref{defepsilon0}, the term in parenthesis above is nonnegative for $k = 1$.  

In that connection we consider the minimization problem:
\be \mu = \mu(\lambda) = \inf_{\phi \in \mathscr{S}} \mathcal{R}(\phi; \lambda) \label{minproblem} \ee
where the Rayleigh quotient $\mathcal{R}$ is
\be \mathcal{R}(\phi; \lambda) := \frac{-g\rho(0)\phi(0)^2 + \int_{p_0}^0 H_p(p;\lambda)^{-3} \phi_p(p)^2 dp }{ \int_{p_0}^0 \phi(p)^2 \left(H_p(p;\lambda)^{-1} + g \rho_p(p) \right)dp} \label{defcalR} \ee
and $\mathscr{S} := \{ \phi \in H^1((p_0,0)) : \phi \nequiv 0,~\phi(p_0) = 0 \}$.  For each $\lambda$, the function $M$ that attains the minimum, $\mu(\lambda)$, will satisfy the boundary conditions of the linearized problem \eqref{linearheighteq} as well as the ODE:
\[ \{H_p^{-3} M_p\}_p = -\mu(\lambda)(H_p^{-1} + g \rho_p)M. \]
The task is to find $\lambda^* > -2B_{\textrm{min}} + \epsilon_0$ such that $\mu(\lambda^*) = -1$.  The corresponding $M$ will be precisely the special solution we seek.

By the continuity of $\mu$ and in light of the (L-B), it suffices to prove that $\mu(\lambda) \geq -1$ for $\lambda$ sufficiently large.  By \eqref{rhoincdepth} we have that $\rho_p \leq 0$, and we found in the previous section that this induces a uniform lower bound on $F(\cdot;\lambda)$. That is, for any $\lambda > -2B_{\textrm{min}} + \epsilon_0$,  
\[\min_p F(Y(p;\lambda);\lambda) \geq B_{\min{}},\]
We may therefore choose $\lambda$ such that
\[ \lambda > g\rho(0)+g\|\rho_p\|_{\infty}^2 + g^{3/2}\sqrt{\rho(0)} \|\rho_p\|_{\infty} - 2 B_{\min{}} \]
and conclude
\[ H_p^{-1} = (\lambda + 2F(Y(p)))^{1/2} \geq (\lambda + 2B_{\min{}})^{1/2} \geq g\|\rho_p\|_\infty+\sqrt{g\rho(0)}. \]
Then, again using the fact that $\rho_p \leq 0$, we obtain
\[ H_p^{-1} +g\rho_p \geq \sqrt{g\rho(0)}.\]

Now let $w \in \mathscr{S}$ be given and fix a $\lambda > -2B_{\textrm{min}} + \epsilon_0$.  Then
\begin{eqnarray*}
 \int_{p_0}^0 (H_p^{-3}w_p^2 + (H_p^{-1}+g\rho_p)w^2) dp & \geq & \sqrt{g\rho(0)} \int_{p_0}^0 (w^2 + g\rho(0)w_p^2)dp \\ & \geq & 2g\rho(0) \int_{p_0}^0 w w_p dp \\ & = & g\rho(0) w(0)^2.\end{eqnarray*}
 Thus $\mathcal{R}(w) \geq -1$.  As this holds for arbitrary $w \in \mathscr{S}$ and all admissible $\lambda$, we conclude $\mu(\lambda) \geq -1$. \qquad\end{proof}  \\

As we indicated in the first section, there is an explicit condition, \eqref{sizecond}, on $p_0$, $\rho$ and $\beta$ that implies (L-B).  Moreover, for any choice of $\beta$, $\rho$, taking $p_0$ sufficiently small (and restricting $\beta$ and $\rho$ to the decreased domain), will be enough to guarantee that this size condition will hold.  Though the size condition is not necessary, it is still general enough to allow for a great variety of flows.  \\

\begin{lemma} \emph{(Sufficiency of Size Condition)} If $p_0$, $\rho$ and $\beta$ satisfy the size condition \eqref{sizecond}, then they satisfy \emph{(L-B)}. \label{sizecondsufficiencylemma} \end{lemma} 

\begin{proof}  We must show that for some $\lambda \geq -2B_{\min{}}+ \epsilon_0$, we have $\mu(\lambda) < -1$ in the sense of \eqref{minproblem}.   In \eqref{lowerbounddpds} we showed that $\frac{dp}{ds}$ was bounded below uniformly in $p$.  Therefore,
 \[ -p_0 = -p(-d) = \int_{-d}^0 \frac{dp}{ds} ds \geq \sqrt{\epsilon_0}d.\]
If $\epsilon_0 = 0$, then we are in the constant density case and this lemma has already been proved in \cite{CS1}.  Otherwise,
\be  d \leq -\frac{p_0}{\sqrt{\epsilon_0}}. \label{relationdepsilon0}\ee
As $Y_p = H_p$,  we see by \eqref{Ypeq}
\begin{eqnarray*}
H_p^{-1}(p;\lambda) & = & \bigg(\lambda + 2{B}(p) + 2\int_p^0 g\rho_p(r)Y_\lambda(r)dr\bigg)^{1/2} \\
& \leq & (\lambda + 2{B}(p) + 2gd(\lambda)|p_0|\|\rho_p\|_\infty)^{1/2}  \\
& \leq & \bigg(\lambda + 2{B}(p) + \epsilon_0 \bigg)^{1/2}\bigg.  \end{eqnarray*}
where the last line comes from \eqref{relationdepsilon0} and the definition of $\epsilon_0$. Pairing this with the size condition \eqref{sizecond}, we have that for $\lambda$ sufficiently near $-2B_{\min{}}+ \epsilon_0$:
 \begin{eqnarray*}
  \int_{p_0}^0 (H_p^{-3}+(p-p_0)^2(H_p^{-1}+g\rho_p))dp &<& \int_{p_0}^0 \bigg((\lambda + 2{B}(p)+\epsilon_0)^{3/2} \\ 
  & & \qquad +(p-p_0)^2((\lambda + 2{B}(p)+\epsilon_0)^{1/2}+g\rho_p)\bigg)dp  \\
  &<& g\rho(0)p_0^2.\end{eqnarray*}
 Now take $w = p-p_0 \in \mathscr{S}$ and let $\lambda$ be such that the above inequality holds.  Then,
 \[ \mathcal{R}(w; \lambda) = \frac{-g\rho(0)p_0^2 + \int_{p_0}^0 H_p^{-3} dp}{\int_{p_0}^0(p-p_0)^2(H_p^{-1}+g\rho_p)dp} < -1.\]
 It follows that $\mu(\lambda) = \inf_{\mathscr{S}} R(\cdot; \lambda) < -1$, as desired.  Hence the Local Bifurcation Condition holds. \qquad\end{proof} \\
 
  For convenience we denote $G(p; \lambda) := 2F(Y(p;\lambda);\lambda)$.  Thus, for instance, we have $H_p(p) = (\lambda + G(p;\lambda))^{-1/2}$.   The great difficulty that arises from the stratification term in \eqref{interlaminar}, in large part stems from the fact that $G$ depends on $\lambda$.  Given that, before proceeding further, we set ourselves to a technical task: We must determining precisely how the $\lambda + G(\cdot; \lambda)$ term varies along $\mathcal{T}$.  
  
  For the next several proofs we will denote differentiation with respect to $\lambda$ by a dot.\\

 \begin{lemma} \emph{(Sign of $1+\dot{G}$)} For $\lambda \geq -2B_{\textrm{min}}+\epsilon_0$, $1+\dot{G}$ is strictly positive.  In fact, we have the following bound: $- 1/2 \leq \dot{G} \leq 0$. \label{signof1plusGdotlemma} \end{lemma}
 
\begin{proof} By rewriting \eqref{FequalBplusintp}, we have the following expression for $G$:
 \be G(p;\lambda) = 2{B}(p) + 2\int_p^0 gY(r; \lambda)\rho_p(r) dr. \label{Gequal2BintYrhop} \ee
 Since $B$ is independent of $\lambda$, and indeed the $\lambda$ of $G$ derives entirely from that of $Y$, we differentiate to find
 \be 1+\dot{G}(p;\lambda) = 1+ 2\int_p^0 g\dot{Y}(r;\lambda) \rho_p (r)dr \label{1plusGdot1} \ee
 But note that from (3.6) we can calculate
 \be \dot{Y}_p = -\frac{1}{2}(\lambda + G(p))^{-3/2}(1+\dot{G}) \Longrightarrow \dot{Y}(p) = \frac{1}{2}\int_p^0 (\lambda + G(r))^{-3/2}(1+\dot{G}(r))dr, \label{Ydotexpression} \ee
 hence
 \[ 1+\dot{G}(p) = 1+ \int_p^0 \int_r^0 g\rho^\prime(r)(\lambda + G(s))^{-3/2}(1+\dot{G}(s))dsdr.\]
Now we reverse the order of integration to find:
 \begin{eqnarray} 1+\dot{G}(p) & = &1 + \int_p^0 \int_p^s g\rho_p(r)(\lambda + G(s))^{-3/2}(1+\dot{G}(s))dr ds \nonumber \\
 & = & 1+\int_p^0 k(p,s)(1+\dot{G}(s))ds, \label{1plusGdotintkernel} \end{eqnarray}
 where $k(p,s) := \int_p^s g\rho_p(r)(\lambda+G(s))^{-3/2}dr$.   This is a linear Volterra integral equation of the second type, albeit a simple one.  As $k$ is a continuous map on $\Delta := \{(p,s): p_0 \leq p \leq s \leq 0\}$ we have that
 \be 1 + \dot{G}(p) = 1+\int_p^0 \tilde{k}(p,s)ds, \label{1plusGdot2} \ee
 where
 \[ k_1(p,s) := k(p,s), \qquad k_m(p,s) := \int_p^s k_{m-1}(p,r)k(r,s)dr, ~m \geq 2 \]
 and the resolvent kernel $\tilde{k}$ is given by
 \[ \tilde{k}(p,s) := \sum_{m=1}^\infty k_m(p,s).\]
The fact that the series converges comes from the following estimate:
 \[ |k_m(p,s)| \leq \bigg(\sup_{(p,s) \in \Delta} |k(p,s)|\bigg) \frac{(s-p)^{m-1}}{(m-1)!}. \]
From this it is clear that controlling the quantity in parentheses allows us to control the size $\tilde{k}$, and thereby the sign of $1+\dot{G}$.  However, as we are taking take $\lambda > -2B_{\min{}}+\epsilon_0$, we have
\[ \sup_{(p,s) \in \Delta} |k(p,s)| \leq g|p_0| \epsilon_0^{-3/2} \|\rho_p\|_\infty, \]
since $G \geq 2B_{\textrm{min}}$ implies $\lambda + G \geq \epsilon_0$. Inserting this into the definition of $\tilde{k}$ we find
\[ \sup_{(p,s)\in\Delta} |\tilde{k}(p,s)| \leq g |p_0| \epsilon_0^{-3/2} \|\rho_p\|_\infty e^{s-p} \leq |p_0| \epsilon_0^{-3/2} \|\rho_p\|_\infty e^{|p_0|},\]
Integrating this expression and applying the definition of $\epsilon_0$
\[ 1+\dot{G}(p) \geq 1- p_0^2 \epsilon_0^{-3/2} \|\rho_p\|_{\infty} e^{|p_0|}  = \frac{1}{2}. \]
Thus we have shown that $1+\dot{G}$ is strictly positive, as desired.

Finally, to derive the upper bound we merely reexamine \eqref{Ydotexpression}.  That is, from \eqref{Ydotexpression} we see that $\dot{Y} > 0$, as we have shown $1+\dot{G} > 0$.   But then \eqref{1plusGdot1} implies $\dot{G} \leq 0$. \qquad\end{proof}  \\

With the sign of $1+\dot{G}$ established, we are now in a position to better understand the relation between $Q$ and $\lambda$ set down by \eqref{defQoflambda}.  \\  
\noindent
\begin{corollary} \emph{(Convexity of Q)} For all $\lambda > -2B_{\min{}} + \epsilon_0$, we have that $Q$ is a convex function of $\lambda$ with a unique minimum $\lambda_0$. \label{convexityQcorollary} \end{corollary}

\begin{proof}  First, recall that by the definition, $Q(\lambda) = \lambda - 2g\rho(0)Y(p_0; \lambda)$, where $Y(\cdot; \lambda)$ is the corresponding solution to the laminar flow.   Differentiating twice in $\lambda$, we obtain
\[ -\frac{\ddot{Q}(\lambda)}{2g\rho(0)} = \ddot{Y}(p_0; \lambda). \]
It follows that to prove the corollary it suffices to show $\ddot{Y}(p_0; \lambda) < 0$ for all $\lambda > -2B_{\min{}} + \epsilon_0$.  

To do so we first note that by differentiating \eqref{1plusGdot1} by $\lambda$
\be \ddot{G}(p; \lambda) = 2 \int_p^0 g \ddot{Y}(r; \lambda) \rho^\prime (r)dr. \label{Gddotexpression} \ee
On the other hand, differentiating \eqref{Ydotexpression} we find
\be \ddot{Y}(p; \lambda) = \frac{1}{2} \int_p^0 \frac{\ddot{G}(r; \lambda)}{(\lambda+G(r; \lambda))^{3/2}} dr - \frac{3}{4} \int_p^0 \frac{(1+\dot{G}(r;\lambda))^2}{(\lambda+G(r;\lambda))^{5/2}} dr. \label{Yddotexpression} \ee
Substituting \eqref{Yddotexpression} into \eqref{Gddotexpression} and exchanging the order of integration we find that for each fixed $\lambda$, $\ddot{Y}( \cdot; \lambda)$ solves the integral equation 
\be \ddot{Y}(p) = \ell(p) + \int_p^0 j(p,r) \ddot{Y}(r) dr, \qquad p \in [p_0, 0] \label{Yddotintegraleq} \ee
where
\[ j(p,r) := \int_p^r \frac{g \rho^\prime(r)}{(\lambda + G(s))^{3/2}} ds, \qquad \ell(p) := - \frac{3}{2} \int_p^0 \frac{(1+\dot{G}(r))^2}{(\lambda+G(r))^{5/2}} dr.\]
Observe that $\ell \leq 0$  and $\|\ell\|_{\infty} = - \ell(p_0)$.  Now this is a Volterra integral equation with solution given by 
\[ \ddot{Y}(p) = \ell(p) + \int_p^0 \tilde{j}(p,r) \ell(r) dr \]
where the resolvent kernel $\tilde{j}$ is given in the same fashion as the previous lemma. From this representation formula we immediately derive the estimate
\be  \ddot{Y}(p_0; \lambda) \leq (1 - |p_0| \| \tilde{j} \|_\infty) \ell(p_0). \label{Yddotestimate} \ee
 As $\ell(p_0) < 0$, we are done if we can show that the quantity in parenthesis is positive.  Unsurprisingly, this is again a consequence of the definition of $\epsilon_0$ in \eqref{sizecond}.  Again let $\Delta$ denote the triangular region $\Delta := \{(p,r) : p_0 \leq p \leq r \leq 0\}$.  Observe that
 \[ \sup_{(p,r) \in \Delta} |j(p,r)| \leq \epsilon_0^{-3/2} g \|\rho_p\|_{\infty} |p_0|,\]
 and thus
 \[ \sup_{(p,r) \in \Delta} |\tilde{j}(p,r)| \leq |p_0| \epsilon_0^{-3/2} g \|\rho_p\|_{\infty} e^{|p_0|}.\]
 But by the definition of $\epsilon_0$, this becomes
 \[ \sup_{(p,r) \in \Delta} |\tilde{j}(p,r)| \leq \frac{1}{2|p_0|}. \]
 Inserting this last expression into \eqref{Yddotestimate} proves convexity of $Q$ (and the concavity of $Y$).
 
 To show the existence of a global minimum we note that
 \[ \dot{Q}(\lambda) = 1- 2g\rho(0)\dot{Y}(p_0; \lambda), \]
 thus $\dot{Q} > 0$ for $\lambda$ sufficiently large, since we have shown $\ddot{Y}(p_0; \lambda) < 0$, for all $\lambda$. \qquad\end{proof} \\
  
We now return to the linearized problem.  With the previous lemma in hand we are prepared to prove the monotonicity of $\mu$ in some neighborhood of $\lambda^*$.  This will in turn give the uniqueness of $\lambda^*$.  For convenience, we denote $a = a(p;\lambda) := H_p(p;\lambda)^{-1}$.  \\

\noindent
\begin{lemma} \emph{(Monotonicity of $\mu$)} $\mu$ is a strictly increasing function of $\lambda$ whenever $\mu(\lambda) < 0$. \label{monotonicitymulemma} \end{lemma}

\begin{proof}  For each $w \in \mathscr{S}$ and each $\lambda$, define $L_\lambda w := -\{H_p^{-3} w_p\}_p$.  For each $\lambda$, let $w(p) = w(p;\lambda)$ be the solution to the eigenvalue problem
\[ L_\lambda w = \mu(H_p^{-1}+g\rho_p)w, \qquad w(p_0) = 0, \qquad w_p(0) = \lambda^{-3/2} g \rho(0) w(0) \]
where $\mu = \mu(\lambda)$ is the minimum eigenvalue.  As in the previous lemma, denote by a dot differentiation with respect to $\lambda$, and for notational convenience we drop the explicit dependence on $\lambda$.  Thus, $\dot{a} = \frac{1+\dot{G}}{2a}$.  Likewise,
\be \frac{\partial}{\partial\lambda} Lw = - \bigg\{\frac{3}{2}a(1+\dot{G})w_p + a^3\dot{w}_p\bigg\}_p. 
\label{ddlambdaLw} \ee
Hence,
\[ L\dot{w} - \frac{3}{2} \bigg\{a(1+\dot{G})w_p\bigg\}_p = \frac{\partial}{\partial\lambda}Lw = \dot{\mu}(a+g\rho_p)w + \mu\bigg(\frac{1+\dot{G}}{2a}\bigg)w + \mu (a+g\rho_p) \dot{w}.\]
We also have from the boundary conditions satisfied by $w$ that $\dot{w}(p_0) = 0$ and
\[ \dot{w}_p(0) = -\frac{3}{2} \lambda^{-5/2} g \rho(0) w(0) + \lambda^{-3/2}g\rho(0)\dot{w}(0) \]
Multiplying \eqref{ddlambdaLw} by $\dot{w}$ and integrating yields
\[ (\dot{w}, Lw) = \mu(\dot{w},(a+g\rho_p)w).\]
Likewise, multiplying \eqref{ddlambdaLw} by $w$ and integrating:
\[ (L\dot{w},w) - \frac{3}{2}\bigg(\bigg\{a(1+\dot{G})w_p\bigg\}_p,w\bigg) = (\dot{\mu}(a+g\rho_p) w,w)+\bigg(\mu\bigg(\frac{1+\dot{G}}{2a}\bigg)w,w\bigg)+(\mu(a+g\rho_p)\dot{w},w). \]
We may integrate by parts to evaluate the second term on the left-hand side in the equation directly above,
\[ - \frac{3}{2}\bigg(\bigg\{a(1+\dot{G})w_p\bigg\}_p,w\bigg) = \frac{3}{2}\bigg(a(1+\dot{G})w_p,w_p\bigg) - \frac{3}{2}a(1+\dot{G})w_pw \bigg|^0\]
On the other hand,
\begin{eqnarray*}
(\dot{w},Lw) - (L\dot{w},w) &=& \int_{p_0}^p \bigg(-\dot{w} \{a^3w_p\}_p +w\{a^3\dot{w}_p\}_p\bigg)dp \\
& = & \bigg(\dot{w}_p a^3w - w_p a^3 \dot{w}\bigg) \bigg|^0. \end{eqnarray*}
Adding the last three equations gives:
\begin{eqnarray*} \frac{3}{2}\bigg(a(1+\dot{G})w_p,w_p\bigg) - \frac{3}{2}a(1+\dot{G})w_pw \bigg|^0 & = & (\dot{\mu}(a+g\rho_p)w,w) \\ 
& & +\bigg(\mu\bigg(\frac{1+\dot{G}}{2a}\bigg)w,w\bigg)+\bigg(\dot{w}_p w- w_p \dot{w}\bigg)a^3\bigg|^0. \end{eqnarray*}
Now as $G(0) = 0$ for all $\lambda$, $\dot{G}(0) = 0$.  Thus the boundary conditions evaluated at $p = 0$ are the following:
\begin{eqnarray*}
a^3(\dot{w}_p w - w_p \dot{w}) + \frac{3}{2}w_p w & = & \lambda^{3/2} \bigg( -\frac{3}{2}\lambda^{-5/2}g\rho w^2 + \lambda^{-3/2} g\rho \dot{w}w\bigg) \\ 
& & - \lambda^{3/2}(\lambda^{-3/2}g\rho\dot{w}w) + \frac{3}{2}\lambda^{1/2}(\lambda^{-3/2}gw^2) \\
& = & 0. \end{eqnarray*}
Thus, dropping the boundary terms, we have
\[ \frac{3}{2}\bigg(a(1+\dot{G})w_p,w_p\bigg) = (\dot{\mu}(a+g\rho_p)w,w)+\bigg(\mu\bigg(\frac{1+\dot{G}}{2a}\bigg)w,w\bigg)\]
But, $a$ (and $a+g\rho_p$) are strictly positive for admissible $\lambda$.  Hence, when $\mu < 0$, we must have $\dot{\mu} > 0$ as claimed. This completes the lemma. \qquad\end{proof} \\

Applying these lemmas, we have, ultimately,  the following result: \\

\begin{lemma} \emph{(Location of $\lambda^*$)} Under the hypotheses of the previous lemmas, the solution $\lambda^*$ of $\mu(\lambda^*) = -1$ (whose existence is given by \textsc{Lemma \ref{eigenvalueproblem}}) is unique.  Moreover, $\lambda^* < \lambda_0$. \label{locationlambdastarlemma} \end{lemma}

\begin{proof}  By continuity of $\mu$ and the preceding lemmas, there exists some $\lambda^*$ with $\mu(\lambda^*) = -1$.  Moreover, since $\mu$ is monotonically increasing when $\mu(\lambda)$ is in any sufficiently small neighborhood of $-1$, this $\lambda^*$ must be unique.

We now prove that the size conditions given are sufficient to ensure that $\lambda_0 \neq \lambda_*$. If $\lambda_0 < -2B_{\textrm{min}} + \epsilon_0$, then we are done.  Likewise, if $\epsilon_0 = 0$, then we are in the constant density case and this result is already known.  So we shall assume $\lambda_0 \geq \epsilon_0 > 0$.  

Fixing $\lambda = \lambda_0$, let $\phi \in \mathscr{S}$ be given by
\[ \phi(p) := \int_{p_0}^p \frac{1+\dot{G}(r; \lambda_0)}{(\lambda_0+G(r; \lambda_0))^{3/2}}dr, \qquad p_0 \leq p \leq 0.\]   
Then, using \eqref{Ydotexpression} we estimate
\[ \int_{p_0}^0 Y_p(p)^{-3} \phi_p(p)^2 dp = \int_{p_0}^0 \frac{(1+\dot{G}(p))^2}{(\lambda_0 + G(p))^{3/2}} dp \leq 2 \dot{Y}(p_0).\]
Now,
\[ \phi(0)^2 = \bigg(\int_{p_0}^0 \frac{1+\dot{G}(p)}{(\lambda_0+G(p))^{3/2}}dp \bigg)^2 = 4 \dot{Y}(p_0)^2.\]
But, for $\lambda_0$, we know from \eqref{defQoflambda} that $\dot{Y}(p_0) = \frac{1}{2g \rho(0)}$.  Combining these estimates we see that the numerator of the Rayleigh quotient $\mathcal{R}(\cdot; \lambda_0)$ is dominated by
\[ -g\rho(0) \phi(0)^2 + \int_{p_0}^0 Y_p(p)^{-3} \phi_p(p)^2 dp \leq - \frac{4\dot{Y}(p_0)}{2 \dot{Y}(p_0)}+2 \dot{Y}(p_0) = 0.\]
Thus $\mu(\lambda_0) \leq 0$.  

Now let $\phi \in \mathcal{S}$ minimize $\mathcal{R}(\phi; \lambda_0)$.  Then multiplying the equation satisfied by $\phi$ and integrating by parts we find that, for $p_0 < p < 0$, 
\[ a^3 \phi \phi_p = -\mu(\lambda_0) \int_{p_0}^p (a+g\rho_p) \phi^2 dr + \int_{p_0}^p a^3 \phi_p^2 dr.\]
Therefore, since we have already shown that $-\mu(\lambda_0) \leq 0$, we see that $a^3 \phi \phi_p$ is a positive and increasing function of $p$.  

If we instead multiply the equation satisfied by $\phi$ by $\phi (1+\dot{G})^{-1}$ and integrate, we arrive at the following identity
\[ -g\rho(0) \phi(0)^2 + \int_{p_0}^0 \frac{a^3}{1+\dot{G}} \phi_p^2 dp - \int_{p_0}^0 \frac{a^3 \dot{G}_p}{(1+\dot{G})^2} \phi \phi_p dp = \mu(\lambda_0)\int_{p_0}^0 \frac{a+g\rho_p}{1+\dot{G}} \phi^2 dp.\]  
Note that we have used the fact that $\dot{G}(0) = 0$.  To show that $\mu(\lambda_0) > -1$, therefore, we need to prove that 
\be -g\rho(0) \phi(0)^2 + \int_{p_0}^0 \frac{a^3}{1+\dot{G}} \phi_p^2 dp - \int_{p_0}^0 \frac{a^3 \dot{G}_p}{(1+\dot{G})^2} \phi \phi_p dp + \int_{p_0}^0 \frac{a+g\rho_p}{1+\dot{G}} \phi^2 dp > 0. \label{mulambda_0bigidentity} \ee

First we observe that, for any $\phi \in \mathscr{S}$,
\begin{eqnarray*}
\phi(0)^2 & = & \bigg( \int_{p_0}^0 \phi_p(p) dp \bigg)^2 \\
& \leq & \bigg( \int_{p_0}^0 \frac{(\lambda_0+G(p))^{3/2}}{1+\dot{G}(p)} \phi_p(p)^2 dp \bigg) \bigg(\int_{p_0}^0 \frac{1+\dot{G}(p)}{(\lambda_0 + G(p))^{3/2}} dp \bigg) \\
& = & \frac{1}{g\rho(0)} \int_{p_0}^0 \frac{(\lambda_0 + G(p))^{3/2}}{1+\dot{G}(p)} \phi_p(p)^2 dp \\ 
& = & \frac{1}{g\rho(0)} \int_{p_0}^0 \frac{a^3}{1+\dot{G}} \phi_p^2 dp.\end{eqnarray*}
Thus the first two terms in \eqref{mulambda_0bigidentity} give a nonnegative contribution.  

Note, however, that the third term is nonpositive, since $\dot{G}_p \geq 0$.  As we have seen, $a^3 \phi \phi_p$ is increasing, therefore we can bound the third term from below by
\begin{eqnarray} - \int_{p_0}^0 \frac{a^3 \dot{G}_p}{(1+\dot{G})^2} \phi \phi_p dp &\geq& -a(0)^3 \phi(0) \phi_p(0) \int_{p_0}^0 \frac{\dot{G}_p}{(1+\dot{G})^2} dp \nonumber \\
& \geq & -4g\rho(0) \phi(0)^2 |p_0| \|\dot{G}_p\|_\infty \\
& \geq & -4g \|\rho_p\|_\infty |p_0| \phi(0)^2. \label{estimate3locationlemma} \end{eqnarray}
On the other hand, integrating the fourth term on the left-hand side of \eqref{mulambda_0bigidentity} by parts gives
\[ \int_{p_0}^0 \frac{a+g\rho_p}{1+\dot{G}} \phi^2 dp = A(0) \phi(0)^2 - 2\int_{p_0}^0 \frac{A}{1+\dot{G}}  \phi \phi_p dp + \int_{p_0}^0 \frac{A  \dot{G}_p}{(1+\dot{G})^2}  \phi^2 dp,\]
where $A(p) := \int_{p_0}^p (a(r; \lambda_0) + g\rho_p(r))dr.$  Note that the final term on the right-hand side above is nonnegative. 

The next task is to estimate these quantities.    In that regard we note
\be
-2\int_{p_0}^0 \frac{a^{-3} A}{1+\dot{G}} dp  \geq  -4|p_0| A(0)\epsilon_0^{-3/2}. \label{estimate2locationlemma} 
\ee  
Then, again exploiting the fact that $a^3 \phi \phi_p$ is increasing, we use \eqref{estimate3locationlemma}-\eqref{estimate2locationlemma} to find
\[ - \int_{p_0}^0 \frac{a^3 \dot{G}_p}{(1+\dot{G})^2} \phi \phi_p dp + \int_{p_0}^0 \frac{a+g\rho_p}{1+\dot{G}} \phi^2 dp \geq \left(|p_0|^{-1} A(0)-4\| \rho_p \|-4A(0)\epsilon_0^{-3/2} g\rho(0) \right) |p_0| \phi(0)^2  .  \]
So long as the quantity in parenthesis on the right-hand side above is nonnegative, we have $\mu(\lambda_0) > -1$.  But, since 
\[  A(0) > (\epsilon_0^{1/2} - g\|\rho_p\|_\infty) |p_0| > \frac{1}{2} |p_0| \epsilon_0^{1/2} ,\]
we have that 
\[ |p_0|^{-1}A(0) - \|\rho_p\|_\infty - 4 A(0) \epsilon_0^{-3/2} g \rho(0) > \frac{1}{2} |p_0|^{-1} A(0) - \|\rho_p\|_\infty > 0, \]
where the last two inequalities follow from the definition of $\epsilon_0$ in \eqref{defepsilon0}. 
\qquad\end{proof}  \\

\subsection{Proof of Local Bifurcation}
All that is left for us now is to verify the hypotheses of the Crandall-Rabinowitz bifurcation theorem presented in \cite{CR}. As in the previous sections let the transformed fluid domain be
\[ R := \{(q,p) : 0 < q < 2\pi, ~p_0 < p < 0 \}, \]
with boundaries
\[ \qquad T := \{ (q,p) \in R : p = 0 \}, \qquad B := \{ (q,p) \in R : p = p_0 \},\]
and define
\[ X := \{ h \in C_{\textrm{per}}^{3+\alpha}(\overline{R}) : h = 0 \textrm{ on } B\}, \qquad
 Y = Y_1 \times Y_2 := C_{\textrm{per}}^{1+\alpha}(\overline{R}) \times C_{\textrm{per}}^{2+\alpha}(T).\]
Let $h(q,p) := H(p) + w(q,p)$.  Then by the full height equation, $w$ must satisfy the following PDE
\begin{eqnarray} (1+w_{qq}^2)(H_{pp}+w_{pp})+w_{qq}(H_p+w_p)^2-2w_q w_{pq} (H_p+w_p) & & \nonumber \\ 
-g(H+w-d(H)-d(w))(H_p+w_p)^3 \rho_p + (H_p+w_p)^3\beta(-p) = 0 & & \textrm{ in } R, \label{interHwheighteq} \\
1+w_q^2 + (H_p+w_p)^2(2g\rho(H+w)-Q) = 0 & & \textrm{ on } T, \label{Hwheighteqboundcond} \end{eqnarray}
together with periodicity in $q$ and vanishing on $B$.  

Conforming to the framework of \cite{CR}, we introduce a nonlinear operator
\[\mathcal{F} = (\mathcal{F}_1, \mathcal{F}_2): (-2B_{\textrm{min}} + \epsilon_0, \infty) \times X \to Y\] 
defined, for $w \in X$, $\lambda > -2B_{\textrm{min}} + \epsilon_0$, by
\begin{eqnarray}
\mathcal{F}_1(\lambda,w) & := & (1+w_{qq}^2)(H_{pp}+w_{pp})+w_{qq}(H_p+w_p)^2-2w_q w_{pq} (H_p+w_p) \nonumber \\ 
& & -g(H+w-d(H)-d(w))(H_p+w_p)^3 \rho_p + (H_p+w_p)^3\beta(-p) \label{defF1} \\
\mathcal{F}_2(\lambda,w) & := & 1+w_q^2 + (H_p+w_p)^2(2g\rho(H+w)-Q). \label{defF2} \end{eqnarray}
Note that by definition of the laminar solution $H(\cdot; \lambda)$,  we have $\mathcal{F}(\lambda,0) \equiv 0$.  For future reference, we evaluate the Fr\'echet derivatives $\mathcal{F}_{1w}$, $\mathcal{F}_{2w}$ at $w = 0$:
\begin{eqnarray}
\mathcal{F}_{1w} &=& \partial_p^2 + H_p^2 \partial_q^2 +3H_p^2 \beta(-p) \partial_p - 3g(H-d(H))H_p^2\rho_p \partial_p - gH_p^3 \rho_p (1-d) \label{F1w} \\ 
\mathcal{F}_{2w} &=& \bigg(2g\rho H_p^2 +2H_p (2g\rho H-Q)\partial_p\bigg)\bigg|_T \nonumber \\
& = & \bigg(2g\rho \lambda^{-1}-2\lambda^{1/2} \partial_p\bigg)\bigg|_T. \label{F2w} \end{eqnarray} \\

In order to show that the eigenvalue at $\lambda^*$ is simple we must characterize the null space and range of the linearized operator $\mathcal{F}_w (\lambda^*, 0)$. This is accomplished in the following two lemmas.\\

\begin{lemma} \emph{(Null Space)} The null space of $\mathcal{F}_w(\lambda^*,0)$ is one-dimensional. \label{nullspacelemma} \end{lemma}

\begin{proof} Fix $\lambda = \lambda^*$.  We have, by our work in \textsc{Lemma \ref{eigenvalueproblem}}, that $M(p)\cos{q} \in \mathcal{N}(\mathcal{F}_w(\lambda^*,0))$, where $\mathcal{N}(\cdot)$ denotes the null space. We need only prove uniqueness to complete the lemma.  

Let $m$ in the null space be given.  Then, since $m$ is even and $2\pi$-periodic in $q$, we may expand
\[ m(q,p) = \sum_{k=0}^\infty \cos{(kq)}m_k(p).\]
The fact that $\mathcal{F}_w(\lambda^*,0)m = 0$ then implies 
\[ \sum_{k=0}^\infty \bigg(\partial_p^2 + H_p^2 \partial_q^2 +3H_p^2 \beta(-p) \partial_p - 3g(H-d(H))H_p^2\rho_p \partial_p - gH_p^3 \rho_p(1-d)\bigg)\cos{(kq)}m_k(p) = 0\]
in $R$, and 
\[ \sum_{k=0}^\infty  \bigg(\bigg(2g\rho H_p^2 +2H_p (2g\rho H-Q)\partial_p\bigg)\cos{(kq)}m_k(p)\bigg)\bigg|_T = 0.\]  
We conclude that each term in the series above must vanish.  But, note that 
\[ d(m_k(p)\cos{(kq)}) = m_k(0) \fint_0^{2\pi} \cos{(kq)} dq = 0,\]
hence, for each $k > 0$, 
\[ \{ -g\rho_p + H_p^{-1}k^2 + \partial_p(H_p^{-3} \partial_p) \} m_k = 0,\qquad p_0 <  p < 0. \]
From the second series and the boundary conditions satisfied by $m$, we have that for $k \geq 0$:
\[ \partial_p m_k - g\rho\lambda^{-3/2} m_k = 0, \qquad \textrm{on } p = 0,\]
along with
\[ m_k = 0, \qquad \textrm{on } p = p_0.\]
 For $k = 1$, then, $m$ must be a constant multiple of $M$, as both satisfy the identical ODE boundary value problem.  Moreover, if $k \geq 2$, we see that $m_k \equiv 0$.  Otherwise, the Rayleigh quotient in the minimization problem will be $-k^2 < -1$, which contradicts the choice of $\lambda^*$.  

The difficulty arises in the case $k=0$, since the $d$-term persists.  Using the fact $d(m_0) = m_0(0)$, we have that $m_0$ satisfies
\[ \{-g\rho_p + \partial_p ( H_p^{-3} \partial_p) \} m_0 = -g \rho_p m_0(0), \qquad p_0 <  p < 0 \]
along with the same boundary conditions as before on $p=0$ and $p=p_0$.  From this we wish to conclude that $m_0 \equiv 0$.  To see why this must be the case we rewrite it in the form:
\[ (a^3 m_0^\prime)^\prime = g\rho^\prime (m_0-m_0(0)) = g\rho^\prime \int_0^p m_0^\prime(r)dr\]
where $p \in (p_0, 0)$ and  prime denotes differentiation with respect to $p$.  Integrating this equation and using the boundary condition at $p = 0$ we find
\begin{eqnarray*}
 a^3 m_0^\prime & = & g\rho(0)m_0(0) + \int_0^p g\rho^\prime(r) \int_0^r m_0^\prime(s)dsdr \\
 & = & g\rho(0)m_0(0)+\int_0^p \bigg(\int_s^p g\rho^\prime(r)dr \bigg) m_0^\prime(s)ds \\
 & = & g\rho(0)m_0(0)+\int_0^p k(p,s)a(s)^3 m_0^\prime(s)ds \end{eqnarray*}
where the kernel, $k$, is given by
\[ k(p,s) := \int_s^p g\rho^\prime(r)a(s)^{-3}dr, \qquad p_0 \leq p < s \leq 0.\]
This is nothing but a Volterra-integral equation of the form
\be \phi(t) = C + \int_0^t k(t,s)\phi(s)ds, \qquad p_0 \leq t < s \leq 0 \label{phiintegraleq} \ee
with $\phi(0) = C$.  As everything here is smooth and we are working on a compact interval, standard theory of integral equations gives a unique solution for $\phi = \phi(p; C)$.   Fix a value of $C$. Then multiplying \eqref{phiintegraleq} by $a^{-3}$ and integrating we find
\begin{eqnarray*}
\int_0^{p_0} \phi(t; C) a(t)^{-3} dt & = & C\int_0^{p_0} a(t)^{-3} dt + \int_0^{p_0} a(t)^{-3} \int_0^t k(t,s)\phi(s; C)dsdt \\
& = &  C\int_0^{p_0} a(t)^{-3} dt + \int_0^{p_0} K(t)\phi(t; C)dt, \end{eqnarray*}
where
\[ K(t) := \int_t^{p_0} k(s,t)a(s)^{-3}ds.\]

Now we note that we must have that $a^3 m_0^\prime = \phi(\cdot; m_0(0))$, where $\phi$ is as above.  Then, by the linearity of $\phi(\cdot; C)$ in $C$, we see that
\[ a^3 m_0^\prime(p) = g \rho(0) m_0(0) \phi(p; 1), \qquad p_0 < p < 0.\]
But observe that $a(p)^{-3} = (\lambda + G(p))^{-3/2}$ and, recalling our argument of the previous section, we have that $\phi(p ; 1) = 1+\dot{G}(p)$, as the kernels of the integral equations \eqref{1plusGdotintkernel} and \eqref{phiintegraleq} (scaled so that $C = 1$) are identical.  Hence integrating the above equation we find
\[ m_0(0) = g\rho(0) m_0(0) \int_{p_0}^0 \frac{1+\dot{G}(p)}{(\lambda+G(p))^{3/2}}dp = 2\dot{Y}(0). \]
Assuming $m_0(0) \neq 0$, we arrive at a contradiction, since by the above identity 
\[ \dot{Q}(\lambda) = 1-2g\rho(0) \dot{Y}(0) = 0.\]
We conclude that $m_0(0) \neq 0$ only when $\lambda = \lambda_0$.  However, \textsc{Lemma \ref{locationlambdastarlemma}} assures us that $\lambda_* < \lambda_0$, so this is impossible.

We have, therefore, shown that $m_0(0)$ must be 0.  This is an immediate consequence of the fact that $m(0) = m^\prime(0) = 0$ while $m$ satisfies a second-order, linear ODE.

 In summary, we have showed that all the modes $k \neq 1$ vanish, thus $m(q,p) = m_1(p) \cos{q}$.  By \textsc{Lemma \ref{eigenvalueproblem}}, see that $m_1$ is a constant multiple of $M$.  Thus the null space is one-diminsional with generator $M(p) \cos{q}$. \qquad\end{proof} \\

\noindent
\begin{lemma} \emph{(Range)} The pair $(\mathcal{A},\mathcal{B})$ belongs to the range of the linear operator $\mathcal{F}_w(\lambda^*,0)$ iff it satisfies the orthogonality condition:
\[ \int \!\!\! \int_{R} \mathcal{A} a^3 \phi^* dqdp + \frac{1}{2} \int_T \mathcal{B} a^2 \phi^* = 0 \]
where $\phi^*$ generates the nulls space $\mathcal{F}_w(\lambda^*,0)$. \label{rangelemma} \end{lemma}

\begin{proof}  Suppose first that $(\mathcal{A},\mathcal{B}) \in Y$ is in the range of $\mathcal{F}_w(\lambda^*,0)$.  Then there exists some $v \in X$ such that $\mathcal{A} = \mathcal{F}_{1w}(\lambda^*,0)v$, and $\mathcal{B} = \mathcal{F}_{2w}(\lambda^*,0)v$.  It follows that
\begin{eqnarray*}
\iint_R \mathcal{A} a^3 \phi^* dqdq & = & \iint_R \big((a^3 v_p)_p + av_{qq}- a^3 g \rho_p v\big) \phi^* dqdp \\
& = & \iint_R \big( (a^3 \phi_p^*)_p + a\phi_{qq}^* - a^3 g \rho_p \phi^*\big) v dqdp + \int_T (a^3 v_p \phi^*-a^3v\phi_p^*) dq \\
& = &  \int_T (a^3 v_p \phi^*-a^3v\phi_p^*) dq, \end{eqnarray*}
as on $B$, both $v$ and $\phi^*$ vanish.  Now, on $T$ we have $a^3 = \lambda^{3/2}$.  Moreover, the ODE satisfied by $\phi^*$ gives $g\rho \phi^* = \lambda^{3/2}\phi_p^*$ on $T$.  Thus,
\begin{eqnarray*}
\frac{1}{2} \int_T \mathcal{B} a^2 \phi^* dq & = & \int_T \big(g\rho v- \lambda^{3/2} v_p \big)\phi^* dq \\
& = & \int_T \big(\lambda^{3/2}\phi_p^* v - \lambda^{3/2}v_p \phi^*\big) dq. \end{eqnarray*}
Combining this expression with the last gives necessity of the orthogonality condition.  

To demonstrate sufficiency we let any pair $(\mathcal{A},\mathcal{B}) \in Y$ be given satisfying the orthogonality condition.   We must show that there exists a solution $u \in X$ to the following:
\begin{equation} \left\{ \begin{array} {ll}
-g\rho_p (u-d(u)) + a^{-2} u_{qq} + a^{-3}\big(a^3 u_p\big)_p = \mathcal{A} & \textrm{in } R, \\
2\big(g\rho a^{-2}u-au_p\big) = \mathcal{B} & \textrm{on } T, \\ 
u = 0 & \textrm{on } B, \end{array} \right. \label{orthoABsuffprob1} \end{equation}
Our methodology here will be to freeze the operator $d$, replacing it with a fixed real number. The resulting problem will be an elliptic PDE, so that we may use Schauder theory to extract solutions. 

We shall do this successively in several stages.  First consider the problem,
\begin{equation} \left\{ \begin{array} {ll}
-\epsilon v^{(\epsilon)} + -g\rho_p v^{(\epsilon)} + a^{-2} v_{qq}^{(\epsilon)} + a^{-3}\big(a^3 v_p^{(\epsilon)}\big)_p = \mathcal{A} & \textrm{in } R, \\
2\big(g\rho a^{-2}v^{(\epsilon)}-av_p^{(\epsilon)}\big) = \mathcal{B} & \textrm{on } T, \\ 
v^{(\epsilon)} = 0 & \textrm{on } B \end{array} \label{orthoABsuffprob2} \right. \end{equation}
where $v^{(\epsilon)}$ is periodic in $q$.  By standard elliptic theory, for each $\epsilon > 0$, problem \eqref{orthoABsuffprob2} will have a unique solution.  

We claim, moreover, that these solutions are bounded in $C_{\textrm{per}}^{1+\alpha}(\overline{R})$.  By contradiction suppose otherwise.  Then there is some sequence of $\epsilon \to 0$ such that $\|v^{(\epsilon)}\|_{C^{1+\alpha}(\overline{R})} \to \infty$.  For each $\epsilon$, put $u^{(\epsilon)} := v^{(\epsilon)}/ \|v^{(\epsilon)}\|_{C^{1+\alpha}(\overline{R})}$.  Thus $u^{(\epsilon)}$ has unit norm for all $\epsilon$.  Given that $v^{(\epsilon)}$ solves \eqref{orthoABsuffprob2}, we have additionally
\[ -g\rho_p u^{(\epsilon)} + a^{-2} u_{qq}^{(\epsilon)} + a^{-3}\big(a^3 u_p^{(\epsilon)}\big)_p \longrightarrow 0 \qquad \textrm{in } C_{\textrm{per}}^{1+\alpha}(\overline{R}) \]
and
\[ g\rho a^{-2}u^{(\epsilon)}-au_p^{(\epsilon)} \longrightarrow 0 \qquad \textrm{in } C_{\textrm{per}}^{1+\alpha}(T).\]
Observe that \eqref{orthoABsuffprob2} is uniformly elliptic.  Applying Schauder estimates ensures that the sequence $\{u^{(\epsilon)} \}$ is uniformly bounded in $C_{\textrm{per}}^{2+\alpha}(\overline{R})$.  So by compactness we have a subsequence converging strongly to some $u \in C_{\textrm{per}}^2(\overline{R})$.  By an argument identical to the previous lemma, we can show  that $u$ is in the null space of $\mathcal{F}_w(\lambda^*,0)$, and hence a multiple of $\phi^*$.  This follows from the fact that problem immediately above is $\mathcal{F}_{1w}$ without the stratification term.  As we have seen, when we expand any element in the null space, this term will drop out in all modes $k \geq 1$.  Similarly, for the $k = 0$ term, we may simply bypass the integral equation argument (as we are, so-to-speak, already given that the right-hand side of the interior equation is zero), and use the same Rayleigh quotient manipulations to conclude $u_0 \equiv 0$.  So indeed, $u$ is a constant multiple of $\phi^*$.

 But by definition of $v^{(\epsilon)}$ we have,
\[ \iint_R \mathcal{A} a^3 \phi^* dqdp = \iint_R a^3\phi^*\big(-\epsilon v^{(\epsilon)}+ \mathcal{F}_{1w}(\lambda^*,0)v^{(\epsilon)}\big)dqdp. \]
Applying the exact same manipulations on the second term in the integrand above as we did in proving necessity yields
\[ 0 = -\epsilon \iint_R a^3 v^{(\epsilon)} \phi^* dqdp = \iint_R a^3 u^{(\epsilon)} \phi^* dqdp \Longrightarrow 0 = \iint_R a^3 u \phi^* dqdp. \]
However, we have found that $u$ is in the linear span of $\phi^*$, so the equation above implies $u \equiv 0$.  This contradicts the fact that $\|u\|_{C^{1+\alpha}(\overline{R})} = 1$.  Hence $\{v^{(\epsilon)}\}$ is bounded.  

It now follows that $\{v^{(\epsilon)}\}$ has a strongly convergent subsequence in $C_{\textrm{per}}^1(\overline{R})$.  The limit, denote it $v$, will satisfy \eqref{orthoABsuffprob2} with $\epsilon = 0$ in the sense of distributions.  But, again appealing to standard elliptic regularity theory, this implies that $v \in X$.  

We have thus shown that for each $(\mathcal{A},\mathcal{B})$ satisfying the orthogonality condition, we may find a unique solution in $X$ to the following problem:
\begin{equation} \left\{ \begin{array} {ll}
-g\rho_p u + a^{-2} u_{qq} + a^{-3}\big(a^3 u_p\big)_p = \mathcal{A} & \textrm{in } R, \\
2\big(g\rho a^{-2}u-au_p\big) = \mathcal{B} & \textrm{on } T, \\ 
u = 0 & \textrm{on } B. \end{array} \right. \label{orthoABsuffprob3} \end{equation}

Now to return to equation \eqref{orthoABsuffprob1} we make the following observation:  if the pair $(\mathcal{A},\mathcal{B})$ satisfies the orthogonality condition, then so does $(\mathcal{A}-f(p), \mathcal{B})$, where $f$ is any smooth function.  To see why this is the case recall that we have proved that $\phi^*$ is of the form $\phi^*(q,p) = M(p)\cos{q}$. Therefore
\begin{eqnarray*}
\int \!\!\! \int_{R} (\mathcal{A}-f) a^3 \phi^* dqdp & = & \int \!\!\! \int_{R} \mathcal{A} a^3 \phi^* dq dp - \int_{p_0}^0 f(p) a(p)^3 M(p) dp \int_{0}^{2\pi} \cos{q} dq \\
& = & \int \!\!\! \int_{R} \mathcal{A} a^3 \phi^* dq dp \\
& = & \frac{1}{2} \int_T \mathcal{B} a^2 \phi^*, \end{eqnarray*}
so, indeed, the orthogonality condition holds for $(\mathcal{A}-f(p), \mathcal{B})$.

Fix $\sigma \in \mathbb{R}$.  In light of our previous observation, and our proof of the uniqueness and existence of solutions to \eqref{orthoABsuffprob3} in $X$, it follows that there exists a unique solution $u^{(\sigma)}$ of
\begin{equation} \left\{ \begin{array} {ll}
-g\rho_p u^{(\sigma)} + a^{-2} u_{qq}^{(\sigma)} + a^{-3}\big(a^3 u_p^{(\sigma)}\big)_p = \mathcal{A} -g\rho_p \sigma & \textrm{in } R, \\
2\big(g\rho a^{-2}u^{(\sigma)}-au_p^{(\sigma)}\big) = \mathcal{B} & \textrm{on } T, \\ 
u^{(\sigma)} = 0 & \textrm{on } B. \end{array} \right. \label{orthoABsuffprob4} \end{equation}
We may therefore define a mapping $\Psi:\mathbb{R} \to \mathbb{R}$ by $\Psi(\sigma) := d(u^{(\sigma)})$, where $u^{(\sigma)}$ solves \eqref{orthoABsuffprob4}.  Suppose that $\Psi$ has a fixed point $\sigma$.  Then we may simply substitute $d(u^{(\sigma)})$ for $\sigma$ in \eqref{orthoABsuffprob4} to find that $u^{(\sigma)}$ is the sought after solution to \eqref{orthoABsuffprob1}.  Thus it suffices to prove that $\Psi$ has a (unique) fixed point.

Let $\sigma$, $\tau$ be given with $\sigma \neq \tau$ and let $u^{(\sigma)}$, $u^{(\tau)}$ solve \eqref{orthoABsuffprob4} for $\sigma,~\tau$ respectively.  Then subtracting the equations satisfied by $u^{(\sigma)}$ and $u^{(\tau)}$ and dividing by $\sigma - \tau$ we arrive a solution to the following problem:
\be \left\{ \begin{array} {ll}
-g\rho_p v + a^{-2} v_{qq} + a^{-3}\left(a^3 v_p \right)_p = -g\rho_p & \textrm{in } R, \\
2\left(g\rho a^{-2}v-av_p\right) = 0 & \textrm{on } T, \\ 
v = 0 & \textrm{on } B. \end{array} \label{orthoABsuffprob5} \right. \ee
By applying the Schauder estimates for Dirchlet and oblique boundary conditions, moreover, we have that the solution, $v$, of this equation is in $X$.

By contradiction, assume that $d(v) = 1$. Then, by virtue of the fact $v$ solves \eqref{orthoABsuffprob5}, we have that $v \in \mathcal{N}(\mathcal{F}_w(\lambda^*,0))$.  It follows from our previous lemma that $v$ is a constant multiple of $\phi^*$.  But this is a contradiction, since 
\[d(\phi^*) = M(0)\fint_0^{2\pi} \cos{q}dq = 0.\]
Thus we are assured $d(v) \neq 1$.  

Fix $\sigma,~h \in \mathbb{R}$.  Examining \eqref{orthoABsuffprob4}, we readily see
\[  \frac{\Psi(\sigma+h)-\Psi(\sigma)}{h} =  \fint_T \left( \frac{u^{(\sigma+h)}-u^{(\sigma)}}{h}\right)dq = d(v).  \]
 Hence $\Psi$ is differentiable everywhere and $\Psi^\prime \equiv d(v)$.  Therefore the function $\sigma \mapsto \Psi(\sigma)-\sigma$ is monotonic with derivative uniformly equal to $d(v) -1 \neq 0$.  So for some $\sigma^*$ we have $0 = \psi(\sigma^*)-\sigma^*$.  This proves the existence of a fixed point of $\Psi$, and thereby a solution to \eqref{orthoABsuffprob1}. The lemma follows. \qquad\end{proof} \\

Finally, we must ensure that the so-called transversality or crossing condition holds.  \\

\noindent
\begin{lemma} \emph{(Technical Condition)} If $\phi^*$ generates the null space of $\mathcal{F}(\lambda^*,0)$, then $\mathcal{F}_{w\lambda}(\lambda^*,0)\phi^* \notin \mathcal{R}(\mathcal{F}_w(\lambda^*,0))$. \label{technicallemma} \end{lemma}

\begin{proof}  First we calculate the mixed Fr\'echet derivatives of $\mathcal{F}$ at $\lambda = \lambda^*$, $w = 0$:
\begin{eqnarray*}
\mathcal{F}_{1\lambda w}(\lambda^*,0) & = & -(1+\dot{G})a^{-4} \partial_q^2 - 3(1+\dot{G})a^{-4} \beta(-p) \partial_p - 3g \dot{Y}a^{-2}\rho_p \partial_p \\
& & +~3gY (1+\dot{G})a^{-4} \rho_p \partial_p + \frac{3}{2} g(1+\dot{G})a^{-5} \rho_p(1-d)\\
\mathcal{F}_{2\lambda w}(\lambda^*,0) & = & \bigg(-2g\rho \lambda^{-2} - \lambda^{-1/2} \partial_p \bigg)\bigg|_T .\end{eqnarray*}
By the previous lemma, it suffices to show that the pair $(\mathcal{F}_{1\lambda w}(\lambda^*,0)\phi^*, \mathcal{F}_{2\lambda w}(\lambda^*,0) \phi^*)$ does not satisfy the orthogonality condition.  Equivalently, if we put 
\begin{eqnarray*}
\Xi & := & \iint_{R}\bigg( (1+\dot{G})a^{-1} (\phi^*)^2 -3(1+\dot{G})a^{-1}\beta(-p)\phi^*\phi_p^* \\
& & -~3g\dot{Y} a \rho_p \phi^* \phi_p^* + 3gY(1+\dot{G})a^{-1} \rho_p \phi^* \phi_p^* + \frac{3}{2} g(1+\dot{G})a^{-2} \rho_p (\phi^*)^2 \bigg)dqdp \\
& & +~\int_T \bigg(-2g\rho a^{-2} (\phi^*)^2 - \frac{1}{2}a \phi^* \phi_p^* \bigg) dq, \end{eqnarray*}
our lemma will be proven if we can show $\Xi \neq 0$ (again, because $d(\phi^*) = 0$).  We will demonstrate that, in fact, $\Xi < 0$.  To keep our notation concise, let $\Xi = \Xi_1+ \ldots + \Xi_7$, where $\Xi_i$ denotes the $i$-th term in the sum above.  It is clear \emph{a priori} that $\Xi_1 > 0$, while $\Xi_5,\Xi_6 < 0$.

Consider the boundary terms $\Xi_6$ and $\Xi_7$.  On $T$, we have $a = \lambda^{1/2}$, and $\phi^*$ satisfies $g\rho \phi^* = \lambda^{3/2}\phi_p^*$.  Therefore,
\be \Xi_7 = - \frac{1}{2}\int_T \lambda^{1/2} \phi^* \phi_p^* dq = -\frac{1}{2}\int_T \lambda^{-1} g \rho (\phi^*)^2 dq = \frac{1}{4} \Xi_6. \label{technicalcond1} \ee 
Now, by the ODE satisfied by $\phi^*$ in $R$, we have
\[ (a+g\rho_p)\phi^* = (a^3 \phi_p^*)_p = 3a(\beta(-p) - gY \rho_p) \phi_p^* + a^3 \phi_{pp}, \]
whence,
\[ 3a^{-1} g Y \rho_p \phi_p^* = a \phi_{pp}^* - (g\rho_p + a)a^{-2} \phi^* +3a^{-1}\beta(-p)\phi_p^*.\]
Substituting this expression into $\Xi_4$ yields the following,
\begin{eqnarray}
\Xi_4 & = & \iint_R 3gY(1+\dot{G})a^{-1} \rho_p \phi^* \phi_p^* dqdp \nonumber \\
& = & \iint_R \bigg( a(1+\dot{G})\phi^* \phi_{pp}^* - (g\rho_p +a)a^{-2}(1+\dot{G})(\phi^*)^2\bigg)dqdp - \Xi_2 \nonumber \\
& = &  \iint_R \bigg( a(1+\dot{G})\phi^* \phi_{pp}^* - g\rho_p a^{-2}(1+\dot{G})(\phi^*)^2\bigg)dqdp - \Xi_1 - \Xi_2. \label{technicalcond2} \end{eqnarray}
Next, consider the quantity
\begin{eqnarray}
\Xi_2+\Xi_3+\Xi_4 & = & \iint_R \bigg( a(1+\dot{G})\phi^* \phi_{pp}^* - g\rho_p a^{-2}(1+\dot{G})(\phi^*)^2 \nonumber \\
& &  - 3g\dot{Y} a \rho_p \phi^* \phi_p^* \bigg)dqdp - \Xi_1. \label{technicalcond3} \end{eqnarray}
Calculating that $\dot{G}_p = -2g\dot{Y}\rho_p$, and $a_p = a^{-1}(\beta(-p)-gY\rho_p)$, we integrate the first term by parts to find:
\begin{eqnarray}
\Xi_2+\Xi_3+\Xi_4 & = & \iint_R \bigg( \big(a^{-1}g Y(1+\dot{G})\rho_p - g a \dot{Y}\rho_p\big)\phi^* \phi_p^* - a^{-1} \beta(-p)(1+\dot{G})\phi^*\phi_p^* \nonumber \\ 
& & -~a(1+\dot{G})(\phi_p^*)^2 - ga^{-2}(1+\dot{G})\rho_p (\phi^*)^2\bigg)dqdp + \int_T a(1+\dot{G})\phi^* \phi_p^* dq - \Xi_1 \nonumber \\
& = & \iint_R \bigg(-g a^{-2} (1+\dot{G})\rho_p (\phi^*)^2 - a(1+\dot{G}) (\phi_p^*)^2\bigg)dqdp \nonumber \\ 
& & + \frac{1}{3} \left(\Xi_2+\Xi_3+\Xi_4\right) - \Xi_1 - 2\Xi_7, \label{technicalcond4} \end{eqnarray}
where we have used the familiar fact that $\dot{G}(0) = 0$.  

Now, simply combining identities \eqref{technicalcond1}, \eqref{technicalcond2}, \eqref{technicalcond3}, and \eqref{technicalcond4} gives
\begin{eqnarray*}
\Xi & = & \Xi_1+ \Xi_2+\Xi_3+\Xi_4+\Xi_5+\frac{5}{4}\Xi_6 \\
& = & \Xi_1+ \frac{3}{2}\bigg(\iint_R \bigg(-g a^{-2}(1+\dot{G}) \rho_p (\phi^*)^2 - a(1+\dot{G}) (\phi_p^*)^2\bigg)dqdp - \Xi_1 - 2\Xi_7\bigg)+\Xi_5+\frac{5}{4}\Xi_6 \\
& = & -\frac{1}{2}\Xi_1 - \frac{3}{2} \iint_R a(1+\dot{G}) (\phi_p^*)^2 dqdp + \frac{1}{2}\Xi_6.\end{eqnarray*}
Thus $\Xi < 0$, and the claim is proven. \qquad\end{proof}  \\

We are now ready to prove the local bifurcation theorem.  \\

\noindent
\puffthm{\ref{localbifurcationtheorem}}  By definition of the laminar solutions, $H(\cdot; \lambda)$, we have $\mathcal{F}(\lambda, 0) = 0$ for all $\lambda \geq -2B_{\textrm{min}} + \epsilon_0$.  Moreover, by the regularity assumptions on $\rho$ and $\beta$ as well as the definition of $X$, we see that $\mathcal{F}_\lambda$, $\mathcal{F}_w$, $\mathcal{F}_{\lambda w}$ and $\mathcal{F}_{ww}$ exist and are continuous.  Applying \textsc{Lemma \ref{nullspacelemma}} and \textsc{Lemma \ref{rangelemma}}, moreover, we see that $\mathcal{N}(F_w(\lambda^*,0))$ and $Y \setminus \mathcal{R}(\mathcal{F}_w(\lambda^*,0)$ are both one-dimensional with the former generated by $\phi^*$.  Finally, \textsc{Lemma \ref{technicallemma}} shows $\mathcal{F}_{w \lambda}(\lambda^*,0) \phi^*$ is not contained in $\mathcal{R}(\mathcal{F}_w(\lambda^*,0))$.  

Thus we have satisfied all the hypotheses of the Crandall-Rabinowitz theorem on bifurcation from a simple eigenvalue (cf. \cite{CR} \textsc{Theorem 2}).  This allows us to conclude that there exists a $C^1$ local bifurcation curve
\[ \mathcal{C}_0^{\prime} = \{\lambda(s), w(s): |s| < \epsilon \},\]
for $\epsilon > 0$ sufficiently small, such that $(\lambda(0),w(0)) = (\lambda^*,0)$ and
\[ \{ (\lambda, w) \in \mathcal{U} : w \nequiv 0, ~\mathcal{F}(\lambda,w) = 0 \} = \mathcal{C}_0^{\prime} \]
where $\mathcal{U}$ is some neighborhood of $(\lambda^*,0) \in [-2B_{\textrm{min}}+\epsilon_0,\infty) \times X$. Moreover, for $|s| < \epsilon$, 
\be w(s) = s\phi^* + o(s), \qquad \textrm{in } X. \label{wissphios} \ee
Furthermore, since $h = H + w$ and $H_p > 0$ in $\overline{R}$, we may restrict our attention to a smaller $C^1$-curve $\mathcal{C}_{\textrm{loc}}^{\prime} \subset \mathcal{C}_0^{\prime}$ containing $(\lambda^*,0)$ and along which $h_p > 0$ in $\overline{R}$.  Then by \textsc{Lemma \ref{equivalencelemma}}, we have the existence of a $C^1$-curve, $\mathcal{C}_{\textrm{loc}}$, of solutions to \eqref{incompress}-\eqref{bedcond} \qquad$\square$  \\

\section{Global Bifurcation}

We now prove that we can continue the curve $C_{\textrm{loc}}^\prime$ whose existence was established in \textsc{Theorem \ref{localbifurcationtheorem}}.   Continuing with our notation from the previous section, denote
 \[ X := \{ h \in C_{\textrm{per}}^{3+\alpha}(\overline{R}) : h = 0 \textrm{ on } B\}, \qquad
 Y = Y_1 \times Y_2 := C_{\textrm{per}}^{1+\alpha}(\overline{R}) \times C_{\textrm{per}}^{2+\alpha}(T),\]
 where the subscript ``per'' indicates $2\pi$-periodicity and evenness in $q$, $R$ is the rectangle $(0,2\pi)\times (p_0,0)$ and $T = (0,2\pi)\times \{p = 0\}$.  Next define
 \[ \mathcal{G} = (\mathcal{G}_1, \mathcal{G}_2) : \mathbb{R}\times X \to Y \]
 by
 \begin{eqnarray}
 \mathcal{G}_1(h) & := & (1+ h_q^2)h_{pp} + h_{qq} h_p^2 - 2 h_q h_p h_{pq} - g(h-d(h)) h_p^3 \rho_p + h_p^3 \beta(-p) \label{defG1} \\
 \mathcal{G}_2(Q,h) & := & \bigg(1+h_q^2 +h_p^2(2g\rho h - Q)\bigg)\bigg|_{p = 0}. \label{defG2} \end{eqnarray}
Recall that here $d:X \to \mathbb{R}$ denotes the linear operator mapping an element of $X$ to its average value on $T$.  We remark, also, that the laminar flow solutions $(H(\lambda), \lambda)$ found in \textsc{Lemma \ref{laminarflowlemma}} satisfy $\mathcal{G}(H(\lambda), Q(\lambda)) = 0$, for all $\lambda \geq -2B_{\min{}} +\epsilon_0$.

In order for our arguments to have any traction we will need to make strong use of \emph{a priori} estimates.  It will often prove convenient to consider uniformly elliptic differential operators that approximate $\mathcal{G}$, or more specifically, $\mathcal{G}_1$.  For any $\sigma \in \mathbb{R}$ we therefore define $\mathcal{G}^{(\sigma)} : \mathbb{R} \times X \to Y$ by
\begin{eqnarray}
\mathcal{G}_1^{(\sigma)}(h) & := & (1+ h_q^2)h_{pp} + h_{qq} h_p^2 - 2h_q h_p h_{pq} - g(h-\sigma) h_p^3 \rho_p + h_p^3 \beta(-p) \label{defG1sigma} \\
 \mathcal{G}_2^{(\sigma)}(Q,h) & := & \mathcal{G}_2(Q,h) = \bigg(1+h_q^2 +h_p^2(2g\rho h - Q)\bigg)\bigg|_{p = 0}. \label{defG2sigma} \end{eqnarray}
That is, we replace the $d$-term of $\mathcal{G}_1$ with the real number $\sigma$.  As will seen later, in the function space we shall work in, this defines a uniformly elliptic differential operator for each $\sigma \in \mathbb{R}$.  This is essential, for it allows us to exploit Schauder theory in order to prove the compactness properties we need.

Let $\delta > 0$ be given.  In order to ensure the uniformity of the ellipticity of $\mathcal{G}_1^{(\sigma)}$ and obliqueness of $\mathcal{G}_2^{(\sigma)}$ we will work in the set 
\[ \mathcal{O}_{\delta} := \bigg\{ (Q,h) \in \mathbb{R}\times X : h_p > \delta \textrm{ in } \overline{R},~ h < \frac{Q-\delta}{2g\rho} \textrm{ on } T\bigg\}.\]
Likewise, we put
\[S_\delta := \textrm{closure in } \mathbb{R}\times X \textrm{ of } \{(Q,h) \in \mathcal{O}_{\delta} : \mathcal{G}(Q,h) = 0, ~ h_q \nequiv 0\},\]
 and let $C_\delta^\prime$ be the component of $S_\delta$ that contains the point $(Q^*, H^*)$, where $Q^* := Q(\lambda^*)$, $H^* := H(\cdot; \lambda^*)$.  Thus $C_\delta^\prime$ contains the local curve $C_{\textrm{loc}}^\prime$.  
 
 For later reference we compute the Fr\'echet derivatives of $\mathcal{G}$ and $\mathcal{G}^{(\sigma)}$.  
 
 \begin{eqnarray}
 \mathcal{G}_{1h}(h) & = & 2 h_q h_{pp} \partial_q + (1+ h_q^2) \partial_p^2 + 2h_p h_{qq} \partial_p + h_p^2 \partial_q^2 \nonumber \\ 
 & & - 2(h_{pq}h_q \partial_p + h_{pq}h_p \partial_q + h_p h_q \partial_p \partial_q) \nonumber  \\
 & & - 3g\rho_p (h-d(h)) h_p^2 \partial_p -g\rho_p h_p^3(1-d) + 3h_p^2 \beta(-p) \partial_p \label{G1h} \\
 \mathcal{G}_{2h} (Q,h) & = & \bigg(2h_q \partial_q + 2g\rho h_p^2 + 2h_p(2g\rho h - Q)\partial_p \bigg)\bigg|_T \label{G2h} \\
 \mathcal{G}_{1h}^{(\sigma)} (h) & = & \mathcal{G}_{1h}(h) + 3g(d(h)-\sigma) h_p^3 \rho_p \partial_p + g\rho_p h_p^3 d \nonumber \\
 & = & 2 h_q h_{pp} \partial_q + (1+h_q^2) \partial_p^2 + 2h_p h_{qq} \partial_p + h_p^2 \partial_q^2 \nonumber \\
 & & - 2(h_{pq}h_q \partial_p + h_{pq}h_p \partial_q + h_p h_q \partial_p \partial_q) \nonumber \\
 & & - 3g\rho_p (h-\sigma) h_p^2 \partial_p -g\rho_p h_p^3 + 3h_p^2 \beta(-p) \partial_p \label{G1hsigma} \\
 \mathcal{G}_{2h}^{(\sigma)}(Q,h) & = & \mathcal{G}_{2h}(Q,h). \label{G2hsigma}
 \end{eqnarray}

In the spirit of Rabinowitz \cite{R1}, we shall apply a degree theoretic argument to get a global continuation theorem in the form of an alternative result:  \\
\begin{theorem} \emph{(Global Bifurcation)} Let $\delta > 0$ be given.  One of the following must hold: 
\begin{romannum}
 \item $\mathcal{C}_\delta^\prime$ is unbounded in $\mathbb{R} \times X$.
 \item $\mathcal{C}_\delta^\prime$ contains another trivial point $(Q(\lambda), H(\lambda)) \in \mathcal{T}$, with $\lambda \neq \lambda^*$.
 \item $\mathcal{C}_\delta^\prime$ contains a point $(Q,h) \in \partial \mathcal{O}_{\delta}$. \end{romannum} 
\label{globalbifurcationtheorem}
\end{theorem}

\noindent \\ Observe, however, that $\mathcal{G}$ is not a compact perturbation of identity, which rules out using the classical Leray-Schauder degree (see, e.g. \cite{FG}).   In light of the nonlinear boundary operator $\mathcal{G}_2$, we instead employ a variant degree theory developed by Healey and Simpson (cf. \cite{HS}).  In order to do so, we must first establish two lemmas on the topological properties of the map $\mathcal{G}$.  The structure of the arguments in both cases will be to begin by using elliptic estimates on $\mathcal{G}^{(\sigma)}$ (which is uniformly elliptic in $\mathcal{O}_\delta$).  Then, taking advantage of the fact that the operator $d$ can be easily estimated in $X$, we reinsert the stratification term and are able to make conclusions about $\mathcal{G}$.  

We emphasize that these lemmas are key.  Once $\mathcal{G}$ has been shown to be admissible in the sense of Healey-Simpson degree (cf. \textsc{Definition 4.10} of \cite{HS}), the proof of \textsc{Theorem 4.1} is identical to that of the homogeneous case in \cite{CS1}.  Indeed, the technical obstacles presented by the stratification term are completely confined to the following three proofs. \\

\begin{lemma} \emph{(Proper Map)} Suppose $K$ is a compact subset of $Y$ and $D$ is a closed, bounded set in $\overline{\mathcal{O}_\delta}$, then $\mathcal{G}^{-1}(K) \cap D$ is compact in $\mathbb{R} \times X$. \label{propermaplemma} \end{lemma}

\begin{proof} First we prove that for each $Q$, $\sigma \in \mathbb{R}$, the operator $h\mapsto \mathcal{G}^{(\sigma)}(Q,h)$ is uniformly elliptic and oblique in $\overline{{O}_{\delta}}$.  The former follows from the fact that the coefficients of the higher order terms in $\mathcal{G}_{1h}$ satisfy:
\[ 4(1+h_q^2)h_p^2 -4 h_q^2 h_p^2 = 4h_p^2  >  \delta^2.\] 
Notice that the bound here does not in any way depend on $\sigma$ or $Q$.  Similarly, we have that the boundary operator $h \mapsto \mathcal{G}_2^{(\sigma)}(h)$ is uniformly oblique as the coefficient of $h_p$ satisfies 
\[ |(2g\rho h -Q)h_p | \geq \delta^2 \qquad \textrm{on } T\]
for $(Q,h) \in \overline{\mathcal{O}_{\delta}}$.

Let $\{(f_j, g_j)\}$ be a convergent sequence in $Y = Y_1 \times Y_2$ and assume that for each $j \geq 1$, $(f_j,g_j) = \mathcal{G}(Q_j, h_j)$ for some $(Q_j,h_j) \in \overline{\mathcal{O}_{\delta}}$ with $\{h_j\}$ bounded in $C_{\textrm{per}}^{3+\alpha}(\overline{R})$, $Q_j$ bounded in $\mathbb{R}$.  We wish to show that there exists a subsequence of $\{(Q_j,h_j)\}$ convergent in $\mathbb{R}\times X$.  

Towards that end we denote $\theta_j := \partial_q h_j$, for $j \geq 1$.  Then, differentiating the relation between $(Q_j, h_j)$ and $(f_j, g_j)$, we find by \eqref{defG1}
\begin{eqnarray*}
\partial_q f_j & = & \partial_q \mathcal{G}_1(h_j) \\
& = & (1+ (\partial_q h_j)^2)\partial_p^2 \theta_j + (\partial_q(1+(\partial_q h_j)^2))\partial_p^2 h_j + (\partial_p h_j)^2 \partial_q^2 \theta_j + (\partial_q^2 h_j)\partial_q(\partial_p h_j)^2 \\
& & - 2(\partial_q h_j)(\partial_p h_j) \partial_p\partial_q \theta_j -2(\partial_q^2 h_j) (\partial_p h_j)( \partial_p\partial_q h_j) - 2(\partial_q h_j) (\partial_p\partial_q h_j)^2 \\
& &  - g (\partial_p h_j)^3 \rho_p \theta_j -g(h_j - d(h_j))\partial_q (\partial_p h_j)^3 + \partial_q( \partial_p h_j)^3 \beta(-p), \qquad \textrm{in } R,~\forall j \geq 1.
\end{eqnarray*}
Grouping terms above we may rewrite this in the form
\begin{eqnarray}
 (1+ (\partial_q h_j)^2) \partial_p^2 \theta_j + (\partial_p h_j)^2 \partial_q^2 \theta_j 
 - 2(\partial_q h_j) (\partial_p h_j) \partial_q\partial_p \theta_j  = & &\nonumber \\ 
 \partial_q f_j + F(\partial_q h_j, \partial_p \partial_q h_j, \partial_q^2 h_j, \partial_p h_j, \partial_p^2 h_j, d(h_j),\rho_p), & & \textrm{ in } R \label{propermapReq} \end{eqnarray}
where $F$ is the cubic polynomial dictated by the previous expression.  Notice that we have picked up a dependence on $d(h_j)$ and $\rho_p$ that was not there in the constant density case presented in \cite{CS1}.  This will not be an issue, however, as these terms can still be bounded in $X$.  

We may likewise differentiate the equation on $T$ to discover that, for each $j \geq 1$,
\begin{eqnarray*}
\partial_q g_j & = & \partial_q \mathcal{G}_2 (Q_j, h_j) \\
& = & 2(\partial_q h_j)\partial_q \theta_j + 2(\partial_p h_j)(2g\rho h_j -Q_j) \partial_p \theta_j + (\partial_p h_j)^2(2g\rho \partial_q h_j), \qquad \textrm{on } T. \end{eqnarray*}
Equivalently, 
\be 2(\partial_q h_j) \partial_q \theta_j + 2(\partial_p h_j)(2g \rho h_j - Q_j) \partial_p \theta_j = \partial_q g_j + G(\partial_q h_j, \partial_p h_j, \rho), \qquad \textrm{on } T,~ \forall j \geq 1 \label{propermapTeq} \ee
where $G$ is the quadratic polynomial determined by the previous equation.  Finally we observe that 
\be \theta_j = 0, \qquad \textrm{on } B. \label{propermapBeq} \ee  

Since $\{(f_j, g_j)\}$ is convergent in $C_{\textrm{per}}^{1+\alpha}(\overline{R}) \times C_{\textrm{per}}^{2+\alpha}(T)$ we have that the sequences $\{\partial_q f_j\}$,  $\{\partial_q g_j\}$ are Cauchy in $C_{\textrm{per}}^\alpha(\overline{R})$ and $C_{\textrm{per}}^{1+\alpha}(T)$ respectively.  Moreover, by the boundedness of $\{h_j\}$ in $C_{\textrm{per}}^{3+\alpha}(\overline{R})$ (and thereby the boundedness of $d(h_j)$ in this space) we have that $F$ is uniformly bounded in $C_{\textrm{per}}^{1+\alpha}(\overline{R})$, and $G$ is uniformly bounded in $C_{\textrm{per}}^{2+\alpha}(T)$.  It follows that each of these, viewed as a sequence in $j$, is Cauchy. Then, by the compactness of the embeddings, the right-hand sides of equations \eqref{propermapReq} and \eqref{propermapTeq}  are pre-compact in $C_{\textrm{per}}^\alpha(\overline{R})$ and $C_{\textrm{per}}^{1+\alpha}(T)$, respectively.  Possibly passing to a subsequence, we may take both to be convergent in these spaces.  

We now consider differences $\theta_j - \theta_k$,  for $j,k \geq 1$.  By \eqref{propermapReq} we have
\[ (1+(\partial_q h_j)^2) \partial_p^2 (\theta_j-\theta_k) + (\partial_p h_j)^2 \partial_q^2 (\theta_j-\theta_k) 
 - 2(\partial_q h_j) (\partial_p h_j) \partial_q\partial_p (\theta_j-\theta_k)  = F_{jk}, \qquad \textrm{in } R \]
where by our arguments in the previous paragraph we know $F_{jk} \to 0$ in $C^{\alpha}(\overline{R})$.  Similarly, from \eqref{propermapBeq} have that $\theta_j - \theta_k$ vanishes on the bottom and on the top \eqref{propermapTeq} tells us
 \[ 2(\partial_q h_j) \partial_q (\theta_j-\theta_k) + 2(\partial_p h_j)(2g \rho h_j - Q_j) \partial_p (\theta_j-\theta_k) = G_{jk}, \qquad \textrm{on } T.\]
Here $G_{jk} \to 0$ in $C^{1+\alpha}(T)$, again by the considerations of the preceding paragraph.  We now apply the mixed-boundary condition Schauder estimates to the differences $\theta_j - \theta_k$ to deduce that $\|\theta_j-\theta_k\|_{C^{2+\alpha}(\overline{R})} \to 0$ as $j,k \to \infty$.  We have shown, therefore, that all third derivatives of $h_j$ are Cauchy, except possibly for $\partial_p^3 h_j$.  To demonstrate the same holds for $\{\partial_p^3 h_j\}$, we use the PDE to express $\partial_p^2 h_j$ in terms of the derivatives of order less than or equal to two: 
 \begin{eqnarray*}
 \partial_p^2 h_j & = &(1+ (\partial_q h_j)^2)^{-1} \bigg( f_j -  \partial_q^2 h_j (\partial_p h_j)^2 + 2(\partial_q h_j )(\partial_p  h_j )( \partial_p \partial _q h_j) \\
 & & + g(h_j-\sigma) (\partial_p h_j)^3 \rho_p + (\partial_p h_j)^3 \beta(-p)\bigg), \qquad \textrm{in } R,~j \geq 1.\end{eqnarray*}
But we have seen that the right-hand side is Cauchy in $C^{1+\alpha}(\overline{R})$, hence $\{\partial_p^3 h_j\}$ is also Cauchy in $C_{\textrm{per}}^\alpha(\overline{R})$.  We conclude that the original sequence, $\{h_j\}$, had a convergent subsequence in $C_{\textrm{per}}^{2+\alpha}(\overline{R})$, and hence $\{(Q_j, h_j)\}$ had a convergent subsequence in $\mathbb{R} \times X$ \qquad\end{proof} \\

\begin{lemma} \emph{(Fredholm Map)} For each $(Q,h) \in \mathcal{O}_{\delta}$, the linearized operator $\mathcal{G}_h(Q,h)$ is a Fredholm map of index 0 from $X$ to $Y$. \label{fredholmlemma} \end{lemma}

\begin{proof}  In the previous lemma we established that $\mathcal{G}^{(\sigma)}$ was uniformly elliptic and oblique $\forall (Q,h) \in \mathcal{O}_{\delta}$ and $\forall \sigma \in \mathbb{R}$.  We remark also that these bounds are independent of $\sigma$.  

Now let $\psi \in C_{\textrm{per}}^{3+\alpha}(\mathbb{R})$ be given and fix $\sigma, Q \in \mathbb{R}$, $h \in X$.  Put 
\[ \phi^{(i)} := \partial_q\left(\mathcal{G}_{ih}^{(\sigma)}(Q,h)[\psi]\right)- \left(\partial_q \mathcal{G}_{ih}^{(\sigma)}(Q,h)\right)[\psi], \qquad \textrm{for } i=1,2.\]  
Here, by $(\partial_q \mathcal{G}_{ih}^{(\sigma)}(Q,h))[\psi]$ we mean differentiating the coefficients of $\mathcal{G}_{ih}^{(\sigma)}$ in $q$,  then applying the resulting operator to $\psi$.  Then $\partial_q \psi$ satisfies:
\[ \left\{ \begin{array}{lll}
 \mathcal{G}_{1h}^{(\sigma)}(h)  \partial_q\psi = \phi^{(1)} & & \textrm{on } R \\
 & & \\
 \mathcal{G}_{2h}^{(\sigma)}(Q,h)  \partial_q \psi =  \phi^{(2)} & & \textrm{on } T \\ 
 & & \\
 \partial_q \psi = 0 & & \textrm{on } B \end{array} \right.\]
 which is a uniformly elliptic PDE with an oblique boundary condition.  The classical Schauder estimates ensure the existence of a constant $C > 0$, independent of $\psi$, such that
 \begin{eqnarray*}
 C \|\partial_q \psi\|_{C^{2+\alpha}(\overline{R})} & \leq &  \|\partial_q \psi \|_{C^\alpha(\overline{R})} + \| \phi^{(1)} \|_{C^\alpha(\overline{R})} + \| \phi^{(2)} \|_{C^{1+\alpha}(T)}  \\
 & \leq & \|\psi\|_{C^{2+\alpha}(\overline{R})} + \|\partial_q (\mathcal{G}_{1h}^{(\sigma)}(h) \psi) \|_{C^\alpha(\overline{R})} + \|\partial_q (\mathcal{G}_{2h}^{(\sigma)}(Q,h) \psi) \|_{C^{1+\alpha}(T)}.  \end{eqnarray*}
On the other hand, we may express $\partial_p^2 \psi$ via the partial differential equation to arrive at an estimate for $\partial_p^3 \psi$ of the same type.  Combining we get
\[ C \|\psi\|_{C^{3+\alpha}(\overline{R})} \leq \|\psi\|_{C^{2+\alpha}(\overline{R})} + \|\partial_q \phi^{(1)}\|_{C^\alpha(\overline{R})} + \|\partial_q \phi^{(2)} \|_{C^{1+\alpha}(T)}. \]

If we now apply the classical Schauder estimates to $\psi$, we find that for some $C > 0$, independent of $\psi$,
\[ C \|\psi\|_{C^{2+\alpha}(\overline{R})} \leq \|\psi\|_{C^{\alpha}(\overline{R})} + \|\partial_q \left(\mathcal{G}_{1h}^{(\sigma)}(h) \psi\right) \|_{C^\alpha(\overline{R})} + \|\partial_q \left(\mathcal{G}_{2h}^{(\sigma)}(Q,h)\psi\right) \|_{C^{1+\alpha}(T)}. \]
Together with the previous line, this implies that there exists a constant $C = C(\sigma) > 0 $ such that, for all $\psi \in C^{3+\alpha}(\overline{R})$ even and periodic in $q$ that vanish on $B$,
\be C(\sigma) \|\psi\|_X \leq \|\psi\|_{C^\alpha(\overline{R})} + \|\mathcal{G}_{1h}^{(\sigma)}(h) \psi\|_{Y_1} + \|\mathcal{G}_{2h}^{(\sigma)}(Q,h) \psi\|_{Y_2}. \label{estimate1fredholmmap}\ee
Of course, this is not quite the estimate we are after, since there is a lingering dependence on $\sigma$.  In order to remedy this we make the following elementary, but very useful, observation: for any $\psi$ as above we have
\[ |d(\psi)| \leq \fint_T |\psi|dq \leq \|\psi\|_{C^\alpha(\overline{R})}. \]
We are therefore able to estimate terms involving $d$ easily in spaces of H\"older continuous functions.  In particular,
\begin{eqnarray}
\|\mathcal{G}_{1h}(h)\psi - \mathcal{G}_{1h}^{(\sigma)}(h)\psi \|_{C^\alpha(\overline{R})} & = & \|3g(d(h)-\sigma) h_p^3 \rho_p \psi_p + g\rho_p h_p^3 d(\psi)\|_{C^\alpha(\overline{R})} \nonumber \\
& \leq & C \|\psi\|_{C^{1+\alpha}(\overline{R})}, \label{GGsigmaCalphaerror} \end{eqnarray}
where the constant $C$ above depends only on $\rho_p, \|h\|_X$ and our choice of $\sigma$.  Likewise, an identical argument gives the estimates
\[ \|\partial_i \mathcal{G}_{1h}(h)\psi - \partial_i \mathcal{G}_{1h}^{(\sigma)}(h)\psi\|_{C^\alpha(\overline{R})} \leq C\|\psi\|_{C^{2+\alpha}(\overline{R})}, \qquad i = p,~q \]
where $C$ depends again on $\rho_p, \|h\|_X$ and the choice of $\sigma$.  Of course we do not need to make such arguments to estimate the $\mathcal{G}_{2h}^{(\sigma)}(h)$ term, as it identical to $\mathcal{G}_{2h}(h)$.

Combining these observations with our first estimate \eqref{estimate1fredholmmap}, we find that, for some $C > 0$ and all $\psi \in X$, 
\be C \|\psi\|_X \leq \|\psi\|_{C^{2+\alpha}(\overline{R})} + \|\mathcal{G}_{1h}(h) \psi\|_{Y_1} + \|\mathcal{G}_{2h}(Q,h) \psi\|_{Y_2}. \label{estimate2fredholmmap} \ee

The key point here is that $X = C_{\textrm{per}}^{3+\alpha}(\overline{R})$ is compactly embedded in the spaces who appear on the right-hand side of \eqref{estimate2fredholmmap}. To show that the null space of $\mathcal{G}_h(Q,h)$ is finite dimensional, for example, it suffices to show that the unit ball is compact.  But for any sequence $\{\psi_n\} \subset \{\psi \in \mathcal{N}(\mathcal{G}_h(Q,h)) : \|\psi\|_{C^{3+\alpha}(\overline{R})} \leq 1\}$, we have, by the compactness of the embedding, there exists a convergent subsequence $\{\psi_{n_k}\}$ in $C^{2+\alpha}(\overline{R})$.  This subsequence is Cauchy in $C^{2+\alpha}(\overline{R})$, so applying \eqref{estimate2fredholmmap}, we see it is also Cauchy in $X$.  Finally, by completeness, $\{\psi_{n_k}\}$ is convergent in $X$.  It follows that the unit ball in the null space is compact, hence the null space is finite dimensional.

The proof that the range is closed is, likewise, standard.  Again, the argument hinges on the compact embedding of $C^{3+\alpha}(\overline{R}) \subset\subset C^{2+\alpha}(\overline{R})$.  For brevity we omit the details.  

Now, the Fredholm index has a discrete range, hence by the connectedness of $\mathcal{O}_{\delta}$ it must be constant on this set.  But $\mathcal{G}_h(Q^*, H^*)$, which we denoted $\mathcal{F}_w(\lambda^*, 0)$ in the previous section, was shown to have a one-dimensional null space in \textsc{Lemma \ref{nullspacelemma}}.  Moreover, by the orthogonality condition of \textsc{Lemma \ref{rangelemma}}, its range has codimension one.  Since $(Q^*, H^*) \in \mathcal{O}_{\delta}$ by construction, the Fredholm index is uniformly 0 along the continuum $\mathcal{C}_{\delta}^\prime$. \qquad\end{proof} \\

Finally we prove a result characterizing the spectrum of the operator $\mathcal{G}$.  \\ 

\begin{lemma} \emph{(Spectral Properties)} \emph{(i)} $\forall \delta> 0$, $\exists c_1, c_2 > 0$ such that for all $(Q,h) \in \mathcal{O}_{\delta}$ with $|Q| + \|h\|_X \leq M$, for all $\psi \in X$ and for all real $\mu \geq c_2$, we have
\[ c_1 \|\psi\|_X \leq \mu^{\frac{\alpha}{2}}\|(A-\mu)\psi \|_{Y_1} + \mu^{\frac{1+\alpha}{2}} \|B\psi \|_{Y_2}\]
where $A = A(Q,h) = \mathcal{G}_{1h}(h)$ and $B = B(Q,h) = \mathcal{G}_{2h}(Q,h)$.  \\

\noindent
\emph{(ii)} Define the spectrum $\Sigma = \Sigma(Q,h)$ by:
 \[\Sigma(Q,h) := \{ \mu \in \mathbb{C} : A -\mu \textrm{ is not isomorphic from } \{ \psi \in X : B\psi = 0 \textrm{ on } T\} \textrm{ onto } Y_1\}.\] 
Then $\Sigma$ consists entirely of the eigenvalues of finite multiplicity with no finite accumulation points.  Furthermore, there is a neighborhood $\mathcal{N}$ of $[0,+\infty)$ in the complex plane such that $\Sigma(\lambda, w) \cap \mathcal{N}$ is a finite set.  \\

 \noindent
 \emph{(iii)} For all $(Q,h) \in \mathcal{O}_{\delta}$, the boundary operator $\mathcal{G}_{2h}(Q,h)$ from $X \to Y_2$ is onto. \label{spectralpropertieslemma} \end{lemma} 

\begin{proof}  In this proof we follow the well-tread path laid forth by Agmon in \cite{Ag}.  The only novelty here is in (i), where we must do a little work before we can apply the elliptic estimates.   The remaining parts are completely standard; their proofs rely only on (i), \textsc{Lemma \ref{propermaplemma}} and \textsc{Lemma \ref{fredholmlemma}} (cf. \cite{Ag}, \cite{CS1} or \cite{HS}, for example).  For that reason we omit them here and devote our attention to (i).

Fix $\sigma \in \mathbb{R}$, and let $\psi \in X$ and $\mu \in \mathbb{C}$ be given.  Put $\theta := \arg \mu$, and suppose for some $\epsilon > 0$, $|\theta| < \pi/2 + \epsilon$.  Consider the operator $D^{(\sigma)} := A^{(\sigma)} + e^{i\theta}\partial_t^2$ on $\mathbb{R}\times R$, where $A^{(\sigma)} := \mathcal{G}_{1h}^{(\sigma)}(h)$.  Then, $D^{(\sigma)}$ is elliptic for each $\sigma$, and the boundary condition on the top is oblique, hence the complementing condition holds.  Let $\zeta: \mathbb{R} \to \mathbb{R}$ be a cutoff function supported compactly in the interval $I := (-1,1)$.  Put 
\[ e(t) := e^{i|\mu|^{1/2} t} \zeta(t), \qquad \phi(t,q,p) := e(t) \psi(q,p).\]
We apply the Schauder estimates in $\mathbb{R}^3$ with the boundary operator $B$ to $\phi(t,q,p)$ to deduce
\be C\|\phi\|_{C^{3+\alpha}(I \times R)} \leq  \|\phi\|_{C^\alpha(I \times R)} + \|D^{(\sigma)} \phi\|_{C^{1+\alpha}(I \times R)} + \|B\phi\|_{C^{2+\alpha}(I \times T)} \label{spectralbasicestimate} \ee
for some constant $C > 0$ independent of $\psi$.  

A quick calculation readily confirms the existence of $C,~C^\prime > 0$, depending only on $\zeta$, with
\be C \mu^{\frac{\alpha}{2}} \leq \| e \|_{C^\alpha(\mathbb{R})} \leq C^\prime \mu^{\frac{\alpha}{2}}, \qquad C \mu^{\frac{1+\alpha}{2}} \leq  \| e \|_{C^{1+\alpha}(\mathbb{R})} \leq C^\prime \mu^{\frac{1+\alpha}{2}}. \label{spectralestimatee}\ee
Using this, we can unpack \eqref{spectralbasicestimate} to derive the following set of estimates: 
\begin{eqnarray*}
 \|D^{(\sigma)} \phi \|_{C^{1+\alpha}(I \times R)} & = & \left\| e(t) \mathcal{G}_{1h}^{(\sigma)}(h) [\psi] - \mu e(t)\psi  + \left(|\mu|^{1/2}\zeta^\prime + \zeta^{\prime\prime}\right) e^{i\theta} e^{i |\mu|^{1/2} t} \psi \right\|_{C^{1+\alpha}(I \times R)} \\
 & \leq & C|\mu|^{\frac{1+\alpha}{2}}  \left(\left\|\left(\mathcal{G}_{1h}(h) -\mu\right) \psi \right\|_{C^{1+\alpha}(R)}+ \|({G}_{1h}(h) - \mathcal{G}_{1h}^{(\sigma)}(h) ) \psi \|_{C^{1+\alpha}(R)}+\| \psi \|_{C^{1+\alpha}(R)} \right)\end{eqnarray*}
 for $|\mu|$ sufficiently large.  We can estimate the second term in parenthesis by appealing to \eqref{GGsigmaCalphaerror} in the previous lemma,  yielding
 \be  \|D^{(\sigma)} \phi \|_{C^{1+\alpha}(I \times R)} \leq C|\mu|^{\frac{1+\alpha}{2}} \left( \left\|\left(\mathcal{G}_{1h}(h) -\mu\right) \psi \right\|_{C^{1+\alpha}(R)}+\| \psi \|_{C^{2+\alpha}(R)} \right), \label{spectralDestimate} \ee
 for $|\mu|$ sufficiently large.
 
 Similarly, analyzing the boundary terms we find for some $C > 0$, independent of $\psi$, and $|\mu|$ sufficiently large,
 \begin{eqnarray}
 \|B \phi\|_{C^{2+\alpha}(I \times T)} & = & \| e(t) \mathcal{G}_{2h}(Q,h)[\psi] \|_{C^{2+\alpha}(I \times T)} \nonumber \\
 & \leq & C|\mu|^{\frac{2+\alpha}{2}} \| \psi\|_{C^{2+\alpha}(T)}. \label{spectralBestimate} \end{eqnarray}
 On the other hand, using \eqref{spectralestimatee} and the fact $C^{3+\alpha} \subset\subset C^{k+\alpha}$, for $k < 3$,  we can show that for some $C > 0$, independent of $\psi$, and for all $|\mu|$ sufficiently large,
 \be \|\phi \|_{C^{3+\alpha}(\mathbb{R}^3)}  \geq C \mu^{\frac{1}{2}} \|\psi \|_{C^{3+\alpha}(R)}.  \label{spectralphiestimate} \ee
 Now, inserting \eqref{spectralDestimate}, \eqref{spectralBestimate} and \eqref{spectralphiestimate} into \eqref{spectralbasicestimate}, we arrive at the inequality in (i). \qquad\end{proof} \\
 
 Let $\mathcal{W}$ be an open bounded subset of $X$ and fix $Q \in\mathbb{R}$.  By virtue of \textsc{Lemma \ref{propermaplemma}}-\textsc{Lemma \ref{spectralpropertieslemma}}, we conclude that $\mathcal{G}(Q,\cdot):\overline{\mathcal{W}} \to Y$ is admissible in the sense of \cite{HS}.  This enables us to use the generalization of the Leray-Schauder degree introduced in that paper, which we now briefly recapitulate.
 
Let $y \in Y\setminus \mathcal{G}(Q,\partial \mathcal{W})$ be a regular value of $\mathcal{G}(Q,\cdot)$.  We define the (Healey-Simpson) degree of $\mathcal{G}(Q,\cdot)$ at $y$ with respect to $\mathcal{W}$ by
 \[ \textrm{deg }(\mathcal{G}(Q,\cdot), \mathcal{W}, y) := \sum_{w \in \mathcal{G}^{-1}(\{y\})} (-1)^{\nu(w)},\]
 where $\nu(w)$ is the number of positive real eigenvalues counted by multiplicity of $\mathcal{G}_{1w}(Q,w)$, subject to the boundary condition $\mathcal{G}_{2w}(Q,w) = 0$ on $T$.  By \textsc{Lemma \ref{spectralpropertieslemma}}, $\nu(w)$ is finite.  Likewise, the properness of $\mathcal{G}(Q,\cdot)$ established in \textsc{Lemma \ref{propermaplemma}} implies $\mathcal{G}^{-1}(\{y\}) \cap \mathcal{W}$ is a finite set.  The degree is therefore well-defined at proper values.  In the usual way, this definition is extended to critical values via the Sard-Smale theorem (this is valid by \textsc{Lemma \ref{fredholmlemma}}).  
 
 Our interest in degree theory stems from the fact that the degree is invariant under any homotopy that respects the boundaries of $\mathcal{W}$ in the following sense.  Let $\mathcal{U} \subset [0,1] \times \mathcal{W}$ be open.  Define $\mathcal{U}_t := \{ w \in \mathcal{W} : (t,w) \in \mathcal{U}\}$ and likewise $\partial \mathcal{U}_t := \{ w \in \mathcal{W} : (t,w) \in \partial \mathcal{U}\}$.  \\
 
 \begin{lemma} \emph{(Homotopy Invariance)} The degree is invariant under admissible homotopies. That is, suppose  $\mathcal{U} \subset [0,1] \times \mathcal{W}$ is open.  If $\mathcal{H} \in C^2(\overline{\mathcal{U}}; Y)$ is proper and, for each $t \in [0,1]$, $\mathcal{H}(t,\cdot)$ is admissible (i.e., satisfies the conclusions of \textsc{Lemma \ref{propermaplemma}}-\textsc{Lemma \ref{spectralpropertieslemma}}), then we say $\mathcal{H}$ is an admissible homotopy and have 
 \[ \mathrm{deg}~(\mathcal{H}(0,\cdot), \mathcal{U}_0, y) = \mathrm{deg}~(\mathcal{H}(1,\cdot), \mathcal{U}_1, y)\]
 provided $y \notin \mathcal{H}(t,\partial \mathcal{U}_t)$ for all $t \in [0,1]$. \label{homotopyinvariancelemma} \end{lemma}
 
\noindent \\ From our previous analysis it follows that $\mathcal{G}(Q_1, \cdot) \mapsto \mathcal{G}(Q_2, \cdot)$ is an admissible homotopy.  Thus, in light of \textsc{Lemma \ref{homotopyinvariancelemma}}, the degree will remain constant as we move along the continuum.  Assuming that all three alternatives of \textsc{Theorem \ref{globalbifurcationtheorem}} fail, we use this feature to generate a contradiction.  At this stage, we have reestablished all the relevant properties of $\mathcal{G}$ that were true in the constant density case.  Repeating (verbatim) the proof in \cite{CS1} we obtain \textsc{Theorem \ref{globalbifurcationtheorem}}.  \\
 
\section{Nodal Pattern}
 We now seek to investigate the second possibility of the global bifurcation theorem, namely that the continuum arcs back and intersects $\mathcal{T}$ at some point aside from $\lambda^*$.  This will be done by assuming there exists some other laminar solution on $\mathcal{C}_\delta^\prime$ and analyzing the nodal pattern to show it to be, in fact, equal to $\lambda^*$.  In particular, we will be concerned with the vanishing of $h_q$.  We therefore work in the set $\Omega := (0,\pi) \times (p_0,0) \subset R$.  Denote 
\begin{eqnarray*} \partial \Omega_t & := & \{(q,0) : q \in (0,\pi) \},  \\ \partial \Omega_b & := & \{(q,p_0): q \in (0,\pi) \}, \\
\partial \Omega_r &:= & \{(\pi,p) : p \in (p_0, 0) \}, \\ \partial \Omega_l & := & \{(0, p) : p \in (p_0, 0) \}. \end{eqnarray*}
It follows that $h = 0$ on $\partial \Omega_b$ for all $(Q,h) \in \mathcal{C}_\delta^\prime$, and for $h \in X$ periodicity and evenness in $q$ yield $h_q = 0$ on $\partial \Omega_r \cup \partial \Omega_l$.  In the following lemmas we will attempt to prove that for $h \in \mathcal{C}_\delta^\prime \setminus \{(Q^*, H^*)\}$ we have additionally
\begin{equation}
\left \{ \begin{array}{cc}
h_q < 0 & \textrm{in } \Omega \cup \partial\Omega_t \\ 
h_{qp} < 0 & \textrm{on } \partial\Omega_b  \\
h_{qq} < 0 & \textrm{on } \partial\Omega_l \\
h_{qq} > 0 & \textrm{on } \partial\Omega_r  \end{array} \right. \label{signhqhqphqq} \end{equation}
and on the bottom corners of $\Omega$:
\begin{equation} h_{qqp}(0,p_0) < 0, \qquad h_{qqp}(\pi, p_0) > 0, \label{hqqptopcorner} \end{equation}
at the top corners, we have
\begin{equation} h_{qq}(\pi, 0) > 0, \qquad h_{qq}(0,0) < 0. \label{hqqbottomcorner} \end{equation}
These inequalities define an open set in $X$.  \\

Our first result is the following: \\

\noindent
\begin{lemma} Properties \eqref{signhqhqphqq}-\eqref{hqqbottomcorner} hold in a small neighborhood of $(Q^*,H^*)$ in $\mathbb{R}\times C_{\textrm{per}}^{3+\alpha}(\bar{\Omega})$ along the bifurcation curve $\mathcal{C}_{\textrm{loc}}^\prime \setminus \{(Q^*,H^*)\}$ originating from the point $(Q^*, H^*)$. \end{lemma}

\noindent \\The proof of this lemma does not make any special reference to the form of the operator $\mathcal{G}$, only the basic properties of the eigenfunction $w^*$ and the local bifurcation curve.  Not surprisingly, therefore, it follow with virtually no modification from the homogeneous case treated in \cite{CS1}.  We therefore omit it and concentrate on the remaining lemmas, which will require more finesse.

Differentiating the relation $\mathcal{G}_1^{(\sigma)}(h) = 0$ by $q$, for any choice of $\sigma \in \mathbb{R}$, yields $\mathcal{G}_{1h}^{(\sigma)}(h)[h_q] = 0$ for $(Q,h) \in \mathcal{C}_\delta^\prime$, where we recall that 
\begin{eqnarray*}
 \mathcal{G}_{1h}^{(\sigma)}(h) [\phi] & = & 2h_q h_{pp} \phi_q + (1+ h_q^2) \phi_{pp} + 2h_p h_{qq} \phi_p + h_p^2 \phi_{qq} \\ 
 & & - 2(h_{pq}h_q \phi_p + h_{q}h_p \phi_q + h_p h_q \phi_{pq})  \\
 & & - 3g\rho_p (h-\sigma) h_p^2 \phi_p -g\rho_p h_p^3\phi + 3h_p^2 \beta(-p) \phi_p. \end{eqnarray*}
Likewise, differentiating the boundary relation we find that $\mathcal{G}_{2h}(Q,h)[h_q] = 0$, for $(Q,h) \in \mathcal{C}_\delta^\prime$, where
\[ \mathcal{G}_{2h}(Q,h)[\phi] := 2h_q \phi_q + 2g\rho h_p^2\phi + 2h_p(2g\rho h - Q)\phi_p .\]

An important observation is that $\mathcal{G}_{1h}^{(\sigma)}$ --- though elliptic for given $\sigma \in \mathbb{R}$ and $h \in \mathcal{O}_\delta$ --- has a zero-th order term.  We have dictated that $\rho_p \leq 0$, so that the sign will go precisely ``the wrong way'', in the sense that the maximum principle will not hold in general.  The argument will be saved by the observation that the $\phi$ under consideration is known \emph{a priori} to be non-positive.  If we denote
\[ \mathcal{H}^{(\sigma)}(h) := \mathcal{G}_{1h}^{(\sigma)}(h) + g\rho_p h_p^3,\]
then $\mathcal{H}^{(\sigma)}(h)$ is uniformly elliptic with no zero-th order terms.  Moreover, every non-positive $\phi$ that is a solution relative to $\mathcal{G}_{1h}^{(\sigma)}(h)$ is a subsolution relative to $\mathcal{H}^{(\sigma)}(h)$.  \\

\noindent
\begin{lemma} Properties \eqref{signhqhqphqq}-\eqref{hqqbottomcorner} hold along $\mathcal{C}_\delta^\prime \setminus\{(Q^*,H^*)\}$ unless $\exists \lambda \neq \lambda^*$ with $(Q(\lambda), H(\lambda)) \in \mathcal{C}_\delta^\prime$.  \end{lemma}

\begin{proof} By the local bifurcation theorem we know that in a sufficiently small neighborhood of $(Q^*,H^*)$ the bifurcation curve $\mathcal{C}_\delta^\prime$ consists entirely of the curve $C_\textrm{loc}^\prime$.  Suppose by contradiction that the statement of the lemma is false.  As the previous lemma implies that \eqref{signhqhqphqq}-\eqref{hqqbottomcorner} will hold near $(Q^*, H^*)$ in $\mathbb{R}\times C^{3+\alpha}(\bar{\Omega})$, we have the existence of some $(Q,h) \in \mathcal{C}_\delta^\prime$ with $h_q \nequiv 0$ where at least one of the properties fail, although they hold on a sequence $(Q_j, h_j) \to (Q,h)$.  Continuity then implies that $h_q \leq 0$ on $\bar{\Omega}$, as this holds for each $h_j$.  Since $\mathcal{C}_\delta^\prime \subset \mathcal{O}_\delta$, we have that $\mathcal{G}_{1h}^{(\sigma)}(Q,h)$ is uniformly elliptic where we take $\sigma := d(h)$.  

We wish to apply exploit Hopf's lemma to conclude the last three inequalities of \eqref{signhqhqphqq} hold.  To do so, we first note that, as $h_q \leq 0$ in $\Omega$ and $h_q \equiv 0$ on $\partial\Omega_l \cup \partial\Omega_r \cup \partial \Omega_b$, we have $\sup_{\overline{\Omega}} h_q = 0$.  But then, by our comments above, we see that $\phi := h_q$ is a subsolution relative to the linear elliptic operator $\mathcal{H}^{(\sigma)}(h)$.  Applying Hopf's lemma to $\mathcal{H}^{(\sigma)}(h)$, we conclude 
\[ \frac{\partial \phi}{\partial \nu}\bigg|_{(q_0,p_0)} > 0,\]
where $(q_0, p_0)$ is any point where $\phi$ vanishes, and $\nu$ is the outward unit normal at $(q_0,p_0)$ (cf, for example, \textsc{Theorem 2.15} of \cite{Fr}).  In particular, since $\partial_\nu$ = $-\partial_q$ on $\partial\Omega_l$, and $h_q \equiv 0$ on this set, we have $h_{qq}< 0$ on $\partial\Omega_l$.  An identical argument shows $h_{qq} > 0$ on $\partial\Omega_r$, and $h_{qp} < 0$ on $\partial\Omega_b$.  

To establish the strict inequality $h_q < 0$ on $\partial \Omega_t$, we appeal to the nonlinear boundary condition, \eqref{defG2}. We already have that $h_q \leq 0$ in this region by continuity.  Suppose that for some $q_0 \in (0,\pi)$ we have $h_q(q_0,0) = 0$.  Letting $\phi = h_q$ we use our expression for $\mathcal{G}_{2h}(Q,h)[\phi]$.  Furthermore, 
\[ 2h_p(2g\rho h -Q)h_{qp} = 0 \qquad \textrm{at } (q_0,0).\]
Again, by Hopf's lemma,  we have that $h_{qp}(q_0,0) > 0$.  Moreover, as $(Q,h) \in \mathcal{O}_\delta$, we are guaranteed $h_p \geq \delta > 0$.  As $\rho > 0$, this implies $2g \rho h = Q$ on $T$.  But then, the nonlinear boundary condition for the original problem $\mathcal{G}(Q,h) = 0$ would imply $1+h_q^2 =0$, which is a contradiction.  Hence we have the strict inequality and therefore \eqref{signhqhqphqq} holds for $h$.

In order to produce a contradiction, therefore, all that remains is to verify the corner properties \eqref{hqqptopcorner} and \eqref{hqqbottomcorner}.  By continuity we have that $h_{qqp}(\pi, p_0) \geq 0$.  Suppose that it vanishes.  We have that $h(q,p_0) = 0$ for all $q$, hence $h_q(\pi,p_0) = h_{qq}(\pi,p_0) = h_{qqq}(\pi,p_0) = \ldots = 0$.  Also, by evenness and periodicity in $q$ we have $h_q(\pi, p) = 0$ for all $p$.  Hence $h_{qp}(\pi,p) = h_{qpp}(\pi,p) = \ldots = 0$, for all $p$.  But then if $\nu = \nu_q \hat{q} + \nu_p \hat{p}$, where $\hat{q} = (1,0)$, $\hat{p} = (0,1)$, is any vector exiting $\Omega$ at $(\pi,p_0)$, then 
\[ \frac{\partial h_q}{\partial\nu} = \nu_q h_{qq} + \nu_p h_{pp} = 0, \qquad \textrm{at }  (\pi, p_0)\]
and
\[ \frac{\partial^2 h_q}{\partial \nu^2} = 0 \qquad \textrm{at } (\pi, p_0)\]
by the previous analysis. This violates Serrin's edge lemma producing a contradiction (see, for example, \textsc{Theorem E.8} of \cite{Fr}).  An identical argument applied to $(0,p_0)$ proves that \eqref{hqqptopcorner} holds.  

For \eqref{hqqbottomcorner} we consider the point (0,0).  By evenness we know that $h_q(0,p) = 0$ for all $p \in (p_0,0)$.  Thus $h_{qp}(0,p) = 0$ for all such $p$.  By continuity, moreover, we are guaranteed that $h_{qq}(0,0) \leq 0$.  By contradiction suppose that $h_{qq}(0,0) = 0$.  Then differentiating the relation $\mathcal{G}_{2h}(Q,h)[h] = 0$ twice yields
\[ h_q h_{qq} + g\rho h_q h_p^2 +h_p  h_{pq}(2g\rho h-Q) = 0, \qquad \textrm{on } \partial\Omega_t \]
and 
\begin{eqnarray*}
 h_{qq}^2+h_{q}h_{qq}+g\rho h_{qq}h_p^2
 + 2g\rho h_q h_{pq} + h_{pq}^2 (2g\rho h-Q) & & \\
  + h_p h_{pqq}(2g\rho h -Q) + 2g\rho h_q h_p h_{pq} &=& 0, \qquad \textrm{on } \partial \Omega_t.\end{eqnarray*}
Evaluating both of these at $(0,0)$ and taking our assumption into account, we have all terms involving $h_q, h_{qq}$ and $h_{pq}$ drop out.  Hence
\[ h_p h_{pqq}(2g\rho h - Q) = 0, \qquad \textrm{at } (0,0).\]
But we already have seen that $2g\rho h - Q$ cannot be zero, as this leads to a violation of the of the nonlinear boundary condition.  Moreover, $h_p$ is bounded uniformly away from 0 in $O_\delta$, so we must conclude that $h_{pqq}(0,0) = 0$. As before this violates Serrin's edge lemma, as any vector $\nu$ leaving $\Omega$ through $(0,0)$ can be written as $\nu = \nu_p \hat{p} + \nu_q \hat{q}$, and therefore
\[ \frac{\partial h_q}{\partial\nu} = \nu_q h_{qq} + \nu_p h_{pq} = 0, \qquad  \textrm{at } (0,0)\]
and 
\[ \frac{\partial^2 h_q}{\partial \nu^2} = 0 \qquad \textrm{at } (0,0),\]
since $h_{pqq}(0,0) = 0$.  This proves that \eqref{hqqbottomcorner} holds for $h$, hence we have a contradiction establishing the lemma.  \qquad\end{proof}  \\

Combining the two previous lemmas we see that the nodal properties \eqref{signhqhqphqq}-\eqref{hqqbottomcorner} will hold along the continuum $\mathcal{C}_\delta^\prime$ unless it intersects the laminar curve somewhere other than $\lambda^*$.  We now conclude this section by showing this cannot occur for $\lambda$ to the right of $-2B_{\textrm{min}}+\epsilon_0$ . \\

\begin{lemma}If a trivial solution $(Q(\lambda), H(\lambda)) \in \mathcal{C}_\delta^\prime$, with $\lambda > -2B_{\textrm{min}} + \epsilon_0$, then $\lambda = \lambda^*$. \end{lemma}

\begin{proof}  Let $(Q(\lambda), H(\lambda)) \in \mathcal{C}_\delta^\prime$ be given.  Clearly if $\lambda = \lambda^*$ we are done, so suppose that this is not the case. As we know locally to $(Q^*, H^*)$ that $\mathcal{C}_\delta^\prime = \mathcal{C}_{\textrm{loc}}^\prime$, we may assume without loss that there exists a curve $\mathcal{C}_\delta^\prime (\lambda, \lambda^*) \subset \mathcal{C}_\delta^\prime$ originating at $(Q^*, H^*)$ and terminating at $(Q,H)$ that does not intersect $\mathcal{T}$ except at the points $\lambda$, $\lambda^*$.  Let $\{(Q(\lambda_n), h_n )\}$ be a sequence in $\mathcal{C}_\delta^\prime(\lambda,\lambda^*) \setminus \mathcal{T}$ such that $(\lambda_n, h_n) \to (\lambda, H(\lambda))$ in $\mathbb{R} \times C^{3+\alpha}(\bar{\Omega})$.  Thus $\partial_q h_n \nequiv 0$, and $\mathcal{G}(Q,h_n) = 0$, $\forall n \geq 1$.   We claim that $\partial_q h_n < 0$ for $n \geq 1$.  By the first lemma of this section this holds true for all $h$ in some sufficiently small neighborhood of $(Q^*, H^*)$ in $X$.  The arguments of the previous lemma, however, can be directly applied to show that the nodal properties  hold along $\mathcal{C}_\delta^\prime(\lambda,\lambda^*) \setminus \{(Q(\lambda), H(\lambda))\}$.  That is, if we assume that for some $h \in \mathcal{C}_\delta^\prime(\lambda, \lambda^*)$, \eqref{signhqhqphqq}-\eqref{hqqbottomcorner} fail, then we may again choose a sequence of functions on $\mathcal{C}_\delta^\prime(\lambda,\lambda^*)$ approaching $h$ where the properties hold.  The previous lemma can then be applied without further modification to produce a contradiction.  This proves the claim.  Thus $\partial_q h_n < 0$, for $n \geq 1$.

We now differentiate the relation $\mathcal{G}(Q_n,h_n)[h_n] = 0$ with respect to $q$ to find that, for each $n \geq 1$,
\begin{eqnarray*}
 0 & = & \partial_q \mathcal{G}_1(h_n)h_n  \\
 & = & (1+ (\partial_q h_n)^2) \partial_p^2 \partial_q h_n + 2(\partial_q h_n)( \partial_q^2 h_n)( \partial_p^2 h_n)  + (\partial_q^3 h_n )(\partial_p h_n)^2 \\
 & & + 2(\partial_q^2 h_n)(\partial_p h_n)(\partial_p \partial_q h_n)  - 2(\partial_q^2 h_n)(\partial_p h_n)(\partial_p\partial_q h_n) - 2(\partial_q h_n)(\partial_p \partial_q h_n)^2 \\ 
 & & - 2(\partial_q h_n)(\partial_p h_n) (\partial_p \partial_q^2 h_n) - 3 g(h_n-d(h_n)) (\partial_p h_n)^2 (\partial_q \partial_p h_n) \rho_p -g(\partial_q h_n) (\partial_p h_n)^3  \rho_p \\ 
 & & + 3 (\partial_p h_n)^2 (\partial_q \partial_p h_n) \beta(-p) \\ 
 & = & \mathcal{G}_{1h}^{(\sigma_n)}(h_n)[v_n] \end{eqnarray*}
 where $v_n := \|\partial_q h_n\|_{C^{2+\alpha}(\bar{\Omega})}^{-1} \partial_q h_n$ and $\sigma_n := d(h_n)$.  Likewise, we may differentiate the boundary relation to conclude that $\mathcal{G}_{2h}^{(\sigma_n)}(Q_n,h_n)[v_n] = \mathcal{G}_{2h}(Q_n, h_n)[v_n]  = 0$, for all $n \geq 1$.

 For brevity denote $(Q,H) := (Q(\lambda), H(\lambda))$ and $\sigma := d(H)$.  We observe that since $h_n \to H$ in $C^{3+\alpha}(\bar{\Omega})$, we must have $\sigma_n \to \sigma$.  Combining this with the convergence of $h_n \to H$ in the $C^{3+\alpha}$-norm, we see $\{\mathcal{G}_{1h}^{(\sigma_n)}(h_n)\}$ is a Cauchy sequence in $\mathcal{L}(C^{2+\alpha}(\bar{\Omega}),C^\alpha(\bar{\Omega}))$.  Then for $i,j \geq 1$, we have
 \[ \|(\mathcal{G}_{1h}^{(\sigma_i)}(h_i) - \mathcal{G}_{1h}^{(\sigma_j)}(h_j)) v_j \|_{C^\alpha(\bar{\Omega})} \leq \|\mathcal{G}_{1h}^{(\sigma_i)}(h_i) - \mathcal{G}_{1h}^{(\sigma_j)}(h_j) \|_{\mathcal{L}(C^{2+\alpha}(\bar{\Omega}),C^\alpha(\bar{\Omega}))} \to 0, \qquad \textrm{as } i,j \to \infty,\]
 since $\|v_j \|_{C^{2+\alpha}(\bar{\Omega})} = 1$.  Likewise, $\{\mathcal{G}_{2h}^{(\sigma_n)}(Q_n,h_n)\}$ is a Cauchy sequence in $\mathcal{L}(C^{2+\alpha}(T),C^{1+\alpha}(T))$, so that
 \[ \| (\mathcal{G}_{2h}^{(\sigma_i)}(Q_i, h_i) - \mathcal{G}_{2h}^{(\sigma_j)} (Q_j, h_j) ) v_j \|_{C^{1+\alpha}(T)} \to 0, \qquad \textrm{as } i,j \to \infty.\] 
 
By linearity we have
\[ \mathcal{G}_h^{(\sigma_i)}(Q_i, h_i)[v_i - v_j] + (\mathcal{G}_h^{(\sigma_i)}(Q_i, h_i) - \mathcal{G}_h^{(\sigma_j)}(Q_j, h_j)) [v_j] = 0, \qquad \forall i,g \geq 1.\]
Notice that for each $i,j \geq 1$, we have that $\mathcal{G}_h^{(\sigma_i)}(Q_i, h_i)$ is uniformly elliptic, hence we apply the Schauder estimates to the above equation and conclude that there exists a generic constant $C > 0$ independent of $i,j$ satisfying: 
\[ C \|v_i - v_j \|_{C^{2+\alpha}(\bar{\Omega})} \leq \|v_i - v_j \|_{C^0(\bar{\Omega})} + \| (\mathcal{G}_{h}^{(\sigma_i)}(Q_i, h_i) - \mathcal{G}_{h}^{(\sigma_j)} (Q_j, h_j) ) v_j \|_{C^\alpha(\bar{\Omega}) \times C^{1+\alpha}(T)}.\] 
Our previous estimates imply that the last term on the right-hand side vanishes as $i, j \to \infty$.  On the other hand, as the $\{v_j\}$ are uniformly bounded in $C^{2+\alpha}(\bar{\Omega})$, the compact embedding of $C^{2+\alpha}(\bar{\Omega}) \subset\subset C^0(\bar{\Omega})$ guarantees there is subsequence convergent in $C^0(\bar{\Omega})$.  But then the inequality above implies this subsequence converges in $C^{2+\alpha}(\bar{\Omega})$ as well. By abuse of notation we identity the subsequence with $\{v_j\}$ itself.  We note also that as $d(v_j) = d(\partial_q h_j) = 0$, for all $j \geq 1$, that the limit must also have mean zero.  Then we may let $m$ be given with $\partial_q m \in C^{2+\alpha}(\bar{\Omega})$ and $v_j \to \partial_q m$ in this space.  It follows that $\|\partial_q m \|_{C^{2+\alpha}(\bar{\Omega})} = 1$, and 
\[ \mathcal{G}_h^{(\sigma)} (Q,H)[\partial_q m] = \mathcal{G}_h(Q,H)[\partial_q m] = 0,\]
where the first equality follows from the fact that $\partial_q m$ has mean zero so the $d$ term in $\mathcal{G}_h$ vanishes.  

As we have argued before, evenness and periodicity in $q$ together imply that $\partial_q h_j$ vanishes on $\partial \Omega_r \cup \partial \Omega_l$ for $j \geq 1$.  Moreover, as $h_j \equiv 0$ on $\partial \Omega_b$, we have that $v_j \equiv 0$ on $\partial\Omega_b \cup \partial\Omega_l \cup \partial \Omega_r$.  We have, additionally, that $\partial_q h_j < 0$ in $\Omega$, for each $j \geq 1$.  Together these imply that $\partial_q m \equiv 0$ on $\partial\Omega_b \cup \partial\Omega_l \cup \partial\Omega_r$ and $\partial_q m \leq 0$ in $\Omega$.  But $\partial_q m$ has unit norm and satisfies the uniformly elliptic equation $\mathcal{G}_h^{(\sigma)}[\partial_q m] = 0$, hence we may apply the maximum principle to conclude $\partial_q m < 0$ in $\Omega$.  

Mirroring our derivation of the eigenfunction at $(Q^*, H^*)$, we expand $\partial_q m$ in a sine series
\[ \partial_q m(q,p) = \sum_{k=0}^\infty f_k(p) \sin{kq}, \qquad q,p \in R.\]
As $\mathcal{G}_h(Q,H)[\partial_q m] = 0$, each mode in the series expansion must also satisfy this relation.  Evaluating $\mathcal{G}_h (Q,H)$ we may drop all terms involving $q$ derivatives to find
\[ \mathcal{G}_{1h}(Q,H) = \partial_p^2 + H_p^2 \partial_q^2 -3g\rho_p (H-d(H))H_p^2 \partial_p - g \rho_p H_p^3 + 3H_p^2 \beta(-p) \partial_p,\]
\[ \mathcal{G}_{2h}(Q,H) = \bigg( 2g\rho H_p^2 +2H_p(2g\rho H - Q)\partial_p \bigg)\bigg|_T = \bigg( g\lambda^{-1} - \lambda^{1/2} \partial_p \bigg)\bigg|_T.\]
In particular, for $k =1$ find 
\[ 0 = \big(\partial_p^2 + H_p^2 \partial_q^2 -3g\rho_p (H-d(H))H_p^2 \partial_p - g \rho_p H_p^3 + 3H_p^2 \beta(-p) \partial_p\big) f_1, \qquad \forall p \in (p_0, 0),\]
or in self-adjoint form:
\[ \partial_p(H_p^{-3} \partial_p f_1) = (H_p^{-1}+g\rho_p)f_1, \qquad \forall p \in (p_0,0).\]
Similarly, applying $\mathcal{G}_{2h}(Q,H)$ we find 
\[ f_1(p_0) = 0, \qquad f_1^\prime(0) = g \rho(0)\lambda^{-3/2} f_1(0).\]
Returning to our original notation, this implies 
\[ -1 = \mathcal{R}(f_1; \lambda) \geq \mu(\lambda)\]
where we recall $\mathcal{R}$ is the Rayleigh quotient defined in \eqref{defcalR}.  If $\mu(\lambda) < -1$, then $-1$ is not the minimum eigenvalue, so $f_1$ is an excited state of the above Sturm-Liouville problem.  We conclude by standard Sturm-Liouville theory that $f_1$ vanishes at some point in $(p_0,0)$.  But we know from explicit computation that 
\[ f_1(p) = \frac{2}{\pi} \int_0^\pi (\partial_q m)(q,p) \sin{q} dq < 0, \qquad p \in (p_0, 0)\]
as $\partial_q m < 0$.  This contradicts the existence of nodes, so we must have instead that $\mu(\lambda) = -1 = \mu(\lambda^*)$.  By \textsc{Lemma \ref{locationlambdastarlemma}} it follows that $\lambda = \lambda^*$ as desired. \qquad\end{proof} \\

\section{Uniform Regularity}
In the previous section we considered the second alternative of \textsc{Theorem \ref{globalbifurcationtheorem}}. Now we wish to delve further into the first.  We make the central aim of this section the establishment of uniform bounds in the H\"older norm along the continuum $\mathcal{C}_\delta^\prime$.  The ultimate product of our efforts will be the following theorem.  \\

\begin{theorem} \emph{(Uniform Regularity)} Let $\delta > 0$ be given.  If $\sup_{(h,Q) \in \mathcal{C}_\delta^\prime} \sup_{\overline{R}} h_p$ and $\sup_{(h,Q) \in \mathcal{C}_\delta^\prime} |Q|$ are finite, then $\sup_{(h,Q) \in \mathcal{C}_\delta^\prime}\|h\|_{C^{3+\alpha}(\overline{R})}$ is finite. \label{uniformregularitytheorem} \end{theorem} 

\noindent \\To treat the second derivatives we shall employ a suite of \emph{a priori} estimates for nonlinear elliptic equations with oblique boundary conditions due to Gilbarg and Trudinger (cf. \cite{LT}).  We shall not, however, need these results in their full generality, so instead we streamline them into a single statement.

First note that due to the periodicity in $q$, the domain $R$ can be considered as embedded on a torus, allowing us to ignore the seeming lack of smoothness at the corner points $q = 0,~2\pi$.  Next, in keeping with Lieberman and Trudinger's notational framework we consider a differential operator $F = F(z, \xi, r) \in C^2(\mathbb{R}\times \mathbb{R}^2 \times \mathbb{S}, \mathbb{R})$ and boundary operator $G = G(z, \xi) \in C^2(\mathbb{R}\times \mathbb{R}^2, \mathbb{R})$.  Here $\mathbb{S}$ denotes the space of $2\times 2$ real symmetric matrices, $D$ denotes the gradient operator and $D^2$ the Hessian. Both of these are taken as acting on the space of smooth real-valued functions on $\overline{R}$ which are $2\pi$-periodic in the first variable.  As usual, we say that $F$ is elliptic at $(h, \xi, r) \in \mathbb{R} \times \mathbb{R}^2 \times \mathbb{S}$ provided that the matrix $F_r := [ \partial F/ \partial r_{ij}]_{1 \leq i, j \leq 2}$ is positive definite at that point.    Moreover, if $\Lambda_1$ and $\Lambda_2$ denote the minimum and maximum eigenvalue of $F_r$, respectively, then $F$ is said to be uniformly elliptic provided that the ratio $\Lambda_2/\Lambda_1$ is bounded.  The boundary operator $G$ is said to be oblique at a point $(q,0)$ if the normal derivative $\chi := G_\xi \cdot (0, -1)$ is positive at $(q,0)$ for all $(h,\xi) \in \mathbb{R} \times \mathbb{R}^2$.

Assume that $F$ and $G$ are given as above and consider the corresponding nonlinear elliptic boundary value problem 
\be \left \{ \begin{array}{ll}
F(h,Dh, D^2h) = 0, & \textrm{in } R \\
G(h, Dh) = 0 & \textrm{on } p = 0 \\
h = 0 & \textrm{on } p = p_0. \end{array} \right. \label{liebtrudbvp} \ee

\begin{theorem} \emph{(Lieberman-Trudinger)} Let $h \in C^2(\overline{R})$ be a solution to \eqref{liebtrudbvp} that is $2\pi$-periodic in the first variable and for some constant $K > 0$ satisfies $|h|+|Dh| \leq K$ in $\overline{R}$. Suppose further that for some positive constant $M$ the functions $F = F(z, \xi, r)$ and $G = G(z,\xi)$ satisfy the following structural conditions:
\begin{eqnarray} \Lambda_2 & \leq & M \Lambda_1, \label{structure1} \\
 |F|,~|F_\xi|,~|F_z| &\leq& M \Lambda_1 (|r|+1), \label{structure2} \\
 F_{rr} &\leq& 0,  \label{structure3} \\
|G|, |G_z|, |G_\xi|, |G_{zz}|, |G_{z\xi}|, |G_{\xi\xi}| &\leq& M \chi, \label{structure4} \end{eqnarray}
for all $(z, \xi, r) \in \mathbb{R} \times \mathbb{R}^2 \times \mathbb{S}$ such that $|z|+|\xi| \leq K$.  Then there are positive constants $\mu = \mu(M)$ and $C = C(K, M)$ such that $h \in C_{\textrm{per}}^{2+\mu}(\overline{R})$ and
\[ \|h \|_{C_{\textrm{per}}^{2+\mu}(\overline{R})} \leq C.\]
\label{liebtrudtheorem}
\end{theorem} 

This is a nontransparent restatement of the original result presented in \cite{LT}, so we take a brief aside to outline its development.  The general idea will be to incrementally estimate the H\"older norms, beginning at $C^{0+\mu}$ and working upwards to $C^{2+\mu}$.  For clarity, all citations from \cite{LT} are marked with the abbreviation L-T.  \\

\begin{proof} Let $F \in C^2(\mathbb{R}\times \mathbb{R}^2 \times \mathbb{S}, \mathbb{R})$ and $G \in C^2(\mathbb{R}\times \mathbb{R}^2, \mathbb{R})$ be given and consider the boundary value problem \eqref{liebtrudbvp}.  First we wish to show that assuming \eqref{structure1}-\eqref{structure4}, $F$ and $G$ meet the hypotheses of \textsc{L-T Theorem 2.1}, namely structural requirements (L-T F1), (L-T F2), and (L-T G2).  But observe that conditions \eqref{structure1} and  \eqref{structure2} imply (L-T F1) and (L-T F2) respectively, where we taking $\mu_0$ to be a constant.  Moreover, \eqref{structure4} ensures that condition (L-T G2) holds, again taking $\mu_0$ to be a constant function.  So indeed, the hypotheses of \textsc{L-T Theorem 2.1} hold. Applying the theorem, we have for any $C^1(\overline{R}) \cap C^2(R)$ solution $h$ with $|h| < M$, there exists positive constants $\mu  = \mu(n, M, \mu_0)$ and $C  = C(n, M, \mu_0, R)$ with \eqref{structure1}
\[ [h]_{\mu; R} \leq C.\]
Here $[\cdot]_{\mu; R}$ denotes the H\"older seminorm.  If we slightly abuse notation and drop the dependencies on $n$ and $R$, it follows immediately that there exists a constant $C = C(K, M)$ such that 
\[ \|h\|_{C_{\textrm{per}}^{\mu}(R)} \leq C.\]

Next we wish to show that structural requirements (L-T 4.1)-(L-T 4.4) are a proper subset of \eqref{structure1}-\eqref{structure4}.  Clearly,  \eqref{structure1} is equivalent to (L-T 4.1).  Also, taking $\theta = 0$ and noting that $F_x = 0$, we see that {\bf\eqref{structure3}} gives both (L-T 4.2) and (L-T 4.3).  Finally, \eqref{structure4} along with the obliqueness of $G$ implies (L-T 4.4), with an appropriate choice of $\mu_1$.   This verifies all of the hypotheses of \textsc{L-T Theorem 4.1}; hence for some $\mu^\prime > 0$ and $C^\prime > 0$, both depending on $M$ and $K$, we have
\[ [Dh]_{\mu^\prime, R} \leq C^\prime.\]  
By abuse of notation we can therefore, choose $\mu$ and $C$ positive constants such that
\[ \|h\|_{C_{\textrm{per}}^{1+\mu}(R)} \leq C. \]
Again we note that these constants depend only on $K$ and $M$.
  
For the higher derivatives we shall need to make use of \textsc{L-T Theorem 5.4} and \textsc{L-T Theorem 6.2}.  Conditions (L-T 5.1) and (L-T 5.2) coincide exactly with conditions (L-T 4.1) and (L-T 4.2), respectively.  Similarly, condition (L-T 5.4) follows from the positive definiteness of $F$ and (L-T 5.22) from \eqref{structure4}.  We have already seen that (L-T 4.3) holds, so applying the second statement of \textsc{L-T Theorem 5.4}, there exists a positive constant --- which we again denote $C$ --- depending only on $K$ and $M$ with
\[ \sup_{R} |D^2 h| \leq C.\]
Lastly, in order to apply \textsc{L-T Theorem 6.2} and get an estimate of the full $C^{2+\mu}$-norm we need only verify that structural conditions (L-T 6.1), (L-T 6.2), (L-T 6.3) and (L-T 6.4)' hold.  But (L-T 5.1) and (L-T 4.3) together imply (L-T 6.1), and (L-T 6.4)' clearly follows from (L-T 5.22). Positive definiteness of $F$, moreover, gives (L-T 6.2) for a sufficiently large $\bar{\mu}_2$. As this implies we have a bound on the $C^{2}$-norm of $\Lambda_1^{-1} F$, (L-T 6.3) also holds for large enough $\mu_3$. Applying \textsc{L-T Theorem 6.2}, then, and combining with the previous estimates yields the theorem. \qquad\end{proof} \\

That settled, we return to the proof of the uniform regularity theorem.  \\

\puffthm{\ref{uniformregularitytheorem}}  We begin by proving the $L^\infty(\overline{R})$ norm of $h$ is uniformly bounded along the continuum.  To do so we note that by \eqref{defh}
\[ 0 \leq h(q,p) = y+d \leq \eta(0)+d,\]
and thus the definition of $p_0$ in \eqref{defp0} yields
\[ \eta(0)+d \leq \frac{|p_0|}{ \inf_{\overline{D_\eta}}(\sqrt{\rho}(c-u)\big)} = |p_0| \sup_{\overline{R}} h_p \]
where the last equality follows from \eqref{hqhpeq}.  Thus assuming the hypothesis of the theorem, that is that $h_p$ is uniformly bounded in $L^\infty$ along the continuum, then the above inequalities together give the boundedness of the $L^\infty$-norm of $h$.

Next, to bound $h_q$ we fix $Q$ and differentiate in $q$ the full height equation \eqref{heighteq} to find
\begin{eqnarray}
- 3 h_{pq} h_p^2 \beta(-p) & = & (1+h_q^2) h_{ppq}+2h_q h_{qq} h_{pp} + 2 h_{qq} h_{pq} h_p  + h_{qqq} h_p^2 \nonumber \\ 
& &- 2h_{qq} h_p h_{pq}  - 2h_q h_{pq}^2 - 2h_q h_p h_{pqq} \nonumber \\
& & - gh_q h_p^3 \rho_p - 3g(h-d) h_{pq} h_p^2 \rho_p, \qquad \textrm{in } R,  \label{innerhqheighteq} \end{eqnarray}
\begin{eqnarray}
2h_q h_{qq} + 2h_{pq}h_p(2g\rho h - Q) + 2g\rho h_q h_p^2 & = & 0, \qquad \textrm{on } p = 0, \label{hqheighteqboundT} \\
h_q & = & 0, \qquad \textrm{on } p = p_0. \label{hqheighteqboundB} \end{eqnarray}
Consider now the function $k(q,p) := h(q,p) - d + \epsilon q e^{np}$, where $\epsilon > 0$ and $n \in \mathbb{N}$ will be specified later.  Inserting $k-\epsilon q e^{np}$ in place of $h-d$ in \eqref{innerhqheighteq} we have
\begin{eqnarray}
O(\epsilon^2) & = & - 6 \epsilon q n e^{np} k_{pq} k_p\beta(-p) + 3 k_{pq} k_p^2 \beta(-p) - 3 \epsilon q n e^{np} k_p^2\beta(-p) + 2 k_q k_{qq} k_{pp} - 2 \epsilon n^2 q e^{np} k_q k_{qq} \nonumber \\ & &  
- \epsilon e^{np} k_{pp} k_{qq} + (1+k_q^2) k_{ppq} - 2 \epsilon e^{np} k_q k_{ppq} - (1+k_q^2) \epsilon n^2 e^{np} - 2 k_q k_{pq}^2  + 2 \epsilon e^{np} k_{pq}^2 \nonumber \\
& &  + 4 \epsilon n e^{np} k_{pq} k_q - 2 k_p k_q k_{pqq} + 2 \epsilon e^{np} k_p k_{pqq} + 2 \epsilon n q e^{np} k_q k_{pqq} + k_p^2 k_{qqq} - 2 \epsilon n q e^{np} k_p k_{qqq} \nonumber \\
& & -g k_q k_p^3 \rho_p + 3 g \epsilon q n e^{np}  k_q k_p^2 \rho_p + g \epsilon e^{np} k_p^3 \rho_p -3g k k_{pq} k_p^2 \rho_p + 6 g \epsilon q n e^{np} k k_p k_{pq}\rho_p  \nonumber \\ 
& & +3 g \epsilon q e^{np} k_{pq} k_p^2 \rho_p+ 3g \epsilon n e^{np} k k_p^2 \rho_p, \qquad \textrm{in } R. 
\label{kqinnerheighteq} \end{eqnarray}
At an interior maximum for $k_q$ in $R$ we would necessarily have
\be k_{pq} = k_{qq} = 0, \qquad k_{qpp} \leq 0, \qquad k_{qqq} \leq 0, \qquad k_{qpp}k_{qqq} \geq k_{qqp}^2. \label{kqinnermax} \ee
Thus evaluating \eqref{kqinnerheighteq} at such a maximum point gives the following simplified equation:
\begin{eqnarray}
O(\epsilon^2) & = & k_{ppq} + k_{ppq}(k_q-\epsilon e^{np})^2 - 2k_{pqq}(k_q - \epsilon e^{np})(k_p - \epsilon n q e^{np}) \nonumber \\  
& & + k_{qqq}(k_p - \epsilon n q e^{np})^2 -g k_q [k_p^3 \rho_p - 3 \epsilon q n e ^{np} k_p^2 \rho_p ] \nonumber \\
& & - \epsilon e^{np}[n^2 + n^2 k_q^2 + 3 \beta(-p) nk_p^2-g k_p^3 \rho_p -3g n k k_p^2 \rho_p]. \label{kqinnerheighteqatmax} \end{eqnarray}
Likewise, inserting $k$ into \eqref{hqheighteqboundB} we find $k_q = \epsilon e^{n p_0}$ on $p = p_0$.  In particular this implies that at the interior maximum $k_q$ is positive.  We may therefore choose $n \in \mathbb{N}$  to ensure that the quantities in square brackets above are positive.  In fact, it suffices merely to have $n^2 + 3 \beta(-p) n k_p^2 > 0$ and $3 \epsilon q n e^{np} - k_p > 0$ in $R$ to achieve this, taking into account the signs of $\rho_p,~k,~k_p$ and $k_q$.  But combining this fact with the signs implied by \eqref{kqinnermax} yields a contradiction, since for $\epsilon$ small we must have the right-hand side of \eqref{kqinnerheighteqatmax} is negative.  We conclude $k_q$ does not have an interior maximum.  On the top we have from \eqref{hqheighteqboundT}
\[ 1+(k_q - \epsilon)^2 + (k_p - \epsilon q n)^2\Big(2g \rho(k+d- \epsilon q) - Q\Big) = 0, \qquad p = 0. \]
By assumption $h_p^2 = (k_p - \epsilon n q e^{np})^2$ is uniformly bounded along the continuum, and by our previous argument $h-d = k- \epsilon q e^{np}$ is as well.  But then, as the maximum of $k_q$ occurs on the boundary, and we have shown that on the boundary $k_q$ is controlled by these quantities, we must have that the maximum of $k_q$ is uniformly bounded along $\mathcal{C}_\delta^\prime$. Clearly, the same must hold for $h_q$.  We can repeat the argument but instead taking the auxiliary function $k = h - \epsilon q e^{np}$, which will show that the minimum of $h_q$ is likewise bounded.

For the second derivatives we fix $\sigma \in \mathbb{R}$ and apply \textsc{Theorem \ref{liebtrudtheorem}} to $\mathcal{G}^{(\sigma)}$.  To do this we set 
\begin{eqnarray*}
 F(z, \xi, r) & := & (1+\xi_1^2) r_{22} + r_{11} \xi_2^2 - 2\xi_1 \xi_2 r_{12} - g(z- \sigma)\xi_2^3 \rho_p + \xi_2^3 \beta(-p) \\
G(z, \xi) & := & 1+ \xi_1^2 + \xi_2^2(2g\rho z - Q) \qquad \textrm{on } p = 0, \end{eqnarray*}
where $\xi = (\xi_1, \xi_2) \in\mathbb{R}^2$ and $r = (r_{11}, r_{12}, r_{22}) \in \mathbb{R}^3$.  By using cutoff functions, \textsc{Theorem \ref{liebtrudtheorem}} is applicable in a subset in which solutions exists \emph{a priori} and for each choice of $\sigma$.  In particular, $\xi_2 \geq \delta$ implies \eqref{structure1}, \eqref{structure2} is immediate and as $F_{rr} = 0$, \eqref{structure3} holds.  Finally, as the only modification to $G$ that occurs when considering the variable density case is the addition of the constant $\rho(0)$, a trivially modification of the constant density argument in \cite{CS1} shows that \eqref{structure4} holds.  It follows that for each $\sigma \in \mathbb{R}$, \textsc{Theorem \ref{liebtrudtheorem}} applies to the solution of $\mathcal{G}^{(\sigma)}[h] = 0$.  Of course, this introduces the possibility that the constant $C$ will depend on $\sigma$. However, every $h \in \mathcal{C}_\delta^\prime$ is such a solution for $\sigma = d(h)$, and recall that we have $\sigma = d(h) \leq \sup_{R} |h|$. Thus we may select $K$ independently of $\sigma$.  Moreover, since in the definition of $F$, $\sigma$ does not occur in any of the $r_{ij}$ terms, it does not affect $\Lambda_1$ and $\Lambda_2$.  Indeed, the characteristic polynomial satisfied by the eigenvalues is
\[ \Lambda_i^2 - (1+\xi_1^2+\xi_2^2)\Lambda_i - (1+\xi_1^2)\xi_2^2 - 4 \xi_1^2 \xi_2^2 = 0, \qquad i =1,2\]
which does not depend on $\sigma$.  This is an obvious consequence of the fact that the effects of density variation in the height equation are consigned to the lower order terms.  Finally, the term containing $\sigma$ in $F$ can be easily estimated in terms of $K$.  This fact, along with the uniform boundedness of the $L^\infty$-norms of $h$ and $Dh$ along the continuum, allows us to conclude that there exists $C > 0$ such that 
\[ \sup_{h \in \mathcal{C}_\delta^\prime} \|h\|_{C_{\textrm{per}}^{2+\mu}(R)} \leq C.\]

Finally we establish third derivative bounds.  Denoting $\theta = h_q$, we can differentiate the height equation in $q$ with fixed $Q$ to find 
\be \left \{ \begin{array}{ll}
(1+h_q^2) \theta_{pp} -2h_p h_q \theta_{pq} + h_p^2 \theta_{qq} & \\
\qquad = f(h, h_p, h_q, h_{pp}, h_{pq}, h_{qq}) & \textrm{in } R \\
h_q \theta_q + g\rho h_p^2 \theta + (2g\sigma \kappa[h] - Q)h_p \theta_p = 0 & \textrm{on } p = 0 \\
\theta = 0 & \textrm{on } p = p_0, \end{array} \right. \label{unifregularitythetaeq} \ee
with $C_{\textrm{per}}^{1+\mu}(\overline{R})$ coefficients and right-hand side $f$ a $C_{\textrm{per}}^\mu(\overline{R})$ function.  Since we have already proved $h$ is bounded uniformly in $C_{\textrm{per}}^{2+\mu}(\overline{R})$, it follows that the right-hand side of \eqref{unifregularitythetaeq} is bounded uniformly in $C_{\textrm{per}}^\mu(\overline{R})$ along the continuum.  As we have seen (in the proof of \textsc{Lemma \ref{propermaplemma}}, for example), so long as $h_p$ is bounded uniformly away from 0 along the continuum, the problem above will be uniformly elliptic with a uniformly oblique boundary condition.  We may therefore apply Schauder theory to obtain uniform \emph{a priori} estimates of $\theta$ in $C_{\textrm{per}}^{2+\mu}(\overline{R})$ over all of $\mathcal{C}_\delta^\prime$.

It remains only to bound $h_{ppp}$ in $C_{\textrm{per}}^\mu(\overline{R})$.  By means of the height equation \eqref{interheighteq}, we may express $h_{pp}$ in terms of the lower order derivatives of $h$:
\[ h_{pp} = -(1+h_q^2)^{-1} \left( h_p^3 \beta(-p) + h_{qq}h_p^2 - 2h_q h_p h_{pq} - g(h-d(h)) h_p^3 \rho_p\right) \]
But the right-hand side above is in $C^{1+\mu}_{\textrm{per}}(\overline{R})$ by the arguments of the previous paragraphs.  Thus $h_{ppp}$ in $C^{\mu}_{\textrm{per}}(\overline{R})$ as desired.  

To transition back to the original H\"older exponent, $\alpha$, we merely note that if $h \in C^{3+\mu}_{\textrm{per}}(\overline{R})$, then it is certainly in $C^{2+\alpha}_{\textrm{per}}(\overline{R})$.  The arguments we have used to derive the third derivative bounds in no way relied on the particular value of $\mu$, so running through them with $\alpha$ instead we find $h \in C^{3+\alpha}_{\textrm{per}}(\overline{R})$. \qquad$\square$ \\

\section{Main Result}
With the regularity established in the previous section, we are now prepared to begin the task of proving the main theorem, \textsc{Theorem \ref{mainresult}}.  \textsc{Theorem \ref{globalbifurcationtheorem}} gave us three possibilities for the continuation of the local bifurcation curve.  Then, \textsc{Theorem 6.1} of the previous section allows us to conclude that if $\max_{\overline{R}} \partial_p h$ and $Q$ remain bounded along the continuum, then $h$ remains bounded in $X$. We now unpack the remaining alternatives and consider their significance in the context of the original problem \eqref{euler2}-\eqref{boundcond}.   \\

\puffthm{\ref{mainresult}}  Fix $\delta > 0$.  If $\mathcal{C}_\delta^\prime$ is unbounded in $\mathbb{R}\times X$, then at least one of the following must occur \\
\begin{itemize}
\item there exists a sequence $(Q_n, h_n) \in \mathcal{C}_\delta^\prime$ with $\lim_{n\to\infty} Q_n = \infty$;
\item there exists a sequence $(Q_n, h_n) \in \mathcal{C}_\delta^\prime$ with $\lim_{n\to\infty} \max_{\overline{R}} \partial_p h_n = \infty$; or,
\item $\mathcal{C}_\delta^\prime$ contains another trivial point $(Q(\lambda), H(\lambda)) \in \mathcal{T}$, $-2B_{\textrm{min}} < \lambda \leq -2B_{\textrm{min}} + \epsilon_0 < \lambda^*$,  
\end{itemize}
where the second alternative follows from \textsc{Theorem \ref{uniformregularitytheorem}} and the third from the lemmas of \S5.  If $\mathcal{C}_\delta^\prime$ contains a point of $\partial\mathcal{O}_\delta$, then either
\begin{itemize}
\item there exists a $(Q, h) \in \mathcal{C}_\delta^\prime$ with $\partial_p h = \delta$ somewhere in $\overline{R}$, or
\item there exists a $(Q, h) \in \mathcal{C}_\delta^\prime$ with $h = \frac{Q-\delta}{2g \rho}$ somewhere on $T$.  
\end{itemize} 
Thus, \textsc{Theorem \ref{globalbifurcationtheorem}} tell us that at least one of the five alternative above must occur.
Consider the first alternative.  By the definition of $p_0$ in \eqref{defp0} we can estimate
\[ \inf_{y \in [-d, \eta(x)]} \sqrt{\rho(x,y)}\Big(c-u(x,y)\Big) \leq \frac{p_0}{\eta(x)+d} \leq \sup_{y \in [-d, \eta(x)]} \sqrt{\rho(x,y)}\Big(c-u(x,y)\Big) \]
and thus appealing to \eqref{defQ} we deduce
\be Q = \rho(0, \eta(0))\Big(c-u(0,\eta(0))\Big)^2 + 2g\rho(0, \eta(0))(\eta(0)+d) \leq \sup_{\overline{D_\eta}} \rho (c-u)^2 + \frac{2g \rho(0, \eta(0)) p_0}{\inf_{\overline{D_\eta}} \sqrt{\rho}(c-u)}. \label{boundQ} \ee
If for some $\delta > 0$ the first alternative holds, we conclude from \eqref{boundQ} that for the corresponding sequence $(u_n, v_n, \rho_n,\eta_n)$ of solutions to \eqref{euler2}-\eqref{boundcond}, either $\sup_{\overline{D_{\eta_n}}} \rho_n (c-u_n) \to \infty$, or $\inf_{\overline{D_{\eta_n}}} \rho_n(c-u_n) \to 0$.  

Similarly, if the second alternative holds, then simply by rearranging the change of variables in \eqref{uvhphq} we have
\[ \partial_p h_n = \frac{1}{\sqrt{\rho_n} (c-u_n)},\]
hence $\max_{\overline{R}} \partial_p h_n \to \infty$ iff $\inf_{\overline{D_{\eta_n}}} \rho_n(c-u_n) \to 0$.  

Next suppose that for some decreasing sequence $\delta_n \to 0$ the fourth alternative holds.  Again appealing to \eqref{uvhphq} we see immediately that the corresponding sequence $(u_n, v_n, \rho_n,\eta_n)$ must satisfy $\sup_{\overline{D_{\eta_n}}} \rho_n (c-u_n) \to \infty$ as $n \to \infty$.  

Finally, if we have a sequence $\delta_n \to 0$ for which the fifth alternative holds, then there exists a sequence $\{(Q_n, h_n)\}$ with $(Q_n, h_n) \in \mathcal{C}_{\delta_n}^\prime$ for each $n \geq 1$ and $\sup_T (2g\rho h_n - Q_n) \to 0$.  We have shown previously that $\partial_q h$ vanishes on $\partial \Omega_l$.  Thus if we evaluate the boundary condition at the crest $(0, 0)$, we find 
\[ \frac{1}{(\partial_p h_n(0,0))^2} = Q_n -2g\rho_n(0) h_n(0,0) \leq Q_n -2g\rho_n(0) h_n(q,0), \qquad \forall q \in [0,2\pi]. \]
Therefore $\partial_p h_n(0,0) \to \infty$ as $n \to \infty$. In particular, we have $\max_{\overline{R}} \partial_p h_n \to \infty$ so that the fourth alternative implies the second.

By definition the family of continua $\mathcal{C}_\delta^\prime$ indexed by $\delta > 0$ is increasing as $\delta$ decreases.  We may therefore define $\mathcal{C}^\prime := \bigcup_{\delta > 0} \mathcal{C}_\delta^\prime$ to be the maximal continuum.  Then by the considerations of the \S 2, and in particular \textsc{Lemma 2.1}, there exists a connected set $\mathcal{C}$ of solutions to \eqref{incompress}-\eqref{bedcond} corresponding to $\mathcal{C}^\prime$.  The arguments of the preceding paragraphs, moreover, imply that either $\sup_{\overline{D_{\eta_n}}} u_n \to c$, $\inf_{\overline{D_{\eta_n}}} u_n \to -\infty$ along some sequence in $\mathcal{C}$, or $\mathcal{C}$ contains more than one distinct laminar flow.  This completes the proof.  \qquad $\square$ \\

\remark In the homogeneous case it was proved in \cite{CS1} that if $\inf_{\overline{D}_{\eta_n}} u_n \to -\infty$ along some sequence in the continuum, then necessarily $\sup_{\overline{D}_{\eta_n}} u_n \to c$ along some sequence.  In other words, there are waves on $\mathcal{C}$ whose speed is arbitrarily close to the speed of the wave profile.  

We conjecture that the same holds true in the heterogeneous case.  Suppose, on the contrary, that for some sequence $\inf_{\overline{D_{\eta_n}}} u_n \to -\infty$, but $u_n$ remains uniformly bounded away from $c$ on $\mathcal{C}$.  Let $(Q_n, h_n)$ be the corresponding sequence of solutions to the height equation.  Then we can show that $\sup_{\overline{D_{\eta_n}}} \partial_p h_n,~\sup_{\overline{D_{\eta_n}}} h \to 0$.
In particular, this implies $\|\eta_n \|_\infty$, and $d(h_n) \to 0$ in the limit.  As the fluid domain is bounded between these to values, this entails that, moving along the continuum, we have regions of the fluid moving arbitrarily fast to left, while simultaneously the fluid is vertically pinching off.  This is a consequence of the fact that, if $u_n \to -\infty$, then in order to keep a constant pseudo-mass flux $p_0$, the domain must be vanishing. Giving the pathology of the scenario, we strongly believe it does not occur, though we have been unable to rule it out. \qquad$\square$

\section*{Acknowledgments}
The author wishes to thank \textsc{W. Strauss} for his indispensable guidance and support during the writing of this paper.


\begin{thebibliography}{99}
\bibitem{Ag}
{\sc S. Agmon}, {\em On the eigenfunctions and on the eigenvalues of general elliptic 
boundary value problems}, Comm. Pure Appl. Math., 15 (1962), pp.~119--147.

\bibitem{A1} 
{\sc C.J. Amick}, {\em Semilinear elliptic eigenvalue problems on an infinite strip with an application to stratified fluids}, Annali Scu. Norm. Sup. Pisa Cl. Sci.  4, 11 (1984), pp.~441--499.

\bibitem{AT}
{\sc C.J. Amick and R.E.L Turner}, {\em A global theory of internal solitary waves in two-fluid systems}, Trans. Amer. Math. Soc., 298 (1986), pp.~431--484.

\bibitem{BBT}
{\sc J. Bona, D.K. Bose and R.E.L. Turner}, {\em Finite amplitude steady waves in stratified fluid}, J. Math. Pures et Appl., 62 (1983), pp.~389--439.

\bibitem{CS1} {\sc A. Constantin and W. Strauss}, 
{\em Exact steady periodic water waves with vorticity}, Commun. Pure Appl. Math., 57 (2004), pp.~ 481--527
 
\bibitem{CS2} 
{\sc A. Constantin and W. Strauss}, {\em Rotational steady water waves near stagnation}, Phil. Trans. R. Soc., 365 (2007), pp.~2227-2239.
 
\bibitem{CR} 
{\sc M. Crandall and P. Rabinowitz}, {\em Bifurcations from simple eigenvalues}, J. Func. Anal., 8 (1971), pp.~321--340.
 
 \bibitem{Cu}
 {\sc B. Cushman-Roisin}, {\em Introduction to Geophysical Fluid Dynamics}, Prentice-Hall, Englewood Cliffs, NJ, 1994.
 
\bibitem{DJ1}
{\sc M.L. Dubreil-Jacotin}, {\em Sur la d\'etermination rigoureuse des ondes permanentes p\'eriodiques d'ampleur finie}, J. Math. Pures et Appl., 13 (1934), pp.~217--291. 
 
\bibitem{DJ2}  
 {\sc M.L. Dubreil-Jacotin}, {\em Sur les th\'eor\`emes d'existence relatifs aux ondes permanentes p\'eriodiques \'a deux dimensions dans les liquides h\'et\'erog\`enes}, J. Math. Pures et Appl. 9, 16 (1937), pp.~43--67.
 
\bibitem{FG}
{\sc I. Fonseca and W. Gangbo}, {\em Degree Theory in Analysis and Applications}, Clarendon Press, Oxford, 1995. 
 
\bibitem{Fr}
{\sc L.E. Fraenkel}, {\em An Introduction to Maximum Principles and Symmetry in Elliptic Problems}, Cambridge University Press, Cambridge, 2000. 

\bibitem{GT}
{\sc D. Gilbarg and N.S. Trudinger}, {\em Elliptic Partial Differential Equations of Second Order}, Springer-Herlag, Berlin, 2001.
 
 \bibitem{HS}
{\sc T. Healey and H. Simpson}, {\em Global continuation in nonlinear elasticity}, Arch. Rational Mech. Anal., 143 (1998), pp.~1--28.

\bibitem{K1}
{\sc K. Kirchg\"assner}, {\em Wave-solutions of reversible systems and applications}, J. Diff. Eq., 45 (1982), pp.~113--127.
 
\bibitem{LF} 
{\sc K. Lankers and G. Friesecke}, {\em Fast, large-amplitude solitary waves in the 2d Euler equations for stratified fluids}, Nonlinear Anal., Theory, Methods and Appl., 29 (1997), pp.~1061--1078.

\bibitem{LT}
{\sc G. Lieberman and N.S. Trudinger}, {\em Nonlinear oblique boundary value problems for nonlinear elliptic equations}, Trans. Am. Math. Soc., 295 (1986), pp.~509--546.

\bibitem{Lo}
{\sc R.R. Long}, {\em Some aspects of the flow of stratified fluids.  Part I: A theoretical investigation}, Tellus, 5 (1953), pp.~42--57.

\bibitem{PS}
{\sc A.S. Peter and J.J. Stoker}, {\em Solitary waves in liquids having nonconstant density}, Comm. Pure Appl. Math., 13 (1960), pp.~115--164.

\bibitem{R1}
{\sc P.H. Rabinowitz}, {\em Some global results for nonlinear eigenvalue problems}, J. Func. Anal., 7 (1971), pp.~487--513. 

\bibitem{St}
{\sc G. Stokes}, {\em On the theory of oscillatory waves}, Trans. Cambridge Phil. Soc., 8 (1847), pp.~441--455.

\bibitem{TK}
{\sc A.M. Ter-Krikorov}, {\em Th\'eorie exacte des ondes longues stationnaires dans un liquide h\'et\'erog\`ene}, J. M\'ecanique, 2 (1963), pp.~351--376.

\bibitem{T1}
{\sc R.E.L. Turner}, {\em Internal waves in fluids with rapidly varying density}, Ann. Scuola Norm. Sup. Pisa Cl. Sci. 8, 4 (1981), pp.~513--573. 

\bibitem{T2}
{\sc R.E.L. Turner}, {\em Traveling waves in natural systems} in Variational and Topological Methods in the Study of Nonlinear Phenomena, Progress in Nonlinear Differential Equations and Their Applications 49., \textsc{V. Benci} eds,  Birkh\"auser, Boston Basel Berlin, 2002, pp.~114--131.

\bibitem{Ya}
{\sc M. Yanowitch}, {\em Gravity waves in a heterogeneous incompressible fluid}, Comm. Pure Appl. Math., 15 (1962), pp.~45-61.

\bibitem{Y1}
{\sc C.S. Yih}, {\em Dynamics of Nonhomogeneous Fluids}, Macmillan Company, New York, 1965. 


\end{thebibliography}
\end{document}